\documentclass[11pt]{article}
\usepackage{amsthm}
\usepackage{amsfonts}
\usepackage{xr}
\usepackage{graphicx}
\usepackage{amsmath,bm}
\usepackage{epsfig,subfigure}
\usepackage{color}
\usepackage{amssymb}

\newcommand{\N}{\mathbb{N}} 
\newcommand{\Z}{\mathbb{Z}} 
\newcommand{\R}{\mathbb{R}} 
\newcommand{\CC}{\mathbb{C}} 
\newcommand{\T}{\mathbb{T}} 

\newcommand{\Dcal}{\mathcal{D}}

\newcommand{\Lcal}{\mathcal{L}}
\newcommand{\Mcal}{\mathcal{M}}

\def\epsilon{\varepsilon}
\def\hat{\widehat}
\def\tilde{\widetilde}

\newcommand{\me}{\mathrm{e}}

\def\XXint#1#2#3{{\setbox0=\hbox{$#1{#2#3}{\int}$ }
		\vcenter{\hbox{$#2#3$ }}\kern-.6\wd0}}

\newcommand{\SE}{\setcounter{equation}{0} \section}
\newcommand{\be}{\begin{equation}}
\newcommand{\ee}{\end{equation}}
\newcommand{\baa}{\begin{array}}
\newcommand{\eaa}{\end{array}}
\newcommand{\ba}{\begin{eqnarray}}
\newcommand{\ea}{\end{eqnarray}}

\newtheorem{theorem}{Theorem}[section]     

\newtheorem{corollary}{Corollary}[section]

\newtheorem{lemma}{Lemma}[section]
\newtheorem{definition}{Definition}[section]

\newtheorem{remark}{Remark}[section]

\begin{document}
\date{}
\title{\bf{Bistable pulsating waves with periodic advection: homogenization and sharp speed asymptotics}}
\author{Weiwei Ding\footnote{School of Mathematical Sciences, South China Normal University, Guangzhou 510631, China (dingweiwei@m.scnu.edu.cn).} 
\qquad Linfeng Xu\footnote{School of Mathematical Sciences and Wu Wen-Tsun Key Laboratory of Mathematics, University of Science and Technology of China, Hefei, Anhui, 230026, China (xlf\_hzfy@mail.ustc.edu.cn)} 
}

\maketitle

\begin{abstract}
We study bistable pulsating waves for reaction-diffusion equations with general periodic advection in arbitrary space dimension, allowing the diffusion matrix to be nonsymmetric. Assuming that the homogenized equation admits a traveling wave with nonzero speed in a given direction, we construct moving pulsating waves for all sufficiently small spatial periods $L$ and prove their convergence to the homogenized wave as $L\to0^+$. The existence range and convergence are uniform in the propagation direction when the homogenized speeds never vanish. We also prove uniqueness of the wave speed for arbitrary periods, profile uniqueness for moving waves, and stationary-wave uniqueness in the standing case under continuity of the competing profile. For spatially homogeneous reactions, we derive the expansion $c_L=c_0+Lc_1+O(L^2)$
and an explicit formula for $c_1$. Examples with constant diffusion and zero-mean periodic advection show that $c_1$ can have either sign, demonstrating that the heterogeneity in advection may accelerate or decelerate propagation relative to the homogenized limit and that the general $O(L)$ speed estimate is sharp.	
	
\end{abstract}	

\SE{Introduction and main results}

Fronts are fundamental objects in the study of propagation phenomena for reaction-diffusion equations, describing the invasion of one state by another in various applications such as biological invasions, disease spreading, and phase transitions. In spatially periodic media, classical traveling waves are generalized by pulsating waves, which capture the interaction between propagation and spatial heterogeneity. Moreover, in realistic environments, propagation is often affected by advection induced by environmental flows or biased dispersal (see \cite{cc,sk,x3} and references therein). This motivates the study of pulsating waves for the following reaction-diffusion equation with advection:
\begin{equation}\label{equation}
	u_t = \nabla \cdot (A_L(x)\nabla u) + B_L(x) \cdot \nabla u + f_L(x,u), \quad
	t \in \R, \; x \in \R^N,
\end{equation}
with $L > 0$ and $N \in \N$.
The diffusion, advection and reaction coefficients $A_L, B_L$ and $f_L$ are, respectively, given by
\begin{equation*}
	A_L(x) = A(x/L), \quad 
	B_L(x) = B(x/L), \quad 
	f_L(x,u) = f(x/L,u),
\end{equation*}
where the functions 
$$A(\cdot)=(A_{ij}(\cdot)): \R^N \to \mathcal{M}_N,\quad  B(\cdot): \R^N \to \R^N, \quad \hbox{and}\quad f(\cdot,u):\R^N\to\R $$
are $\Z^N$-periodic, that is, they are $1$-periodic in all variables $x_1,\cdots,x_N$. Here, $\mathcal{M}_N$ denotes the set of $N\times N$ real matrices. Throughout this paper, we assume that the matrix  $A(\cdot)$ and the vector $B(\cdot)$ are, respectively, of class $C^{2}(\R^N)$ and $C^{1}(\R^N)$, and there exists $\eta_0 > 0$ such that
\begin{equation}\label{uniform-elliptic}
	\sum_{1\leq i,j\leq N} A_{ij}(x) \xi_i  \xi_j \geq \eta_0 |\xi|^2\quad\hbox{for all }\, x\in\R^N,\;(\xi_i)_{1\leq i\leq N}\in \R^N.
\end{equation}

The function $f: \R^N \times [0,1] \to \R$, $(x,u) \mapsto f(x,u)$ is of class $C^{ 3}$.
We further assume that 
\begin{equation}\label{f-zero}
	f(x,0) = f(x,1) = 0 \quad\hbox{for all }\,x\in\R^N,
\end{equation}
and that $0$ and $1$ are uniformly (in $x$) stable zeros of $f(x,\cdot)$, in the sense that there exist $\gamma_0 > 0$
and $\delta_0 \in (0,1/2)$ such that
\begin{equation}\label{strong stability}
	\partial_u f(x,u) \leq -\gamma_0 \quad 
	\hbox{for all }\, (x,u) \in \R^N \times ([0,\delta_0] \cup  [1-\delta_0,1]).
\end{equation}
By the periodicity and regularity of $f$,
\eqref{strong stability} is equivalent to $\max\{\partial_u f(x,0), \partial_u f(x,1)\} < 0$ for all $x \in \R^N$. Hence, $0$ and $1$ are two linearly stable $L$-periodic steady states of \eqref{equation}. 
For mathematical convenience, 
we smoothly extend the function $f$ to $\R^N \times (\R \setminus [0,1])$ such that $f \in C^{ 3}(\R^N \times \R)$ with bounded derivatives and
\begin{equation}\label{extension}
\partial_u f(x,u) \leq - \gamma_0 \quad
\hbox{for all }\, (x,u) \in \R^N \times ((-\infty,\delta_0] \cup [1-\delta_0,+\infty)).
\end{equation}
Thus,
$f$ is $\Z^N$-periodic in $x$,
$\min_{x \in \R^N} f(x,u) > 0$ for  $u < 0$,
and $\max_{x \in \R^N} f(x,u) < 0$ for  $u > 1$.
This extension does not affect the behavior of pulsating waves connecting $0$ and $1$ defined below.

\begin{definition} \label{defi-pulsating}
For every $L>0$ and every $e \in \mathbb{S}^{N-1}$,	a pair $(\phi_{L,e},c_{L,e})$ with $\phi_{L,e}: \R \times \R^N \to (0,1)$ and $c_{L,e} \in \R$ is said to be a \textbf{pulsating wave} of \eqref{equation} with effective speed $c_{L,e}$ in the direction $e$ connecting $0$ and $1$ if the following two conditions are satisfied:
	\begin{enumerate}
		\item[{\rm (i)}] The function $U_{L,e}(t,x) := \phi_{L,e}(x \cdot e -c_{L,e}t,x/L)$ is an entire $($classical$)$ solution of the parabolic equation \eqref{equation};
		
		\item[{\rm (ii)}] The profile $\phi_{L,e}(\xi,y)$ is $\Z^N$-periodic in $y\in \R^N$, and satisfies
		\begin{equation}\label{limit-condition}
			\lim\limits_{\xi \to - \infty} \phi_{L,e}(\xi,y) = 1 
			\quad \hbox{ and } \quad 
			\lim\limits_{\xi \to +\infty} \phi_{L,e}(\xi,y) = 0
			\quad \hbox{uniformly for }\, y \in \R^N.
		\end{equation}
	\end{enumerate} 
	In the sequel, we shall say that a pulsating  wave $(\phi_{L,e},c_{L,e})$ of equation \eqref{equation} is:
	\begin{enumerate}
		\item[{\rm (a)}] \textbf{A standing pulsating wave} if $c_{L,e} = 0$;
		\item[{\rm (b)}] \textbf{A moving pulsating wave} if $c_ {L,e} \neq 0$.
	\end{enumerate}
\end{definition} 

The notion of pulsating waves was first introduced in \cite{skt,x1} as a natural extension of classical traveling waves $u(t,x)=\phi(x\cdot e-ct)$ in homogeneous media. However, unlike the classical notion, moving and standing pulsating waves need to be distinguished, as explained below.

If $U_{L,e}(t,x)= \phi_{L,e}(x \cdot e -c_{L,e}t,x/L)$ is a moving pulsating wave in the direction $e$, then the map $(t,x) \mapsto (x \cdot e - c_{L,e}t ,x/L)$ is a bijection from $\R\times\R^N$ to $\R\times \R^N$. Consequently, the profile $\phi_{L,e}$ is uniquely determined by $U_{L,e}$ via
\begin{equation}\label{phi-U}
\phi_{L,e}(\xi,y)=U_{L,e}\left(\frac{Ly\cdot e - \xi}{c_{L,e}},Ly\right)\,\,\hbox{ for all }\,\, (\xi,y)\in\R\times\T^N,  
\end{equation}
where $\T^N:=\R^N/\Z^N$ is the $N$-dimensional torus. In this setting, $\phi_{L,e}(\xi,y)$ satisfies the limit condition \eqref{limit-condition}, and is a classical solution of the semilinear degenerate elliptic equation 
\begin{equation}\label{front equation}
		\tilde{\nabla}_{L,e} \cdot (A(y)\tilde{\nabla}_{L,e} \phi_{L,e}) + B(y) \cdot \tilde{\nabla}_{L,e} \phi_{L,e} + c_{L,e} \partial_\xi \phi_{L,e} + f(y,\phi_{L,e}) = 0,\quad  (\xi,y) \in \R \times \T^N, 
\end{equation}
where 
\begin{equation*}
	\tilde{\nabla}_{L,e} = e \partial_\xi + \frac{1}{L}\nabla_y.
\end{equation*}
Reciprocally, the entire solution $U_{L,e}(t,x)$ of \eqref{equation} satisfies
\begin{equation*}
	U_{L,e}\left(t+ \frac{z\cdot e}{c_{L,e}},x+z\right)=U_{L,e}\left(t,x\right) \quad \hbox{for all }\, (t,x,z)\in\R\times\R^N\times L\Z^N,
\end{equation*}
together with the asymptotic conditions
$$\lim_{r\to -\infty} U_{L,e}(t,re+x)=1 \quad \hbox{and}\quad \lim_{r\to +\infty} U_{L,e}(t,re+x)=0, $$
where the convergences hold locally uniformly in $t\in\R$ and uniformly in $x\in e^{\perp}:= \{x\in\R^N: x\cdot e=0\}$.

The existence of moving pulsating waves for bistable equations in periodic media has been extensively studied. In the absence of advection, namely $B=0$, perturbative existence results were obtained in \cite{x1,fz,dg}, while homogenization results for equations with constant coefficients in rapidly oscillating perforated domains were established in \cite{he}. The small- and large-period regimes were studied for one-dimensional equations in \cite{dhz1}, and multidimensional existence results were obtained in \cite{du}. The homogenization and asymptotic behavior of pulsating waves and wave speeds as $L\to0^+$ were investigated in \cite{hps,ds}, and the large-$L$ behavior in one dimension was recently studied in \cite{dhl}. More general moving fronts were also studied in \cite{m,n2,nr,z} within the framework of generalized transition waves in the sense of \cite{bh12}. However, much less is known in the presence of advection. To our knowledge, Xin obtained existence results for periodic advection under the assumptions that the drift is divergence-free with zero mean and that the heterogeneity and drift are sufficiently small; see \cite{x2} and the review \cite{x3}. Beyond such special settings, the existence of bistable pulsating waves with general advection remains largely open. On the other hand, once a moving pulsating wave exists, its uniqueness up to translations in time is known in several of these settings; see, e.g., \cite{dhz1,gr25,m,x1,x2,x3}.

The standing case is more delicate. When the wave speed $c_{L,e}=0$, a standing pulsating wave $U_{L,e}(t,x)=\phi_{L,e}(x\cdot e,x/L)$ is a stationary solution of \eqref{equation} connecting $0$ and $1$ in the sense of \eqref{limit-condition}. In such a situation, the map $(t,x)\mapsto(x\cdot e-c_{L,e}t,x/L)$ is no longer invertible. Consequently, a standing wave does not determine its profile away from the set $\{(x\cdot e,x/L):x\in\R^N\}$.
Even a $C^2$ periodic representation of a standing wave need not satisfy \eqref{front equation} on the whole cylinder $\R\times\T^N$.
This is why standing pulsating waves need to be treated separately.

A broader class of stationary waves has been studied in connection with wave-blocking phenomena for bistable equations; see \cite{dhz2,dll,dr,k,n2,x2,z} and references therein. Following \cite{du}, we call such a wave a {\bf standing transition wave}, namely, a stationary solution
$U_{L,e}\in C^2(\R^N)$ of \eqref{equation} satisfying only
$$
\lim_{r\to-\infty}\inf_{x\in e^\perp}U_{L,e}(re+x)=1,
\qquad
\lim_{r\to+\infty}\sup_{x\in e^\perp}U_{L,e}(re+x)=0.$$
Every standing pulsating wave gives rise to a standing transition wave, but the converse is false in general; see \cite{x2} and \cite[Remark 1.7]{du}. Even in one spatial dimension, where any standing transition wave can formally be written in pulsating form, such a representation is not unique and the corresponding profile need not be a classical solution of \eqref{front equation}. Moreover, standing transition waves themselves need not be unique up to periodic shifts; see Remark \ref{rem-nonuniqueness} below. Therefore, uniqueness in the standing case is substantially more delicate than in the moving case.

We also mention that, in the absence of advection, pulsating waves exist under the abstract bistable assumption that no $L$-periodic stable state lies strictly between $0$ and $1$; see \cite{du,dgm,fz,gr20}. These results generally do not determine whether the wave is moving or standing. An exception is \cite{dgm}, where a positive-speed pulsating front is obtained under the additional assumption that some compactly supported initial data give rise to
solutions converging locally uniformly to $1$ as $t\to+\infty$.

These considerations lead to three problems addressed in this paper. First, we study the existence of moving pulsating waves in the presence of general advection and prove their existence and homogenization in the small-period regime. Second, we study the uniqueness of pulsating waves for arbitrary periods, including the standing case, under a natural regularity assumption on the wave profile. Finally, for the small-period moving waves, we derive a sharp convergence rate and an explicit first-order expansion of the wave speed as $L\to0^+$.


\subsection{Existence and homogenization of moving pulsating waves for small $L>0$}
In this section, we present the existence of moving pulsating waves for \eqref{equation} when $L>0$ is small under the assumption that the homogenized equation admits a classical traveling wave connecting $0$ and $1$ with nonzero speed. To present the homogenized equation, let us first introduce some notations. Denote by $A^{{\rm hom}} \in \Mcal_N$  the homogenized diffusion matrix of $A$, defined by 
\begin{equation}\label{a harmonic mean}
	A^{{\rm hom}} = \int_{\T^N}A(I_N + \nabla \chi)dy,
\end{equation}
where $I_N$ is the identity matrix, and $\chi : \T^N \to \R^N$ is the unique solution of the cell problem
\begin{equation}\label{Acorrector}
	\nabla \cdot (A(I_N + \nabla \chi)) = 0 \,\,\hbox{ in } \,\,\T^N, \quad \hbox{and}\quad \int_{\T^N} \chi dy=0.
\end{equation}
We also introduce the effective drift associated with $A$, denoted by $\langle B \rangle_{A}$, which is defined by  
\begin{equation}\label{homogenized limit w.r.t. A}
	\langle B \rangle_{A}=  \int_{\T^N}  (B-A^T \nabla \zeta) dy, 
\end{equation}
where $A^T$ is the transpose of $A$, and $\zeta:\T^N\to \R$ is the unique solution of 
\begin{equation}\label{Bcorrector}
	\nabla \cdot (A^T \nabla \zeta - B) = 0 \,\,
	\hbox{ in } \,\, \T^N, \quad\hbox{and}\quad \int_{\T^N} \zeta dy=0.
\end{equation} 
The functions $\chi$ and $\zeta$ are standard objects in periodic homogenization theory (see e.g., \cite{blp,jikov}). More precisely,  $\chi$ is the corrector accounting for the microscopic oscillations of the diffusion flux, while the corrector $\zeta$ arises from the coupling between the heterogeneous diffusion and drift term.
Since $A$ satisfies the uniform elliptic condition \eqref{uniform-elliptic}, the homogenized matrix $A^{{\rm hom}}$ inherits the same uniform elliptic property (see \cite{blp,jikov}). In the special one-dimensional case, where $A$ and $B$ are scalar-valued, one directly verifies that
$A^{{\rm hom}}$ and $\langle B \rangle_{A}$ reduce respectively to
$$A^{{\rm hom}}=\left(\int_0^1 A^{-1}(y)dy\right)^{-1}\quad \hbox{and}\quad \langle B \rangle_{A} =A^{{\rm hom}} \int_{0}^1 B(y)A^{-1}(y) dy.$$ 
As for the reaction term $f(y,u)$, we denote by $\bar{f}$ its arithmetic mean with respect to $y$, namely,
\begin{equation*}
	\bar{f}(u) = \int_{\T^N} f(y,u)dy \quad
	\hbox{for } \, u \in [0,1].
\end{equation*}
Then, $\bar{f}:[0,1]\to \R$ is of class $C^{ 3}$ and it follows from \eqref{f-zero} and \eqref{strong stability} that
\begin{equation}\label{property-barf}
\bar{f}(0)=\bar{f}(1)=0 \quad\hbox{and}\quad \bar{f}'(0)\leq -\gamma_0,\quad 	\bar{f}'(1)\leq -\gamma_0.
\end{equation} 
With these notations, the homogenized equation takes the form
\begin{equation}\label{homo-equation}
u_t=	\nabla \cdot (A^{{\rm hom}}\nabla u) + \langle B \rangle_{A} \cdot \nabla u + \bar{f}(u),\quad t\in\R,\;x\in\R^N.
\end{equation}

Assume that 
\begin{itemize}
	\item [\rm \bf (H1)] For each $e\in \mathbb{S}^{N-1}$, the homogenized equation \eqref{homo-equation} admits a traveling wave $u(t,x)=\phi_{0,e}(x\cdot e-c_{0,e}t)$ connecting $0$ and $1$ in the sense that $\phi_{0,e}(-\infty)=1$ and $\phi_{0,e}(+\infty)=0$, with speed $c_{0,e}\in\R$.
\end{itemize}

One readily checks that the profile $\phi_{0,e}(\cdot)$ satisfies the ODE
\begin{equation}\label{homo-wave}
	\begin{cases}
		(A^{{\rm hom}}e \cdot e)\phi_{0,e}'' + (\langle B \rangle_{A}\cdot e + c_{0,e})\phi_{0,e}' + \bar{f}(\phi_{0,e}) = 0 \quad \hbox{in }\, \R, \vspace{5pt}\\
		\phi_{0,e}(-\infty) = 1, \ \phi_{0,e}(+\infty) = 0.
	\end{cases} 
\end{equation}
Thanks to \eqref{property-barf}, namely, $0$ and $1$ are stable zeros of $\bar{f}$, it is well known  that for each given direction $e$, the speed $c_{0,e}$ is unique and the profile is necessarily decreasing and unique up to shifts (see e.g., \cite{aw,fm,hr24}). Moreover, $c_{0,e}$ has the sign of $\int_0^1\bar{f}(u)du$ provided that $\langle B \rangle_{A}\cdot e=0$. We also point out that both $c_{0,e}$ and $\phi_{0,e}$ depend on the direction $e$ explicitly in terms of an isotropic traveling wave (see the proof of Lemma \ref{continuity-homo-wave} below).

Note that condition (H1) is automatically fulfilled if the function $\bar{f}:[0,1]\to \R$ is of the bistable type, that is, in addition to \eqref{property-barf}, there exists a unique $\theta_0\in (0,1)$ such that $\bar{f}(\theta_0)=0$. It is also satisfied for some functions $\bar{f}$ having multiple oscillations in the interval $[0,1]$ (see \cite{fm,po,dingll}).  
In particular, our assumptions do not require $f(y,\cdot)$ to have a unique zero in $(0,1)$, and multiple intermediate zeros are allowed for each $y\in\T^N$.

\begin{theorem}\label{existence small period}
Let {\rm (H1)} hold. Then, the following statements hold.
\begin{itemize}
	\item [{\rm (i)}] For any $e \in \mathbb{S}^{N-1}$ such that $c_{0,e}\neq 0$,
	there exists $L_*(e)>0$ such that for every $0<L<L_*(e)$, equation \eqref{equation} admits a pulsating wave $U_{L,e}(t,x)=\phi_{L,e}(x\cdot e-c_{L,e}t,x/L)$ with speed $c_{L,e}\neq 0$ in the direction $e$.  Moreover, after a suitable shift of $\phi_{0,e}$, it holds that 
	\begin{equation}\label{convergence-phiL}
		\phi_{L,e} - \phi_{0,e} \to 0 \,\, \hbox{ in }\,\, H^{1}(\R \times \T^N)\quad\hbox{and} \quad c_{L,e}\to c_{0,e} \quad \hbox{as }\; L\to 0^+.
	\end{equation}
\item [{\rm (ii)}] If $c_{0,e}\neq 0$ for every direction $e \in \mathbb{S}^{N-1}$, then there exists $L_*>0$, independent of $e$, such that the conclusion in {\rm (i)} holds for every $e \in \mathbb{S}^{N-1}$ and every $0<L<L_*$. Moreover, the convergence \eqref{convergence-phiL} is uniform with respect to $e$. 
\end{itemize} 
\end{theorem}

Several remarks are in order. First, statement {\rm (i)} gives the existence of moving pulsating waves 
in a given direction $e$ for all sufficiently small $L>0$, provided that $c_{0,e}\neq 0$. 
The nonvanishing assumption on the limiting speed is essential for the present result. 
When $c_{0,e}=0$, one may instead expect that, for small $L>0$, equation \eqref{equation} admits either a moving or a standing pulsating wave. However, the standing-wave case is substantially more delicate, since multiple standing waves and heteroclinic connections between them may occur.  A more complete description of the propagation dynamics in this regime will be addressed in a forthcoming work.

Statement {\rm (ii)} further gives a common small-period regime for all propagation directions, together with convergence to the homogenized waves uniformly in $e$. To our knowledge, this uniform-in-direction homogenization result does not appear in the previous literature, even in the absence of advection. In the presence of advection, however, the assumption that $c_{0,e}\neq0$ for every $e\in\mathbb S^{N-1}$ imposes a restriction on the directional effect of the drift. Indeed, a sufficiently strong drift may reverse the sign of the wave speed in some directions while leaving it unchanged in others. 
For $N\geq 2$, since the sphere $\mathbb S^{N-1}$ is connected and the map $e\mapsto c_{0,e}$ is continuous (see Lemma \ref{continuity-homo-wave}), such a change of sign necessarily produces an intermediate direction in which the homogenized wave speed vanishes.  Hence, statement (ii) does not cover regimes in which the advection is strong enough to induce such direction-dependent sign changes.

The proof of Theorem \ref{existence small period} is based on the implicit function theorem in suitable Banach spaces, inspired by \cite{he,dhz1} for equations without advection. Heinze \cite{he} first developed this approach for
equations with constant diagonal diffusion and spatially homogeneous reaction terms in periodically perforated domains, and it was later adapted in \cite{dhz1} to one-dimensional equations with spatially dependent diffusion and reaction terms. 
Our setting is more delicate because of the advection term $B$, the possible lack of symmetry of $A$, and the uniformity with respect to $e$. To overcome these difficulties, we develop new energy-estimate techniques for the associated degenerate linear problems, and treat $L$ and $e$ simultaneously as parameters in the implicit-function argument.

Finally, \eqref{convergence-phiL} shows the convergence of moving pulsating waves of \eqref{equation} to the traveling wave of the homogenized equation \eqref{homo-equation} as $L\to0^+$. We refer to \cite{ckm1,ckm2,he1} for related homogenization results on pulsating waves in combustion-type equations, and to \cite{e,ehr} for the homogenization of minimal wave speeds in Fisher-KPP equations. In those settings, moving pulsating waves are known to exist for all $L>0$ (see \cite{bh02,x92}), so the homogenization analysis can build on this prior existence theory. Moreover, these works concern equations with symmetric diffusion and either no advection or divergence-free advection with zero mean. In contrast, we consider bistable equations with possibly nonsymmetric diffusion and general advection, for which even the existence of moving pulsating waves is nontrivial.


\subsection{Uniqueness of pulsating waves for general $L>0$}
In this section, for a fixed oscillation parameter $L>0$, not necessarily small, and a fixed propagation direction $e\in\mathbb{S}^{N-1}$, we study the uniqueness of pulsating waves of \eqref{equation} under the following existence and regularity assumption.

\begin{itemize}
	\item [\rm \bf (H2)] Equation \eqref{equation} admits a pulsating wave $U_{L,e}(t,x)=\phi_{L,e}(x\cdot e-c_{L,e}t, x/L)$ connecting $0$ and $1$ in the direction $e$, and $(\phi_{L,e},c_{L,e})$ is a classical solution of \eqref{front equation}.   
\end{itemize}

As mentioned earlier, if $(\phi_{L,e},c_{L,e})$ is a moving pulsating wave, then the parabolic regularity of \eqref{equation} automatically implies that it is a classical solution of \eqref{front equation}. 
For a standing wave, however, this conclusion does not follow from smoothness of the profile alone. In that case, (H2) additionally requires compatibility with the full profile equation on $\R\times\T^N$, and hence provides a family of stationary solutions under arbitrary phase shifts. The next theorem establishes uniqueness relative to this reference family.

\begin{theorem}\label{theo-unique}
	Let {\rm (H2)} hold.  Assume that $\tilde{U}_{L,e}(t,x)=\tilde{\phi}_{L,e}(x\cdot e-\tilde{c}_{L,e}t, x/L)$ is another pulsating wave of \eqref{equation} connecting $0$ and $1$ in the direction $e$. Then $\tilde{c}_{L,e} = c_{L,e}$. 
	Furthermore, the following statements hold.
	\begin{itemize}
		\item [{\rm (i)}] If $c_{L,e}\neq 0$, then $\tilde{\phi}_{L,e}(\xi,y)=\phi_{L,e}(\xi+\tau_0,y)$ in $\R\times \T^N$ for some $\tau_0\in\R$;
		\item [{\rm (ii)}] If $c_{L,e}= 0$ and $\tilde{\phi}_{L,e}\in C(\R\times\T^N)$, then $\tilde{\phi}_{L,e}(x\cdot e,x/L) = \phi_{L,e}(x\cdot e + \tau_0,x/L)$ in $\R^N$ for some $\tau_0\in\R$. 
	\end{itemize}
\end{theorem}

Theorem \ref{theo-unique} gives, in particular, the uniqueness of the wave speed for \eqref{equation}.
When a moving pulsating wave exists, statement {\rm (i)}
gives the classical uniqueness of the wave profile up to
translations in $\xi$; see e.g.,
\cite{dhz1,gr25,x1,x2,x3}.

The situation is different for standing pulsating waves. In this case, statement {\rm (ii)} gives uniqueness of
the wave as a stationary solution in the original spatial variable $x$, but does not imply uniqueness of its profile
as a function of $(\xi,y)\in\R\times\T^N$. In fact, full profile uniqueness cannot be expected even
if both profiles are classical solutions of \eqref{front equation}. To illustrate this point, consider the one-dimensional
homogeneous equation with
$N=1$, $L=1$, $e=1$, $A\equiv1$, $B\equiv0$,
and $ f(u)=u(1-u)(u-1/2)$.
The function $v(\xi)=(1+e^{\xi/\sqrt{2}})^{-1}$
satisfies $v''+f(v)=0$ in $\R$, and $v(-\infty)=1$, $v(+\infty)=0$.
Define $\phi(\xi,y)=v(\xi)$ and 
$\tilde{\phi}(\xi,y)= v(\xi+\sin(2\pi(\xi-y)))$.
Clearly, $\phi,\tilde\phi\in C^\infty(\R\times\T)$. Moreover, both $\phi$ and $\tilde\phi$ are classical
solutions of the standing profile equation \eqref{front equation} and satisfy the required limits at $\xi=\pm\infty$ uniformly in $y$. They generate exactly the same stationary wave, since $\tilde\phi(x,x)=v(x)=\phi(x,x)$ for $x\in\R$. 
On the other hand, there is no $\tau\in\R$ such that
$\tilde\phi(\xi,y)= \phi(\xi+\tau,y)$ in $\R\times\T^N$.  Indeed, $\phi$ is independent of $y$,
whereas $\tilde\phi$ depends nontrivially on $y$. This example shows that, in the standing case, even
classical solutions of the profile equation need not be unique up to translations in $\xi$.

As an immediate consequence of Theorem \ref{theo-unique}, we have the following observation.

\begin{corollary}\label{uniqunee-samllL} 
Let $(\phi_{L,e},c_{L,e})$ be the moving pulsating wave of \eqref{equation} provided by Theorem {\rm \ref{existence small period}} for small $L>0$. Then, $c_{L,e}$ is the unique wave speed, and  $\phi_{L,e}(\xi,y)$ is unique up to translations in $\xi$.  
\end{corollary}

\begin{remark}\label{rem-nonuniqueness}{\rm 
We point out that the uniqueness of standing pulsating waves stated in Theorem \ref{theo-unique} (ii) does not conflict with the non-uniqueness of standing transition waves observed in \cite{dhz2}. Indeed, \cite[Theorem 1.7]{dhz2} provides a one-dimensional example without advection in a slowly periodically oscillating medium, assuming that the function $\int_{0}^1f(x,u)du$ changes sign with respect to $x$. In this example, there exist multiple ordered standing transition waves together with heteroclinic orbits connecting them. Such a dynamical structure shows that uniqueness in the sense of
statement {\rm (ii)} cannot be expected in the broader class of standing transition waves. The key point is that these standing transition waves need not admit a classical periodic profile satisfying the standing profile equation
\eqref{front equation}.
}\end{remark}

Given the uniqueness of pulsating waves, it is natural to ask whether they are also globally asymptotically stable, as in the homogeneous case established in \cite{fm}. In the absence of advection, it was proved in \cite{dhz1} that any moving pulsating wave is globally asymptotically stable in one dimension. More recently, in higher dimensions, \cite{gr25} showed that, whenever a moving pulsating wave exists, its wave speed coincides with the asymptotic spreading speed of solutions to the corresponding Cauchy problem with front-like or compactly supported initial data. Nevertheless, the global asymptotic stability of pulsating waves in higher dimensions remains an open problem.


\subsection{Sharp convergence rate for wave speeds as $L\to 0^+$}
In this section, we fix a propagation direction $e\in\mathbb{S}^{N-1}$ such that $c_{0,e}\neq 0$ and study the convergence rate of the moving pulsating wave $(\phi_{L,e},c_{L,e})$ towards the homogeneous wave $(\phi_{0,e},c_{0,e})$ as $L\to 0^+$. This problem was previously studied in \cite{hps,ds} without advection, assuming the existence and uniqueness of moving pulsating waves with nonzero limiting speed.  Based on the existence and uniqueness results established above, we treat here the general case in the presence of advection. 
Since the direction $e$ remains fixed throughout this part, we suppress its dependence and simply write $(\phi_L,c_L)$ and $(\phi_0,c_0)$.

The following theorem gives an $O(L)$ convergence rate.

\begin{theorem}\label{theo-convergence}
Assume that {\rm (H1)} holds and $c_0\neq 0$. Let $(\phi_L,c_L)$ be the moving pulsating wave obtained in Theorem {\rm \ref{existence small period}} when $L>0$ is small. Then, there exists a constant $C>0 $, independent of $L$, such that after an appropriate shift of $\phi_0$, 
\begin{equation*}
		\|\phi_L - \phi_0\|_{H^{ 1}(\R \times \T^N)} \leq C L
		\quad \hbox{and} \quad 
		|c_L - c_0| \leq CL.
\end{equation*}		
\end{theorem}

Next, we derive an asymptotic expansion of the wave speed $c_L$ with respect to small $L > 0$, which shows that the $O(L)$ convergence rate is sharp, even in the case where $f$ is independent of $x$. To obtain 
an explicit expression for the first-order coefficient in the expansion, besides the correctors $\chi$ and $\zeta$ defined in \eqref{Acorrector} and \eqref{Bcorrector}, we introduce several additional correctors. Specifically,  denote by 
\begin{equation}\label{corrector-chi1}
\chi_1=\chi\cdot e,
\end{equation}	 
and consider the following two cell problems:
\begin{equation}\label{corrector-chi2}
\nabla \cdot (A \nabla \chi_2) = A^{\rm hom}e\cdot e - A(e + \nabla \chi_1) \cdot e -\nabla \cdot \left(Ae\chi_1\right)\,\, \hbox{ in } \,\, \T^N, \quad\hbox{and}\quad \int_{\T^N} \chi_2dy=0, 
\end{equation}
and 
\begin{equation}\label{corrector-eta1}
\nabla \cdot (A \nabla \eta_1) = \langle B \rangle_{A}\cdot e - B \cdot (e+\nabla \chi_1) \,\,	\hbox{ in } \,\, \T^N,
		\quad\hbox{and}\quad \int_{\T^N} \eta_1dy=0. 
\end{equation}
The existence and uniqueness of solutions to the above cell problems will be given in Lemma \ref{exisence-chi2-eta1}.  The correctors $\chi_2$ and $\eta_1$ are introduced to describe the higher-order effects of the diffusion and drift oscillations, which are essential for proving Theorem \ref{theo-convergence} and deriving the asymptotic expansion of $c_L$ in the following theorem.

\begin{theorem}\label{theo-sharp-speed}
Assume that {\rm (H1)} holds, $c_0\neq 0$, and $f=f(u)$ is independent of $x$. 
Then for all small $L$, the wave speed $c_L$ admits the expansion
\begin{equation*}
c_L = c_0 + Lc_1 +  O(L^2),
\end{equation*}
where the first-order coefficient $c_1$ is explicitly given by 
\begin{equation}\label{solvability condition-1}
\begin{aligned}
c_1=-d_1-\left(d_2\int_{\R} \phi_0''(\xi)w(\xi)d\xi+d_3  \int_{\R} \phi_0'''(\xi)w(\xi)d\xi\right)\left(\int_{\R}\phi_0'(\xi)w(\xi)d\xi\right)^{-1}. 
\end{aligned}
\end{equation}
Here, the constants $d_1, d_2, d_3$, and the function $w\in L^2(\R)$ are defined, respectively, by 
\begin{equation}\label{formula-d123}
	\left\{\begin{aligned}
		&d_1= \int_{\T^N} B \cdot \nabla \eta_1dy, \\
		&d_2 = \int_{\T^N} \left(A \nabla \eta_1 \cdot e + B \cdot e \chi_1 + B \cdot \nabla \chi_2\right)dy, \\
		&d_3 = \int_{\T^N} \left((A e \cdot e) \chi_1 + A \nabla \chi_2 \cdot e\right)dy,
	\end{aligned}\right.
\end{equation}
and
\begin{equation}\label{adjoint-kernel}
	w(\xi)=\phi_0'(\xi){\rm exp}\left(\frac{\langle B \rangle_{A}\cdot e + c_0}{A^{{\rm hom}}e \cdot e}\xi\right). 
\end{equation}

If, in addition, the diffusion matrix $A$ is symmetric, then the above expression for $c_1$ reduces to
\begin{equation}\label{solvability conditio2}
c_1=-\tilde d_1+\frac{\langle B\rangle_A\cdot e+c_0}{2A^{\rm hom}e\cdot e}\,\tilde d_2.
\end{equation}
where
\begin{equation}\label{simple-tildec12}
\tilde{d}_1=\int_{\T^N}  B \cdot (\nabla \chi_1 + e) \zeta dy,\quad\hbox{and}\quad \tilde{d}_2= \int_{\T^N}  B \cdot \nabla \left(\chi_2 - \frac{\chi_1^2}{2}\right)dy.
\end{equation}

\end{theorem}

Several remarks are in order. First, in the special case where $B=0$, one readily has $\eta_1=0$, and hence $d_1=d_2=0$. In this case, the convergence rate in Theorem \ref{theo-convergence} reduces to that obtained in \cite{hps}, while the wave-speed expansion in Theorem \ref{theo-sharp-speed} coincides with that derived in \cite{ds}. The proofs in \cite{hps,ds} rely on a min-max variational characterization of $c_L$, which in turn uses the global asymptotic stability of moving pulsating waves. As mentioned above, such global stability is not known in general.  Our approach, by contrast, does not use any variational characterization of the wave speed. It is based on the construction in the proof of Theorem \ref{existence small period}, combined with higher-order approximations of $\phi_L$ in terms of $L$ around the homogenized profile $\phi_0$. Moreover, the method applies to \eqref{equation} with a general drift term $B$.

Secondly, we identify several situations in which the first-order correction vanishes. On the one hand, 
if the diffusion matrix $A$ is symmetric and $\nabla\cdot B=0$, then $\zeta=0$, and consequently, $\tilde{d}_1=\tilde{d}_2=0$. It follows from \eqref{solvability conditio2} that $c_1=0$.
On the other hand, if $A$ is a constant matrix, then $\chi_1 = \chi_2 = 0$, and hence, 
$d_2=d_3=0$, 
$$c_1=-d_1=-\int_{\T^N} B \cdot \nabla \eta_1dy,$$ 
which in particular implies that $c_1$ is independent of the nonlinearity $f$.
If $A$ is further assumed to be symmetric, then
$c_1 = -\tilde{d}_1   = -\int_{\T^N}B \cdot e \zeta dy$. 
In this case, since $\zeta$ is the solution of \eqref{Bcorrector}, 
the vector field $B$ admits the decomposition 
\begin{equation}\label{decomposition}
B = A\nabla \zeta + T + \int_{\T^N}B(y)dy,
\end{equation}
where $T$ is a divergence-free periodic vector field with zero mean. Therefore, we have
\begin{equation}\label{simple-c1}
c_1 = -\int_{\T^N}\left((A\nabla \zeta \cdot e) \zeta + (T\cdot e)\zeta   \right)dy= -\int_{\T^N} (T \cdot e)\zeta dy. 
\end{equation}
In particular, if $N=1$, then $T=0$, which implies $c_1=0$. These observations yield the following corollary. 

\begin{corollary}\label{zero-c1}Under the assumptions of Theorem
	{\rm \ref{theo-sharp-speed}}, $c_1=0$ if either
\begin{itemize}
\item[{\rm (i)}]  $A$ is symmetric and $\nabla\cdot B=0$;
\item[{\rm (ii)}] $N=1$ and $A$ is a positive constant.
\end{itemize}
In either case, we have $c_L=c_0+O(L^2)$ as $L\to0^+$.
\end{corollary}

\begin{remark}{\rm 
The above corollary implies that, for equations with rapidly oscillating shear flows, 
the wave speed converges to the corresponding homogenized wave speed with an error of order $O(L^2)$. More precisely, consider the case where $B(x)=(b(\tilde{x}),0,\cdots,0)$ with $x=(x_1,\tilde{x})$, and assume that  $e=(1,0,\cdots,0)$. For each $L>0$, the corresponding equation 
$$
u_t = \Delta u + b(\tilde{x}/L)\partial_{x_1} u + f(u), \qquad t\in\mathbb{R},\; x\in\mathbb{R}^N,
$$
can be viewed as a periodic analogue of the problem in infinite cylinders with bounded cross sections and Neumann boundary conditions studied by Berestycki and Nirenberg \cite{bn92}. In particular, 
there exists a traveling wave of the form $U_L(t,x)=\phi_L(x_1-c_Lt,\tilde{x}/L)$, where the profile $\phi_L$ is $(1,\cdots,1)$-periodic in its second variable. Moreover, since $\nabla\cdot B=0$, Theorem \ref{theo-sharp-speed} and Corollary \ref{zero-c1} imply $c_L=c_0+O(L^2)$ provided that $c_0\neq 0$, where $c_0$ is the unique speed for the homogenized equation
$u_t = \Delta u + b_0\partial_{x_1}u+ f(u)$ with $b_0=\int_{\T^{N-1}} b(\tilde{x})d\tilde{x}$.
}\end{remark}

Finally, we point out that the first-order coefficient $c_1$ does not vanish in general, which implies that the $O(L)$  convergence rate is sharp. For example, consider
\begin{equation*}
N=2,\quad  e=(1,0),\quad  A=I_2,\quad  \quad\hbox{and}\quad  B(y)=(\cos(2\pi y_2), -\sin(2\pi y_2)),
\end{equation*} 
with $y=(y_1,y_2)$. 
It is straightforward to verify that $B(\cdot)$ is $\Z^2$-periodic, has zero mean over $\T^2$, and satisfies $\nabla \cdot B=-2\pi \cos(2\pi y_2)$. Moreover, in this situation,  the unique solution to \eqref{Bcorrector} is given by $\zeta(y)=\cos(2\pi y_2)/(2\pi)$, 
and the divergence-free periodic vector $T$ in the decomposition \eqref{decomposition} takes the form $T(y)= (\cos(2\pi y_2),0)$. Consequently, by \eqref{simple-c1}, we obtain 
$$c_1= -\int_{\T^N} (T \cdot e)\zeta dy=-\frac{1}{2\pi}  \int_0^1\int_0^1 \cos^2(2\pi y_2)dy_2dy_1= -\frac{1}{4\pi}\neq 0.  $$ 
Furthermore, if $B$ is replaced by $(-\cos(2\pi y_2),-\sin(2\pi y_2))$, then the corresponding $T$ becomes $(-\cos(2\pi y_2),0)$ while $\zeta$ remains unchanged, yielding that the corresponding value of $c_1$ changes sign. 
Therefore, even under the restrictions that $A$ and $f$ are homogeneous, and $B$ has zero mean, waves near the homogenization limit can be slower or faster compared to the homogenized wave.

The above example also illustrates that rapid oscillations in the drift term $B$ can generate a nontrivial first-order correction even when $A$ is a constant matrix. In contrast, Corollary \ref{zero-c1} shows that if $B$ is a constant vector and $A$ is symmetric, then the first-order correction vanishes. Hence, the heterogeneity of the drift term $B$ plays an essential role in the first-order correction of the wave speed.


\vskip 10pt 

\noindent{\bf{Outline of the paper.}} 
In Section 2, we study an auxiliary linear problem involving the linear part of equation \eqref{front equation} and its homogenized counterpart. The properties for these linear problems will serve as the foundation in Section 3 to prove the existence of pulsating waves for small $L>0$ (Theorem \ref{existence small period}), and in Section 4 to derive convergence rate estimates towards the homogenized wave as $L\to 0^+$ (Theorems \ref{theo-convergence} and \ref{theo-sharp-speed}). Finally, Section 5 is devoted to the proof of Theorem \ref{theo-unique} on the uniqueness of pulsating waves for general $L>0$ by a sliding argument.


\SE{Homogenization of an auxiliary linear problem}

In this section, we study the homogenization of an auxiliary degenerate linear elliptic problem and prove several properties needed for the subsequent analysis. Although related results are available in special cases with diagonal diffusion matrices and no advection \cite{he,dhz1}, the presence of advection and the possible lack of symmetry of the diffusion matrix require new energy-estimate techniques and a different argument for strong convergence.

To formulate the linear problems, we first introduce some notations. Let $L^2(\R \times \T^N)$ and $H^1(\R \times \T^N)$ be the Banach spaces defined by
\begin{equation*}
	\begin{aligned}
		L^2(\R \times \T^N) = 
		\Big\{&v \in L_{\mathrm{loc}}^2(\R \times \R^N) \  | \ 
		v \in L^2(\R \times (0,1)^N) \vspace{5pt}\\
		&\hbox{ and } v(\xi, y+z) = v(\xi,y) \hbox{ almost everywhere in } \R^{N+1} \hbox{ for any } z \in \Z^N\Big\},
	\end{aligned}
\end{equation*}
\begin{equation*}
	\begin{aligned}
		H^1(\R \times \T^N) = 
		\Big\{&v \in H_{\mathrm{loc}}^1(\R \times \R^N) \  | \ 
		v \in H^1(\R \times (0,1)^N) \vspace{5pt}\\
		&\hbox{ and } v(\xi, y+z) = v(\xi,y) \hbox{ almost everywhere in } \R^{N+1} \hbox{ for any } z \in \Z^N\Big\},
	\end{aligned}
\end{equation*}
endowed,respectively, with the norms $\|v\|_{L^2(\R \times \T^N)} = \|v\|_{L^2(\R \times (0,1)^N)}$
and
\begin{equation*}
	\|v\|_{H^1(\R \times \T^N)} =
	\|v\|_{H^1(\R \times (0,1)^N)} =
	\left(\|v\|^2_{L^2(\R \times \T^N)}+\|\partial_{\xi}v\|^2_{L^2(\R \times \T^N)}+\sum_{i=1}^{N}\|\partial_{y_i} v\|_{L^2(\R \times \T^N)}^2\right)^{1/2}.
\end{equation*}

Fix a large constant $\beta>0$ such that
\begin{equation}\label{beta-gamma-bound}
	\beta \geq \frac{\|B\|^2_{L^{\infty}(\T^N)}}{\eta_0},
\end{equation}
where $\eta_0>0$ is the constant from \eqref{uniform-elliptic}.
Recall that $A^{{\rm hom}}$ and $\langle B \rangle_{A}$ are, respectively, the constant matrix and vector defined in \eqref{a harmonic mean} and \eqref{homogenized limit w.r.t. A}.
Now, for any $c >0$ and $e\in \mathbb{S}^{N-1}$, we define 
\begin{equation*}
M_{c,L,e}(v) = \tilde{\nabla}_{L,e} \cdot (A \tilde{\nabla}_{L,e} v) + B \cdot \tilde{\nabla}_{L,e} v + c \partial_\xi v - \beta v,
\end{equation*}
where 
$$ v \in \Dcal_{L,e}: = 
\left\{v \in H^{1}(\R \times \T^N) \ | \ 
\tilde{\nabla}_{L,e} \cdot (A \tilde{\nabla}_{L,e} v) \in L^{2}(\R \times \T^N)\right\} \quad\hbox{and}\quad 
L \in \R^*, $$
and
\begin{equation}\label{defi-M0}
M_{c,0,e}(v) = (A^{{\rm hom}} e \cdot e) v'' + (\langle B \rangle_{A}\cdot e + c)v' - \beta v,
\end{equation}
where $v \in D(M_{c,0,e}): = H^{2}(\R)$. 
We point out that the operator $M_{c,L,e}$ is also defined for negative values of $L$, and that  $v\in\mathcal{D}_{L,e}$ means $v\in H^1(\R\times\T^N)$, and for all $\varphi\in H^1(\R\times\T^N)$, 
$$\left|\int_{\R\times\T^N}A \tilde{\nabla}_{L,e} v \cdot \tilde{\nabla}_{L,e} \varphi\right| \leq C\|\varphi\|_{L^2(\R\times\T^N)}$$ 
for some constant $C>0$, and
$$\int_{\R\times\T^N}A\tilde{\nabla}_{L,e} v\cdot\tilde{\nabla}_{L,e} \varphi=-\int_{\R\times\T^N}\tilde{\nabla}_{L,e}\cdot (A\tilde{\nabla}_{L,e} v)\varphi. $$

We first show the invertibility of $M_{c,L,e}$ and $M_{c,0,e}$, together with the corresponding energy estimates.

\begin{lemma}\label{solvability-M}
For every $c >0$ and $e\in \mathbb{S}^{N-1}$, the operators 
$M_{c,0,e}: H^{2}(\R) \to L^2(\R)$
and 
$M_{c,L,e}: \Dcal_{L,e} \to L^2(\R \times \T^N)$ for $L\neq 0$ 
are invertible. Furthermore, for every $L_0>0$ and $\gamma>0$, there exists a constant $C = C(L_0, \beta,\gamma, A, B)$ such that for all $|L| \leq  L_0$, $c\geq \gamma$, $g \in L^2(\R \times \T^N)$, $e\in \mathbb{S}^{N-1}$ and $\varphi \in L^2(\R)$,
\begin{equation}\label{energy-H1}
	\begin{cases}
		\|M_{c,L,e}^{-1}(g)\|_{H^1(\R \times \T^N)}
		\leq C \|g\|_{L^2(\R \times \T^N)},\vspace{5pt}\\
		\|\tilde{\nabla}_{L,e}(M_{c,L,e}^{-1}(g)) \|_{L^2(\R \times \T^N)} \leq C \|g\|_{L^2(\R \times \T^N)},
	\end{cases} \quad 
	\hbox{ if } \;L \neq 0,
\end{equation}	
	and 
	\begin{equation*}
		\|M_{c,0,e}^{-1}(\varphi)\|_{H^{1}(\R)} \leq C\|\varphi\|_{L^{2}(\R)}.
	\end{equation*}
\end{lemma}

\begin{proof}
We only consider the case $L\neq 0$, since the case $L=0$ can be treated similarly and the proof is simpler. For clarity, we divide the proof into two steps. 

{\bf Step 1}:  For any $c\geq \gamma$, $L\in [-L_0,0)\cup (0,L_0]$ and $e\in\mathbb{S}^{N-1}$,  assuming $M_{c,L,e}(v) = g$ with $v \in \Dcal_{L,e}$ and $g \in L^2(\R \times \T^N)$, we show the estimates in \eqref{energy-H1}. 

First, integrating $M_{c,L,e}(v) = g$ against $v$ gives
	\begin{equation*}
		\begin{aligned}
			-\int_{\R \times \T^N} g v
			&=\int_{\R \times \T^N} A(y)\tilde{\nabla}_{L,e} v \cdot \tilde{\nabla}_{L,e} v 
			-(B(y) \cdot\tilde{\nabla}_{L,e} v) v + \beta v^2 \\ 
			&\geq \int_{\R \times \T^N} \eta_0 |\tilde{\nabla}_{L,e} v|^2 
			- (B(y) \cdot \tilde{\nabla}_{L,e} v)v + \beta v^2.
		\end{aligned}
	\end{equation*}
	Since $ \beta\geq \|B\|_{L^{\infty}(\T^N)}^2/\eta_0$ by \eqref{beta-gamma-bound}, we have
	\begin{equation}\label{key-inequality}
		(B \cdot \tilde{\nabla}_{L,e} v) v \leq \frac{\eta_0}{2}|\tilde{\nabla}_{L,e} v|^2 + \frac{\beta}{2} v^2.
	\end{equation}
By the Cauchy–Schwarz and Young inequalities, it then follows that
	\begin{equation*}
		\int_{\R \times \T^N}\eta_0| \tilde{\nabla}_{L,e} v|^2 + \beta v^2 
		\leq -2 \int_{\R \times \T^N} gv 
		\leq \frac{2}{\beta} \int_{\R \times \T^N} g^2 + \frac{\beta}{2}\int_{\R \times \T^N} v^2. 
	\end{equation*}
	Thus, we obtain
	\begin{equation}\label{energy 1}
		\int_{\R \times \T^N}\eta_0 | \tilde{\nabla}_{L,e} v|^2 + \frac{\beta}{2} v^2 
		\leq \frac{2}{\beta}\int_{\R \times \T^N}g^2. 
	\end{equation}

	To estimate $\partial_\xi v$ without assuming that $A$ is symmetric, we take the unitary partial Fourier transform in $\xi$, understood in the $L^2$ sense, and write
	$$w_\tau(y):=\widehat v(\tau,y)=\frac{1}{\sqrt{2\pi}}\int_{\R}\me^{-i\tau\xi}v(\xi,y)d\xi,\quad\hbox{and}\quad 	h_\tau(y)=\widehat g(\tau,y).$$
Since $v\in H^1(\R\times\T^N)$, we have $w_\tau\in H^1(\T^N;\CC)$ for almost every $\tau$. Since the coefficients of $M_{c,L,e}$ are independent of $\xi$, and since $\widehat{\partial_\xi v}(\tau,y)=i\tau\widehat v(\tau,y)$, the transformed equation of $M_{c,L,e}(v) = g$ is
\begin{equation}\label{fourier-M-equation}
	D_\tau\cdot(AD_\tau w_\tau)+B\cdot D_\tau w_\tau
	+ic\tau w_\tau-\beta w_\tau=h_\tau
	\quad\hbox{in }\,H^{-1}(\T^N;\CC), 
\end{equation}
where 
\begin{equation}\label{diffenernatial-Dtau}
	D_\tau=\frac1L\nabla_y+i\tau e.
\end{equation}
Testing \eqref{fourier-M-equation} against the complex conjugate  $\overline{w_\tau}$, and integrating by parts over $\T^N$, yields
\begin{equation}\label{fourier-M-identity}
	-\int_{\T^N}AD_\tau w_\tau\cdot\overline{D_\tau w_\tau} +\int_{\T^N}(B\cdot D_\tau w_\tau)\overline{w_\tau}+ic\tau \|w_\tau\|_{L^2(\T^N)}^2-\beta \|w_\tau\|_{L^2(\T^N)}^2
	=\int_{\T^N}h_\tau\overline{w_\tau}
\end{equation}
for almost every $\tau\in\R$.

Next, for almost every $\tau\in\R$, define
$$ a_\tau = \int_{\T^N} AD_\tau w_\tau\cdot\overline{D_\tau w_\tau} dy.$$
We claim that
\begin{equation}\label{re-im-atau}
	\operatorname{Re}a_\tau
	\geq
	\eta_0\|D_\tau w_\tau\|_{L^2(\T^N)}^2,
	\qquad
	|\operatorname{Im}a_\tau|
	\leq
	\|A\|_{L^{\infty}(\T^N)}\|D_\tau w_\tau\|_{L^2(\T^N)}^2,
\end{equation}
where $\operatorname{Re}a_\tau$ and
$\operatorname{Im}a_\tau$ denote the real and imaginary parts of $a_\tau$, respectively.
Indeed, decompose $A$ into its symmetric part $A^s=(A+A^T)/2$ and anti-symmetric part $A^a=(A-A^T)/2$. 
Since $A$ is real-valued, for every
$z=p+iq\in\CC^N$ with $p,q\in\R^N$, the symmetry
of $A^s$ gives
$
(A^sz)\cdot\overline z=
(A^sp)\cdot p+(A^sq)\cdot q
\in\R$.
On the other hand, the anti-symmetry of $A^a$ yields
$
\overline{(A^az)\cdot\overline z}
=(A^a\overline z)\cdot z
=-(A^az)\cdot\overline z$, which implies that $(A^az)\cdot\overline z$ is purely imaginary.
Consequently, by the uniform ellipticity of $A$ (see \eqref{uniform-elliptic}), we have 
$$
\operatorname{Re}\bigl((Az)\cdot\overline z\bigr) =(A^sz)\cdot\overline z=(Ap)\cdot p+(Aq)\cdot q \geq \eta_0\bigl(|p|^2+|q|^2\bigr)
	=\eta_0|z|^2.
$$
Moreover, for each $y\in\T^N$, there holds
$$
	\left|
	\operatorname{Im}\bigl((Az)\cdot\overline z\bigr)
	\right|=\left|(A^az)\cdot\overline z\right|\leq |A^az|\,|z|\leq \|A^a\|_{L^{\infty}(\T^N)}|z|^2\leq  \|A\|_{L^{\infty}(\T^N)} |z|^2.
$$
Applying these pointwise inequalities with
$z=D_\tau w_\tau(y)$ and integrating over $\T^N$,
we obtain \eqref{re-im-atau}.

Now, taking the real part of \eqref{fourier-M-identity},  and using the Cauchy-Schwarz inequality and the inequality \eqref{beta-gamma-bound}, we obtain
\begin{equation*}
	\begin{aligned}
	&\eta_0\|D_\tau w_\tau\|_{L^2(\T^N)}^2 +\beta \| w_\tau\|_{L^2(\T^N)}^2\\
	\leq\, & \|B\|_{L^{\infty}(\T^N)}\|D_\tau w_\tau\|_{L^2(\T^N)}\| w_\tau\|_{L^2(\T^N)}  +\|h_\tau\|_{L^2(\T^N)}\| w_\tau\|_{L^2(\T^N)}\\
	\leq\, &\frac{\eta_0}{2}\|D_\tau w_\tau\|^2_{L^2(\T^N)}+\frac{\beta}{2}\| w_\tau\|^2_{L^2(\T^N)} +\|h_\tau\|_{L^2(\T^N)}\| w_\tau\|_{L^2(\T^N)}.
\end{aligned}
\end{equation*}
This implies 
\begin{equation*}
	\eta_0\|D_\tau w_\tau\|^2_{L^2(\T^N)}+\beta\| w_\tau\|^2_{L^2(\T^N)} \leq2 \|h_\tau\|_{L^2(\T^N)}\| w_\tau\|_{L^2(\T^N)}.
\end{equation*}
In particular, we have
\begin{equation}\label{fourier-real-energy}
	\|D_\tau w_\tau\|^2_{L^2(\T^N)} \leq \frac{2}{\eta_0} \|h_\tau\|_{L^2(\T^N)}\| w_\tau\|_{L^2(\T^N)}, 
\end{equation}
and 
\begin{equation}\label{fourier-real-energy-1}
	\begin{aligned}
\|B\|_{L^{\infty}(\T^N)}\|D_\tau w_\tau\|_{L^2(\T^N)}\| w_\tau\|_{L^2(\T^N)} \leq & \frac{\eta_0}{2}\|D_\tau w_\tau\|^2_{L^2(\T^N)}+\frac{\beta}{2} \| w_\tau\|^2_{L^2(\T^N)}\\
\leq & \|h_\tau\|_{L^2(\T^N)}\| w_\tau\|_{L^2(\T^N)}.
\end{aligned}
\end{equation}
Furthermore, taking the imaginary part of \eqref{fourier-M-identity} gives
\begin{equation*}
	\begin{aligned}
	c|\tau|\|w_\tau\|^2_{L^2(\T^N)} \leq  \|h_\tau\|_{L^2(\T^N)}\| w_\tau\|_{L^2(\T^N)}&+\|A\|_{L^{\infty}(\T^N)}\|D_\tau w_\tau\|_{L^2(\T^N)}^2\\
	&+ \|B\|_{L^{\infty}(\T^N)}\|D_\tau w_\tau\|_{L^2(\T^N)}\| w_\tau\|_{L^2(\T^N)}.   
\end{aligned}
\end{equation*}
It then follows from \eqref{fourier-real-energy} and \eqref{fourier-real-energy-1} that 
\begin{equation*}
c|\tau|\|w_\tau\|_{L^2(\T^N)} \leq 2\left(1+\frac{\|A\|_{L^{\infty}(\T^N)}}{\eta_0}\right) \|h_\tau\|_{L^2(\T^N)}.
\end{equation*}
Squaring this inequality, integrating over $\tau\in\R$, and using Plancherel's theorem, we obtain
\begin{equation}\label{energy 2}
	\|\partial_\xi v\|_{L^2(\R\times\T^N)}^2
	\leq \frac{4}{c^2}\left(1+\frac{\|A\|_{L^{\infty}(\T^N)}}{\eta_0}\right)^2 \|g\|_{L^2(\R\times\T^N)}^2.
\end{equation}

Finally, since $\nabla_yv=L(\tilde\nabla_{L,e}v-e\partial_\xi v)$, \eqref{energy 1} and \eqref{energy 2} imply
\begin{equation*}
	\begin{aligned}
		\|\nabla_yv\|_{L^2(\R\times\T^N)}^2
		&\leq2L^2\left(\|\tilde\nabla_{L,e}v\|_{L^2(\R\times\T^N)}^2
		+\|\partial_\xi v\|_{L^2(\R\times\T^N)}^2\right)\\
		&\leq2L^2\left(\frac{2}{\eta_0\beta}+\frac{4}{c^2}\left(1+\frac{\|A\|_{L^{\infty}(\T^N)}}{\eta_0}\right)^2\right)
		\|g\|_{L^2(\R\times\T^N)}^2.
	\end{aligned}
\end{equation*}
Combining these estimates, we obtain
\begin{equation*}
	\|v\|_{H^1(\R\times\T^N)}^2
	\leq  \left(\frac{4}{\beta^2}+\frac{4L^2}{\eta_0\beta}
	+\frac{4(1+2L^2)}{c^2}\left(1+\frac{\|A\|_{L^{\infty}(\T^N)}}{\eta_0}\right)^2\right) \|g\|_{L^2(\R\times\T^N)}^2.
\end{equation*}
Since $|L|\leq L_0$ and $c\geq \gamma>0$, this completes the proof of Step 1.

{\bf Step 2}: We show the invertibility of the operator $M_{c,L,e}$.

Fix $c>0$, $L\neq 0$ and $e\in\mathbb{S}^{N-1}$. We first prove the surjectivity of
$M_{c,L,e}$. Let $g\in L^2(\R\times\T^N)$ be arbitrary, and write
$h_\tau(y)=\widehat g(\tau,y)$ for its partial Fourier transform in $\xi$. 
For almost every $\tau\in\R$, define the sesquilinear form
$\mathcal{A}_{\tau}:H^1(\T^N;\CC)\times H^1(\T^N;\CC)\to\CC$ by
\begin{equation*}
\mathcal{A}_{\tau}(w,z)=\int_{\T^N}A D_\tau w\cdot\overline{D_\tau z}
		-\int_{\T^N}(B\cdot D_\tau w)\overline z   +(\beta-ic\tau)\int_{\T^N}w\overline z,
\end{equation*}
where $D_\tau$ is the differential operator defined in \eqref{diffenernatial-Dtau}.
Clearly, for each fixed $\tau\in\R$, the form $\mathcal{A}_{\tau}$ is continuous on
$H^1(\T^N;\CC)\times H^1(\T^N;\CC)$.
Moreover, by the uniform ellipticity of $A$ and \eqref{beta-gamma-bound}, we have, for every $w\in H^1(\T^N;\CC)$,
\begin{equation*}
\begin{aligned}
\operatorname{Re}\mathcal{A}_{\tau}(w,w)
	&\geq\eta_0\|D_\tau w\|_{L^2(\T^N)}^2-\|B\|_{L^\infty(\T^N)}
		\|D_\tau w\|_{L^2(\T^N)}\|w\|_{L^2(\T^N)}+\beta\|w\|_{L^2(\T^N)}^2\\
		&\geq\frac{\eta_0}{2}\|D_\tau w\|_{L^2(\T^N)}^2+\frac{\beta}{2}\|w\|_{L^2(\T^N)}^2.
	\end{aligned}
\end{equation*}
For fixed $L\neq0$ and $\tau\in\R$, the right-hand side controls the
$H^1(\T^N)$-norm of $w$, since $\nabla_yw=L(D_\tau w-i\tau ew)$.
Therefore, the complex Lax--Milgram theorem yields a unique
$w_\tau\in H^1(\T^N;\CC)$ such that
\begin{equation}\label{weak-fourier-M}
	\mathcal{A}_{\tau}(w_\tau,z)=-\int_{\T^N}h_\tau\overline z
	\quad\hbox{for every }\,z\in H^1(\T^N;\CC).
\end{equation}
Equivalently, $w_\tau$ is the unique weak solution of \eqref{fourier-M-equation},
and hence the arguments in Step 1 apply directly. In particular, for almost every
$\tau\in\R$, we have
\begin{equation*}
	\begin{cases}
	\|w_\tau\|_{L^2(\T^N)} \leq \displaystyle\frac{2}{\beta}\|h_\tau\|_{L^2(\T^N)},\vspace{5pt}\\
\|D_\tau w_\tau\|_{L^2(\T^N)}^2\leq \displaystyle\frac{4}{\eta_0\beta}\|h_\tau\|_{L^2(\T^N)}^2,\vspace{5pt}\\
c|\tau|\|w_\tau\|_{L^2(\T^N)} \leq \displaystyle 2\left(1+\frac{\|A\|_{L^\infty(\T^N)}}{\eta_0}\right)
\|h_\tau\|_{L^2(\T^N)}.
	\end{cases}
\end{equation*}
Integrating the above inequalities with respect to $\tau$ and using Plancherel's theorem, we obtain
a function $v\in L^2(\R\times\T^N)$ whose partial Fourier transform in $\xi$ is $w_\tau$,  such that
$\partial_\xi v\in L^2(\R\times\T^N)$ and $\tilde{\nabla}_{L,e}v\in L^2(\R\times\T^N)$.
Moreover, since
$\nabla_yw_\tau=L(D_\tau w_\tau-i\tau ew_\tau)$,
the estimates above also imply that $\nabla_yv\in L^2(\R\times\T^N)$.
Therefore, $v\in H^1(\R\times\T^N)$. Furthermore, taking the inverse Fourier transform in
\eqref{fourier-M-equation}, we obtain $M_{c,L,e}(v)=g$ in the sense of distributions in $\R\times\T^N$. 
Since $g$ is real-valued, the imaginary part of $v$ satisfies
$M_{c,L,e}(\operatorname{Im}v)=0$. By \eqref{energy 1}, $\operatorname{Im}v=0$, and hence $v$ is real-valued.
Consequently,
$$\tilde{\nabla}_{L,e}\cdot(A\tilde{\nabla}_{L,e}v)=
g-B\cdot\tilde{\nabla}_{L,e}v-c\partial_\xi v+\beta v
\in L^2(\R\times\T^N),$$
which shows that $v\in\Dcal_{L,e}$. Hence,
$M_{c,L,e}:\Dcal_{L,e}\to L^2(\R\times\T^N)$ is surjective.

On the other hand, if $M_{c,L,e}(v)=0$ for some $v\in\Dcal_{L,e}$, then \eqref{energy 1} immediately gives $v=0$.
Thus, $M_{c,L,e}$ is also injective, and hence invertible. The estimates in \eqref{energy-H1} follow directly from Step 1. This completes the proof of Lemma \ref{solvability-M}.
\end{proof}

We next study the homogenization of $M_{c,L,e}$ as $L\to 0$, with $M_{c,0,e}$ as the limiting operator. 
For any $g\in L^2(\R\times\T^N)$, define its mean $\bar{g}\in L^2(\R)$ by
\begin{equation}\label{mean definition}
	\bar{g}(\xi)=\int_{\T^N}g(\xi,y) dy
	\quad\hbox{for almost every }\,\xi\in\R.
\end{equation}
By Lemma \ref{solvability-M}, $M_{c,0,e}^{-1}(\bar{g})\in H^2(\R)$. 
In what follows, we identify this function with its $y$-independent extension to $\R\times\T^N$, which belongs to $H^2(\R\times\T^N)$.

Since $A$ is not assumed to be symmetric, we also introduce the adjoint corrector $\tilde{\chi}:\T^N\to\R^N$, defined as the unique solution of
\begin{equation}\label{ATcorrector}
\nabla\cdot\bigl(A^T(I_N+\nabla\tilde{\chi})\bigr)=0
\quad\hbox{in }\T^N, \qquad
	\int_{\T^N}\tilde{\chi} dy=0.
\end{equation}
Let $(A^T)^{\rm hom}$ denote the homogenized matrix associated with $A^T$, defined by \eqref{a harmonic mean} with $(A,\chi)$ replaced by $(A^T,\tilde{\chi})$. 
The standard duality identity for homogenized matrices (see \cite{blp,jikov}) yields
\begin{equation}\label{adjoint-homo-A}
	(A^T)^{\rm hom}=(A^{\rm hom})^T.
\end{equation}

\begin{lemma}\label{Homogenization}
Fix $c>0$ and $e\in \mathbb{S}^{N-1}$. Then, for every $g \in L^2(\R \times \T^N)$, we have
\begin{equation}\label{homogenization process}
M_{c_n,L_n,e_n}^{-1}(g_n) \to M_{c,0,e}^{-1}(\bar{g}) \;
\hbox{ in }\; H^{1}(\R \times \T^N) \;\hbox{ as }\; n\to+\infty,
\end{equation}
and 
\begin{equation}\label{limit-Ln0}
	M_{c_n,0,e_n}^{-1}(\varphi_n) \to M_{c,0,e}^{-1}(\bar{g}) \;
	\hbox{ in }\; H^{1}(\R \times \T^N) \;\hbox{ as }\; n\to+\infty,
\end{equation}	
 for all sequences $\{g_n\}_{n \in \N} \subset L^2(\R \times \T^N)$, $\{\varphi_n\}_{n \in \N} \subset L^2(\R)$,
	$\{c_n\}_{n \in \N} \subset (0, + \infty)$,
	$\{L_n\}_{n \in \N} \subset \R^*$, and $\{e_n\}_{n\in\N} \subset \mathbb{S}^{N-1}$
	such that $\|g_n - g\|_{L^2(\R \times \T^N)} \to 0$, $\|\varphi_n - \bar{g}\|_{L^2(\R)} \to 0$,
	$c_n \to c$, $L_n \to 0$ and $e_n\to e$ as $n \to +\infty$, where $\bar{g}$ is the mean of $g$ over $y\in\T^N$, as defined in \eqref{mean definition}. 
\end{lemma}

\begin{proof}
We first prove \eqref{homogenization process} in three steps and then verify \eqref{limit-Ln0}. 
For each $n\in\N$, let 
$$v_n = M_{c_n,L_n,e_n}^{-1}(g_n) \in H^{1}(\R \times \T^N).$$ By Lemma \ref{solvability-M},  
the sequence $\{v_n\}_{n\in\N}$ is bounded in $H^1(\R \times \T^N)$, and $ \{\tilde{\nabla}_{L_n,e_n}v_n\}_{n\in\N}$ is bounded in $L^2(\R \times \T^N)$. Hence, after passing to a subsequence, we have
$ v_n\rightharpoonup v_0$ weakly in $H^1(\R \times \T^N)$
for some $v_0\in H^1(\R \times \T^N)$. Moreover, since
$ \nabla_yv_n=L_n(\tilde{\nabla}_{L_n,e_n}v_n-e_n\partial_\xi v_n)$,
we obtain
\begin{equation}\label{gradient-vn}
	\|\nabla_y v_n\|_{L^2(\R\times \T^N )}  \leq C|L_n|, 
\end{equation}
for some constant $C>0$ independent of $n$. Consequently, we have $\nabla_y v_0 = 0$, and hence $v_0$ can be viewed as an $H^{1}(\R)$ function, and we can set $v_0' = \partial_\xi v_0$.

{\bf Step 1:} We show that $M_{c,0,e}(v_0) = \bar{g}$. 
	
For any $n\in\N$ and $\phi \in H^2(\R)$,
define 
$$\psi_n(\xi,y) = \phi(\xi) + L_n \phi'(\xi) \tilde{\chi}(y) \cdot e_n+ L_n \phi(\xi) \zeta(y) \quad \hbox{for }\; \xi\in\R,\,y\in\T^N, $$
where $\tilde{\chi} \in C^{ 2}(\T^N;\R^N)$ and $\zeta \in C^{ 1}(\T^N;\R)$ are, respectively, the unique solutions of \eqref{ATcorrector} and \eqref{Bcorrector}. Then $\psi_n\in H^1(\R\times\T^N)$. Taking $\psi_n$ as a test function in $M_{c_n,L_n,e_n}(v_n) = g_n$ gives
	\begin{equation}\label{right-left}
		\begin{aligned}
			&\quad - \int_{\R \times \T^N} g_n \left(\phi + L_n \phi' \tilde{\chi} \cdot e_n + L_n  \phi \zeta\right) \\
			&= \int_{\R \times \T^N} A\tilde{\nabla}_{L_n,e_n} v_n \cdot 
			\left(\phi'(I_N+ \nabla \tilde{\chi}) e_n 
			+ \phi \nabla\zeta  
			+ L_n \phi'' (\tilde{\chi} \cdot e_n) e_n
			+ L_n \phi' \zeta e_n\right) \\
			&\quad + \int_{\R \times \T^N} 
			\left(\beta v_n 
			- c_n \partial_\xi v_n 
			- B \cdot \tilde{\nabla}_{L_n,e_n} v_n\right) 
			\left(\phi + L_n \phi' \tilde{\chi} \cdot e_n + L_n \phi \zeta\right),
		\end{aligned}
	\end{equation}
	where $A$, $B$, $\tilde{\chi}$, $\nabla\tilde{\chi}$, $\zeta$ and $\nabla\zeta$ are evaluated at $y$,
	while $\phi$, $\phi'$ and $\phi''$ are evaluated at $\xi$.
The right-hand side of \eqref{right-left} can be rewritten as $J_n^1+J_n^2$, 
where
\begin{equation*}
\begin{aligned}
	J_n^1=\int_{\R \times \T^N} & A^T(I_N+ \nabla\tilde{\chi})e_n \cdot (\tilde{\nabla}_{L_n,e_n} v_n) \phi' \\
	& + \int_{\R \times \T^N} (A^T \nabla \zeta - B) \cdot (\tilde{\nabla}_{L_n,e_n} v_n)  \phi 
	+ \beta v_n \phi- c_n \partial_\xi v_n \phi,
\end{aligned}
\end{equation*}
and 
\begin{equation*}
		\begin{aligned}
	 J_n^2=&L_n \int_{\R \times \T^N} A\tilde{\nabla}_{L_n,e_n}v_n \cdot
			\left(\phi'' (\tilde{\chi} \cdot e_n) + \phi' \zeta\right)e_n \\
			&\quad +L_n \int_{\R \times \T^N} \left(\beta v_n - c_n \partial_\xi v_n - B\cdot \tilde{\nabla}_{L_n,e_n} v_n\right)
			\left(\phi' \tilde{\chi} \cdot e_n + \phi \zeta\right). 
		\end{aligned}
	\end{equation*} 

Next, we identify the limits of $J_n^1$ and $J_n^2$ as $n\to+\infty$.  Since $\tilde{\chi}(\cdot)$ and $\zeta(\cdot)$ satisfy \eqref{ATcorrector} and \eqref{Bcorrector}, we have
\begin{equation*}
	\int_{\R \times \T^N} A^T(I_N+ \nabla \tilde{\chi}) e_n \cdot \nabla_y v_n \phi'
	= - \int_{\R \times \T^N} \nabla \cdot (A^T(I_N+ \nabla \tilde{\chi})) \cdot e_n v_n \phi' = 0,
\end{equation*}
and
\begin{equation*}
	\int_{\R \times \T^N} (A^T \nabla \zeta - B) \cdot \nabla_y v_n \phi
	= - \int_{\R \times \T^N} \nabla \cdot (A^T \nabla \zeta - B) v_n \phi 
	= 0.
\end{equation*}
Consequently, $J_n^1$ simplifies to 
\begin{equation*}
J_n^1=\int_{\R \times \T^N} A^T(I_N+ \nabla\tilde{\chi})e_n \cdot e_n \partial_{\xi}v_n\phi' 
+ (A^T \nabla \zeta - B) \cdot e_n\partial_{\xi}v_n  \phi 
+ \beta v_n \phi
- c_n \partial_\xi v_n \phi.
\end{equation*} 
Since $c_n\to c$, $e_n\to e$ and $v_n \rightharpoonup v_0$ weakly in $H^{1}(\R\times \T^N)$ as $n\to+\infty$, and since the sequence $\{\partial_{\xi}v_n\}_{n\in\N}$ is bounded in $L^2(\R\times \T^N)$ by Lemma \ref{solvability-M}, we deduce that, as $n\to+\infty$, 
\begin{equation*}
	J_n^1\to \int_{\R \times \T^N} A^T(I_N+ \nabla \tilde{\chi})e \cdot e v_0' \phi' 
	+ (A^T \nabla \zeta - B) \cdot e v_0' \phi+ \beta v_0 \phi - c v_0'\phi.
\end{equation*} 
Furthermore, by the definitions of $(A^T)^{\rm hom}$ and $\langle B \rangle_{A}$, we have
\begin{equation*}
	J_n^1\to  \int_{\R} ((A^T)^{\rm hom}  e \cdot e) v_0' \phi'-\langle B \rangle_{A}\cdot e v_0' \phi
+ \beta v_0 \phi- c v_0' \phi. 
\end{equation*}
Owing to the identity \eqref{adjoint-homo-A} and the fact that $(A^{\rm hom})^T e\cdot e=A^{\rm hom} e\cdot e$, we obtain 
$$J_n^1\to   \int_{\R} (A^{\rm hom}  e \cdot e) v_0' \phi'-\langle B \rangle_{A}\cdot e v_0' \phi
+ \beta v_0 \phi- c v_0'\phi.$$

Regarding the limit of $J_n^2$, it follows again from Lemma \ref{solvability-M} that both $\{\|v_n\|_{H^1(\R\times \T^N)}\}_{n\in\N}$ and $\{\|\tilde{\nabla}_{L_n,e_n}v_n\|_{L^2(\R\times \T^N)}\}_{n\in\N}$ are bounded. Together with the boundedness of $\{c_n\}_{n\in\N}$, this implies that 
$J_n^2\to 0$ as $n\to+\infty$.

To identify the limit of the left-hand side of \eqref{right-left},  we note that $g_n \to g$ and  $\phi + L_n \phi' \tilde{\chi} \cdot e_n + L_n \phi \zeta \to \phi$ in $L^{2}(\R\times \T^N)$ as $n\to+\infty$. Therefore, passing to the limit as $n\to +\infty$ gives
\begin{equation*}
	-\int_{\R \times \T^N} g_n (\phi + L_n \phi' \tilde{\chi} \cdot e_n + L_n \phi \zeta) 
	\to - \int_{\R \times \T^N} g \phi 
	= -\int_{\R} \bar{g} \phi.
\end{equation*}
Combining these estimates, we obtain
	\begin{equation*}
		\int_{\R} (A^{{\rm hom}} e \cdot e) v_0' \phi'
		- \langle B \rangle_{A}\cdot e v_0' \phi
		+ \beta v_0 \phi 
		- c v_0' \phi
		= -\int_{\R} \bar{g} \phi 
	\end{equation*}
	for all $\phi \in H^2(\R)$,
	and then for all $\phi \in H^1(\R)$ by density. Since $A^{\rm hom}e\cdot e>0$, this identity yields
	\[
	v_0''
	=
	\frac{\bar g-(\langle B\rangle_A\cdot e+c)v_0'
		+\beta v_0}{A^{\rm hom}e\cdot e}
	\in L^2(\R).
	\]
	Thus, $v_0\in H^2(\R)$ and
	$v_0=M_{c,0,e}^{-1}(\bar g)$.
	This implies that $v_0 \in H^2(\R)$ and $M_{c,0,e}(v_0) = \bar{g}$.

{\bf Step 2:} We show that 
\begin{equation*}
	v_n \to v_0 \quad\hbox{and} \quad\tilde{\nabla}_{L_n,e_n}v_n \to (I_N+\nabla \chi)e v_0' \quad \hbox{in } \; L^2(\R \times \T^N)\;
 \hbox{ as }	\;n \to +\infty,
\end{equation*}	
where $\chi \in C^{ 2}(\T^N;\R^N)$ is the corrector of $A$ defined in \eqref{Acorrector}.

To do so, for each $n\in\N$, define
$$w_n(\xi,y) = v_n(\xi,y) - v_0(\xi) - L_nv_0'(\xi)\chi(y) \cdot e_n  \quad\hbox{for }\; (\xi,y)\in \R\times\T^N.$$ 
Since $v_0 \in H^2(\R)$, and  the sequence $\{\|v_n\|_{H^1(\R\times \T^N)}\}_{n\in\N}$ is bounded, it is clear that $\{w_n\}_{n\in\N}$ is bounded in $H^1(\R\times\T^N)$. 
Recall that $M_{c_n,L_n,e_n}(v_n)=g_n$. A direct computation shows that $w_n$ satisfies
	\begin{equation*}
		\begin{aligned}
			M_{c_n,L_n,e_n}(w_n)
			&= -L_n A e_n \cdot e_n (\chi \cdot e_n) v_0^{(3)}
			-A(I_N+ \nabla \chi) e_n \cdot e_n v_0'' \vspace{5pt}\\
			&\quad- \nabla \cdot (A (\chi \cdot e_n)e_n) v_0'' - B\cdot(I_N+ \nabla \chi) e_n v_0'
		 - L_n (B \cdot e_n) (\chi \cdot e_n) v_0'' \vspace{5pt}\\
		&\quad	- L_n c_n  (\chi \cdot e_n) v_0''
			+ L_n \beta (\chi \cdot e_n) v_0'
			+ g_n - c_n v_0'+\beta v_0,
		\end{aligned}
	\end{equation*}
where  $A e_n \cdot e_n (\chi \cdot e_n) v_0^{(3)}$ is understood as an element of $H^{-1}(\R \times \T^N)$.
We write this identity as follows:
\begin{equation}\label{convergence}
	M_{c_n,L_n,e_n}w_n=H_n^1+H_n^2+H_n^3,
\end{equation}
where
\begin{equation*}
	\left\{\begin{aligned}
		H_n^1&=-A(I_N+ \nabla \chi)e_n \cdot e_n v_0''- \nabla \cdot (A (\chi \cdot e_n)e_n) v_0''- B\cdot (I_N+\nabla \chi) e_nv_0' +g_n-c_nv_0'+\beta v_0,  \vspace{5pt}\\
		H_n^2&= -L_n\left((B \cdot e_n) (\chi \cdot e_n) v_0''+  c_n  (\chi \cdot e_n) v_0''- \beta (\chi \cdot e_n) v_0'\right),\vspace{5pt}\\
		H_n^3 &= -L_nA e_n \cdot e_n (\chi \cdot e_n) v_0^{(3)}. 
	\end{aligned}\right.
\end{equation*}

Now, taking $w_n$ as a test function in \eqref{convergence} gives
$$
\left(M_{c_n,L_n,e_n}(w_n),w_n\right)_{(H^{-1};H^{1})} 
=\int_{\R \times \T^N} \left(H_n^1+H_n^2\right)w_n+\left(H_n^3,w_n\right)_{(H^{-1};H^{1})},  
$$
where 
$(\cdot,\cdot)_{(H^{-1};H^1)}$ denotes the duality pairing
between $H^{-1}(\R\times\T^N)$ and $H^1(\R\times\T^N)$.

On the one hand, 
since $c_n\to c$, $e_n\to e$, $L_n\to 0$, and $g_n \to g$ in $L^{2}(\R\times \T^N)$, we have 
$$ H_n^1 \to -A(I_N+ \nabla \chi)e \cdot e v_0''- \nabla \cdot (A (\chi \cdot e)e) v_0''- B\cdot (I_N+\nabla \chi) ev_0' +g-cv_0'+\beta v_0,$$
and $ H_n^2\to 0$ as $n\to+\infty$ in $L^2(\R\times \T^N)$. 
Since $H_n^1+H_n^2$ converges strongly in
$L^2(\R\times\T^N)$ and $w_n\rightharpoonup0$
weakly in this space, we obtain
$$ \int_{\R \times \T^N} \left(H_n^1+H_n^2\right)w_n \to 0 \quad\hbox{as }\; n\to+\infty.$$
Moreover, integration by parts in $\xi$ and the Cauchy-Schwarz inequality give
$$
\begin{aligned}
	\left|\left(H_n^3,w_n\right)_{(H^{-1};H^1)}\right| &=|L_n|\left|\int_{\R\times\T^N}
	(Ae_n\cdot e_n)(\chi\cdot e_n)v_0''\partial_\xi w_n\right|\\
	&\leq C|L_n|\|v_0''\|_{L^2(\R)}
	\|\partial_\xi w_n\|_{L^2(\R\times\T^N)} \mathop{\longrightarrow}_{n\to+\infty} 0.
\end{aligned}
$$

On the other hand, by using an estimate similar to \eqref{key-inequality}, we deduce that for each $n\in\N$, 
\begin{equation*}
	\begin{aligned}
		-\left(M_{c_n,L_n,e_n}(w_n),w_n\right)_{(H^{-1};H^{1})}=& \int_{\R \times \T^N} A \tilde{\nabla}_{L_n,e_n}w_n \cdot \tilde{\nabla}_{L_n,e_n}w_n 
		- B\cdot (\tilde{\nabla}_{L_n,e_n}w_n)w_n + \beta w_n^2 \\
		\geq& \frac{1}{2}\int_{\R \times \T^N} \eta_0|\tilde{\nabla}_{L_n,e_n}w_n|^2 + \beta w_n^2.
	\end{aligned}
\end{equation*}
Thus, combining the above, we obtain 
 $v_n \to v_0$ and $\tilde{\nabla}_{L_n,e_n}v_n \to (I_N+\nabla \chi)e v_0'$ in $L^2(\R \times \T^N)$ as $n \to +\infty$.
Finally, since every subsequence admits a further subsequence for which the above convergences hold, and the limit is uniquely determined by $v_0=M^{-1}_{c,0,e}(\bar{g})$, these convergences hold for the entire sequence. This completes the proof of Step 2.

{\bf Step 3:} We show that $v_n \to v_0$ in $H^1(\R \times \T^N)$ as $n \to +\infty$. 

Let $\rho\in C_c^\infty(\mathbb R)$ be nonnegative with $\int_{\mathbb R}\rho=1$. For any $\varepsilon>0$, define
$\rho_ \varepsilon(\xi)= \varepsilon^{-1}\rho(\xi/ \varepsilon)$, and
$$ S_\varepsilon u(\xi,y):=(\rho_ \varepsilon*_\xi u)(\xi,y)=\int_{\R} \rho_{\varepsilon}(z)u(\xi-z,y)dz\quad\hbox{for } \,u\in L^2(\R\times\T^N). $$
By Young's convolution inequality,
$S_\varepsilon$ maps $H^1(\R \times \T^N)$ into itself.
Since the coefficients of $M_{c_n,L_n,e_n}$ are independent of $\xi$, the convolution in $\xi$ commutes with $M_{c_n,L_n,e_n}$.
Moreover, since $v_n\in\mathcal D_{L_n,e_n}$, we have
$S_\varepsilon v_n\in H^1(\R \times \T^N)$ and
$$\tilde{\nabla}_{L_n,e_n}\cdot(A\tilde{\nabla}_{L_n,e_n}S_\varepsilon v_n)=S_\varepsilon\bigl[\tilde{\nabla}_{L_n,e_n}\cdot(A\tilde{\nabla}_{L_n,e_n}v_n)\bigr] \in L^2(\R\times\T^N).$$
Thus, $S_\varepsilon v_n\in\mathcal D_{L_n,e_n}$, and
$$ M_{c_n,L_n,e_n}(v_n-S_ \varepsilon v_n)=g_n-S_ \varepsilon g_n. $$
Since $c_n\to c>0$ and $L_n\to0$ as $n\to+\infty$, the uniform estimate in
Lemma \ref{solvability-M} gives, for all sufficiently large $n$,
\begin{equation}\label{S-convlution}
\|\partial_\xi(v_n-S_ \varepsilon v_n)\|_{L^2(\R\times\T^N)} \le C\|g_n-S_ \varepsilon g_n\|_{L^2(\R\times\T^N)}, 
\end{equation} 
where $C$ is independent of $n$ and $ \varepsilon$.

For each fixed $ \varepsilon>0$ and $n\in\N$, Young’s convolution inequality also yields
\begin{equation}\label{paritialxi-S} 
 \|\partial_\xi S_ \varepsilon(v_n-v_0)\|_{L^2(\R\times\T^N)} = \|\rho_ \varepsilon' *_\xi(v_n-v_0)\|_{L^2(\R\times\T^N)} 
\leq \|\rho_ \varepsilon'\|_{L^1(\mathbb R)} \|v_n-v_0\|_{L^2(\R\times\T^N)}. 
\end{equation}
Combining \eqref{S-convlution} and \eqref{paritialxi-S}, we obtain
\begin{equation*}
\begin{aligned} \|\partial_\xi v_n-v_0'\|_{L^2(\R\times\T^N)} & \leq  C\|g_n-S_ \varepsilon g_n\|_{L^2(\R\times\T^N)}+\|S_ \varepsilon v_0'-v_0'\|_{L^2(\R\times\T^N)}\\ 
	&\qquad\qquad \qquad\qquad\qquad\qquad +\|\rho_ \varepsilon'\|_{L^1(\mathbb R)} \|v_n-v_0\|_{L^2(\R\times\T^N)}. \end{aligned} 
\end{equation*}
Since $g_n\to g$ by assumption and $v_n\to v_0$ in $L^2(\R\times\T^N)$ by Step 2, letting $n\to+\infty$ with $\varepsilon>0$ fixed, yields
$$\limsup_{n\to+\infty} \|\partial_\xi v_n-v_0'\|_{L^2(\R\times\T^N)} \leq C\|g-S_ \varepsilon g\|_{L^2(\R\times\T^N)}+\|S_ \varepsilon v_0'-v_0'\|_{L^2(\R\times\T^N)}.  $$
Notice that both terms on the right-hand side tend to zero as $ \varepsilon\to0^+$. Consequently, we have
$$ \partial_\xi v_n\to v_0' \quad\text{strongly in }\,L^2(\R\times\T^N). $$
Together with the $L^2$-convergence from Step 2 and
$\|\nabla_yv_n\|_{L^2(\R\times\T^N)}\to0$ from \eqref{gradient-vn}, this proves $v_n \to v_0$ strongly in $H^1(\R \times \T^N)$. Thus, \eqref{homogenization process} is obtained.

Finally, to prove \eqref{limit-Ln0}, recall that
$v_0=M_{c,0,e}^{-1}(\bar g)\in H^2(\R)$.
A direct computation gives
$$
\begin{aligned}
	\varphi_n-M_{c_n,0,e_n}(v_0)
	={}&\varphi_n-\bar g
	+\bigl(A^{\rm hom}e\cdot e
	-A^{\rm hom}e_n\cdot e_n\bigr)v_0''\\
	&+\bigl(\langle B\rangle_A\cdot(e-e_n)
	+c-c_n\bigr)v_0'
	\to0
	\quad\hbox{in }L^2(\R).
\end{aligned}$$
Since $
M_{c_n,0,e_n}^{-1}(\varphi_n)-v_0= M_{c_n,0,e_n}^{-1}(\varphi_n-M_{c_n,0,e_n}v_0)$,
Lemma \ref{solvability-M} yields \eqref{limit-Ln0}.
The proof of Lemma \ref{Homogenization} is thus complete.
\end{proof}

The following is the counterpart of Lemma \ref{Homogenization} for the case when $L_n\to L\neq0$ as $n\to+\infty$.

\begin{lemma}\label{continuem2}
Fix $c>0$, $L\in \R^*$ and $e\in \mathbb{S}^{N-1}$. Then for any $g \in L^2(\R \times \T^N)$ and any sequences $\{g_n\}_{n\in\N}\subset L^2(\R\times\T^N)$, $\{c_n\}_{n\in\N}\subset (0, +\infty)$, $\{L_n\}_{n\in\N} \subset \R^*$ and $\{e_n\}_{n\in\N} \subset \mathbb{S}^{N-1}$ such that $\|g_n-g\|_{L^2(\R\times\T^N)}\to 0$, $c_n\to c$, $L_n\to L$ and $e_n\to e$ as $n\to+\infty$, we have
\begin{equation}\label{converge-L}
	M_{c_n,L_n,e_n}^{-1}(g_n) \to M_{c,L,e}^{-1}(g) \;
	\hbox{ in }\; H^{1}(\R \times \T^N) \;\hbox{ as }\; n\to+\infty.
\end{equation}
\end{lemma}

\begin{proof}
Let $w_n=M_{c_n,L_n,e_n}^{-1}(g_n)$ and $w=M_{c,L,e}^{-1}(g)$. Since $c_n\to c>0$ and $L_n\to L\neq0$, Lemma \ref{solvability-M} implies that $\{w_n\}_{n\in\N}$ is bounded in $H^1(\R\times\T^N)$. For any weakly convergent subsequence with limit $u$, passing to the limit in the weak equation gives $M_{c,L,e}u=g$ in $H^{-1}(\R\times\T^N)$. Moreover, since $\tilde\nabla_{L,e}\cdot(A\tilde\nabla_{L,e}u) = g-B\cdot\tilde\nabla_{L,e}u -c\partial_\xi u+\beta u \in L^2(\R\times\T^N)$, we have $u\in\mathcal D_{L,e}$ and, by uniqueness, $u=w$. Consequently, the entire sequence satisfies $w_n\to w$ weakly in $H^1(\R\times\T^N)$. 
It is straightforward to check that
\begin{equation*}
\begin{aligned}
M_{c_n,L_n,e_n}(w_n-w)&=-M_{c_n,L_n,e_n}(w)+M_{c,L,e}(w)+g_n-g\\
&=\left(\tilde{\nabla}_{L,e} \cdot (A \tilde{\nabla}_{L,e} w)- \tilde{\nabla}_{L_n,e_n} \cdot (A \tilde{\nabla}_{L_n,e_n} w)\right)+\left(B \cdot \tilde{\nabla}_{L,e} w - B \cdot \tilde{\nabla}_{L_n,e_n} w \right)\\
&\qquad +(c-c_n)\partial_{\xi}w+(g_n-g).
\end{aligned}
\end{equation*}
Testing this identity with $-(w_n-w)$ and integrating
the second diffusion term by parts gives
$$-\left(M_{c_n,L_n,e_n}(w_n-w),w_n-w\right)_{(H^{-1};H^{1})}=\int_{\R \times \T^N}\left(P_n^1+P_n^2+P_n^3\right), $$
where $(\cdot,\cdot)_{(H^{-1};H^1)}$ denotes the duality pairing
between $H^{-1}(\R\times\T^N)$ and $H^1(\R\times\T^N)$, and
\begin{equation*}
\left\{\begin{aligned}
P_n^1&=- (w_n-w)\tilde{\nabla}_{L,e} \cdot (A \tilde{\nabla}_{L,e} w)-(A\tilde{\nabla}_{L_n,e_n}w)\cdot \tilde{\nabla}_{L_n,e_n} (w_n-w),  \vspace{5pt}\\
P_n^2&=\bigl(B\cdot\tilde\nabla_{L_n,e_n}w
-B\cdot\tilde\nabla_{L,e}w\bigr)(w_n-w),\vspace{5pt}\\
P_n^3 &= (c_n-c)\partial_{\xi}w(w_n-w) +(g-g_n)(w_n-w). 
	\end{aligned}\right.
\end{equation*}
Notice that
$\tilde\nabla_{L,e}\cdot(A\tilde\nabla_{L,e}w)
\in L^2(\R\times\T^N)$ and, since $L_n\to L\neq0$ and
$e_n\to e$, we have
$$
\tilde\nabla_{L_n,e_n}w\to\tilde\nabla_{L,e}w
\quad\hbox{strongly in }\,L^2(\R\times\T^N),$$
and
$$
\tilde\nabla_{L_n,e_n}(w_n-w)\rightharpoonup0
\quad\hbox{weakly in }\,L^2(\R\times\T^N).$$
These convergences, together with $w_n\rightharpoonup w$
in $H^1(\R\times\T^N)$, $c_n\to c$ and
$g_n\to g$ in $L^2(\R\times\T^N)$, imply that
$$\int_{\R \times \T^N} \left(P_n^1+P_n^2+P_n^3 \right)\to 0 \quad\hbox{as }\; n\to+\infty.$$
On the other hand, employing an estimate analogous to \eqref{key-inequality}, we obtain that for every $n\in\N$, 
\begin{equation*}
-\left(M_{c_n,L_n,e_n}(w_n-w),w_n-w\right)_{(H^{-1};H^{1})} \geq \frac{1}{2}\int_{\R \times \T^N} \eta_0|\tilde{\nabla}_{L_n,e_n}(w_n-w)|^2 + \beta (w_n-w)^2.
\end{equation*}
As a consequence, $\|w_n - w\|_{L^2(\R\times\T^N)} \to 0$ and
\begin{equation}\label{converge-wn-w}
  \|\tilde{\nabla}_{L_n,e_n}(w_n-w)\|_{L^2(\R\times\T^N)} \to 0 \;\hbox{ as }\;n\to+\infty.
\end{equation}

Furthermore, since the coefficients of $M_{c_n,L_n,e_n}$ are independent of $\xi$, the mollification
argument in Step 3 of the proof of Lemma \ref{Homogenization}
applies here with $v_n$ and $v_0$ replaced by $w_n$ and $w$.
Using $g_n\to g$ in $L^2(\R\times\T^N)$, the global
$L^2$-convergence above and the uniform estimate in
Lemma \ref{solvability-M}, we obtain
\begin{equation}\label{converge-wn-xi}
	\partial_\xi w_n\to\partial_\xi w
	\quad\hbox{strongly in }L^2(\R\times\T^N).
\end{equation}

Finally, notice that
\begin{equation*}
	\begin{aligned}
\int_{\R\times\T^N}|\nabla_y(w_n-w)|^2&=L_n^2\int_{\R\times\T^N} |\tilde{\nabla}_{L_n,e_n}(w_n-w)-e_n\partial_{\xi}(w_n-w)|^2 \vspace{5pt}\\
&\leq 2L_n^2   \int_{\R\times\T^N} |\tilde{\nabla}_{L_n,e_n}(w_n-w)|^2+|\partial_{\xi}(w_n-w)|^2	\mathop{\longrightarrow}_{n\to+\infty} 0,
	\end{aligned}
\end{equation*}
where the convergence follows directly from \eqref{converge-wn-w}, \eqref{converge-wn-xi} and the boundedness of $\{L_n\}$. Together with the $L^2$-convergence, this proves that
$w_n\to w$ strongly in $H^1(\R\times\T^N)$.
\end{proof}

\SE{Existence of pulsating waves: Proof of Theorem \ref{existence small period}}

In this section, we prove that, for all sufficiently small $L>0$,
problem \eqref{front equation} admits a solution
$(\phi_{L,e},c_{L,e})$ satisfying the limits at infinity \eqref{limit-condition} by applying the implicit function theorem.  
The proof also gives that these solutions converge to the homogeneous wave
$(\phi_{0,e},c_{0,e})$ of \eqref{homo-wave} as $L\to0^+$,
up to translation of the limiting profile.
To fix the translation invariance of $\phi_{0,e}$, we impose
the normalization
\begin{equation}\label{normalization0}
	\phi_{0,e}(0)=\frac12
	\quad\hbox{for all }e\in\mathbb S^{N-1}.
\end{equation}


\subsection{Basic properties of the homogeneous wave $(\phi_{0,e},c_{0,e})$}

We first collect several basic properties of the homogeneous wave $(\phi_{0,e},c_{0,e})$, including its dependence on the propagation direction and the spectral properties of the associated linearized operator.

\begin{lemma}\label{continuity-homo-wave}
For each $e\in\mathbb S^{N-1}$, the functions $\phi_{0,e}(\xi)$ and $1-\phi_{0,e}(\xi)$ decay exponentially as $\xi\to +\infty$ and $\xi\to -\infty$, respectively. Moreover, the normalized homogeneous wave $(\phi_{0,e}(\cdot),c_{0,e})$ depends continuously on $e$ in the sense that for every sequence $\{e_n\}_{n\in\N}\subset \mathbb{S}^{N-1}$ such that $e_n\to e$, we have 
$$c_{0,e_n}\to c_{0,e}, \quad   \phi_{0,e_n}\to \phi_{0,e}\,\hbox{ in }\,L^{\infty}(\R),$$ and 
$$\phi^{(k)}_{0,e_n}\to \phi^{(k)}_{0,e}\,\hbox{ in }\, L^2(\R)\cap L^{\infty}(\R)\,\hbox{ for all }\,1\leq k\leq  5.$$ 
\end{lemma}

\begin{proof}
We express $(\phi_{0,e},c_{0,e})$ in terms of a normalized traveling wave of an isotropic equation. More precisely,
under {\rm (H1)}, let $(\tilde\phi,\tilde c)$ denote
the unique traveling wave satisfying
	$$	\tilde{\phi}'' + \tilde{c}\,\tilde{\phi}' + \bar{f}(\tilde{\phi}) = 0 \; \text{ in }\; \R, \quad 
\tilde{\phi}(-\infty)=1,\quad \tilde{\phi}(+\infty)=0,\quad\hbox{and}\quad \tilde{\phi}(0)=\frac{1}{2}. $$
Clearly, the pair $(\tilde\phi,\tilde c)$ is independent of $e$.
Comparing this equation with \eqref{homo-wave} and using the normalization \eqref{normalization0},
we obtain
\begin{equation}\label{scaling-relation}
	\phi_{0,e}(\xi)= \tilde{\phi}\!\left( \frac{\xi}{\sqrt{A^{\rm hom}e\cdot e}}\right) \;\hbox{ for } \xi\in\mathbb R,
	\quad\text{and}\quad
	c_{0,e}=\tilde{c}\sqrt{A^{\rm hom}e\cdot e}-\langle B\rangle_A\cdot e.
\end{equation}
Moreover, since $\bar{f}'(0)<0$ and $\bar{f}'(1)<0$, it is well known (see e.g. \cite{aw,fm}) that $\tilde{\phi}(\xi)$ and $1-\tilde{\phi}(\xi)$ decay exponentially as $\xi\to+\infty$ and $\xi\to-\infty$, respectively. 
Since $\bar f\in C^3([0,1])$, the traveling-wave equation implies that $\tilde\phi\in C^5(\R)$.
Differentiating the equation further shows that
$\tilde\phi^{(k)}(\xi)$ decays exponentially as $\xi\to\pm \infty$ for $1\le k\le4$.

By the positivity of the effective diffusion and
the compactness of $\mathbb S^{N-1}$, there exist
constants $a_*,a^*>0$ such that
$a_*\le A^{\rm hom}e\cdot e\le a^*$ for all $e\in\mathbb S^{N-1}$.
Thus, \eqref{scaling-relation} yields positive constants
$C>0$ and $\mu>0$, independent of $e$, such that
$$	\phi_{0,e}(\xi)\leq  Ce^{-\mu\xi} \,\hbox{ for }\, \xi\geq 0, \quad 	1-\phi_{0,e}(\xi) \leq  Ce^{\mu\xi}\, \hbox{ for }\, \xi\leq 0,$$
and 
 $$|\phi_{0,e}^{(k)}(\xi)|\le Ce^{-\mu|\xi|} \,\hbox{ for }\,\xi\in\R,\, 1\leq k\leq 5. $$
In particular, we have
$\sup_{e\in\mathbb S^{N-1}}
\|\phi_{0,e}'\|_{H^4(\R)}<+\infty$. 

Now, for every sequence $\{e_n\}_{n\in\N}\subset \mathbb{S}^{N-1}$ such that $e_n\to e$,
the convergence $c_{0,e_n}\to c_{0,e}$ follows immediately from \eqref{scaling-relation}.
For the profiles, the same relation gives
$$\phi_{0,e}^{(k)}(\xi)=
(A^{\rm hom}e\cdot e)^{-k/2}
\tilde\phi^{(k)}\!\left(\frac{\xi}{\sqrt{A^{\rm hom}e\cdot e}}
\right),\qquad 0\le k\leq 5.$$
Consequently,
$\phi_{0,e_n}^{(k)}\to\phi_{0,e}^{(k)}$ uniformly on every compact interval as $n\to+\infty$. 
Together with the uniform exponential estimates above, this implies the convergence in $L^\infty(\R)$ for $0\leq k\leq 5$.
For each $1\le k\leq 5$, the same estimates and the dominated convergence theorem also give the convergence in $L^2(\R)$.
This completes the proof of Lemma \ref{continuity-homo-wave}. 
\end{proof}

Next, we fix $e\in \mathbb{S}^{N-1}$, and consider the linearization of equation \eqref{homo-wave} at $\phi_{0,e}$.   
Define
\begin{equation}\label{define-He}
	H_e: H^{2}(\R) \to L^2(\R),
	\quad 
	u \mapsto (A^{{\rm hom}} e \cdot e) u'' + (\langle B \rangle_{A}\cdot e + c_{0,e})u' + \bar{f}'(\phi_{0,e})u,
\end{equation}
and the adjoint operator 
$$H_e^*: H^{2}(\R) \to L^2(\R),\quad v\mapsto (A^{{\rm hom}} e \cdot e) v'' - (\langle B \rangle_A \cdot e + c_{0,e})v' + \bar{f}'(\phi_{0,e})v,$$
in such a way that
$(H^*_e(v),u)_{L^2(\R)}=(v,H_e(u))_{L^2(\R)}$
for all $u,v \in H^{2}(\R)$.
By the definition of $M_{c,0,e}$ in \eqref{defi-M0}, it is easily seen that, when $c_{0,e}>0$,  
\begin{equation}\label{collection-M-H}
	H_e(v)=M_{c_{0,e},0,e}(v)+(\bar{f}'(\phi_{0,e})+\beta)  v \quad\hbox{for } \;v\in H^{2}(\R).
\end{equation}
In the following lemma, we collect the spectral properties of $H_e$ and $H^*_e$.  

\begin{lemma}\label{properties of H}
 $H_e$ and $H_e^*$ have an algebraically simple eigenvalue $0$, and ${\rm Ker} (H_e)=\R \phi'_{0,e}$, ${\rm Ker} (H^*_e)=\R w_e$, where $w_e$ is given by \eqref{adjoint-kernel}, with the dependence on $e$ made explicit.
  Moreover, the range of $H_e$ is closed in $L^2(\R)$.
\end{lemma}

\begin{proof}
By \eqref{scaling-relation}, we have
$\langle B\rangle_A\cdot e+c_{0,e}=\tilde c\sqrt{A^{\rm hom}e\cdot e}$.
Hence, under the change of variable
$\xi=\sqrt{A^{\rm hom}e\cdot e}\,s$, the operators $H_e$ and $H_e^*$ reduce to
$$\frac{d^2}{ds^2}+\tilde c\frac{d}{ds}+\bar f'(\tilde\phi(s))
\quad\hbox{and}\quad \frac{d^2}{ds^2}-\tilde c\frac{d}{ds}+\bar f'(\tilde\phi(s)),$$
respectively. Then, the arguments in \cite[Section 4]{he}
yield all the stated spectral properties. We do not repeat the details here. 
\end{proof}

\subsection{Application of the implicit function theorem}

Combining these properties with the preliminary results established in Section 2, we are now ready to prove Theorem \ref{existence small period}. We begin by introducing the framework needed for the implicit function theorem.
Recall that $\chi:\T^N\to\R^N$ is the unique solution of the cell problem \eqref{Acorrector}.
For any $v \in H^{ 1}(\R \times \T^N)$, $c>0$, $e\in\mathbb{S}^{N-1}$, and $L\in \R$, define
\begin{equation*}
\begin{aligned}
K(v,c,L,e) =\;	& (A(I_N+ \nabla \chi) e \cdot e + \nabla \cdot (A (\chi \cdot e)e))\phi_{0,e}'' 
	+ L(A e \cdot e) (\chi \cdot e)\phi_{0,e}^{(3)}  \vspace{5pt}\\
	&\quad + B \cdot (I_N+ \nabla \chi)e \phi_{0,e}' + L (B \cdot e) (\chi \cdot e)\phi_{0,e}''
	+ c(\phi_{0,e}' + L(\chi \cdot e)\phi_{0,e}'') \vspace{5pt}\\
		&\quad 	 + f(y,\phi_{0,e} + L(\chi \cdot e)\phi_{0,e}' + v(\xi,y))+ \beta v(\xi,y),
	\end{aligned}
\end{equation*}
where $A$, $B$, $\chi$, and $\nabla\chi$ are evaluated at $y$, while $\phi_{0,e}$ and its derivatives are evaluated at $\xi$. For $L\neq 0$, a direct computation yields 
\begin{equation}\label{raltion-K-M0}
K(v,c,L,e)=M_{c,L,e}(\phi_{0,e}+L(\chi \cdot e )\phi'_{0,e})+f(y,\phi_{0,e} + L(\chi \cdot e)\phi_{0,e}' + v)+\beta (\phi_{0,e} + L(\chi \cdot e)\phi_{0,e}' + v).
\end{equation}

We claim that  $K(v,c,L,e)\in L^{2}(\R \times \T^N)$. Indeed,
by the exponential decay of $\phi_{0,e}(\xi)$ and $1-\phi_{0,e}(\xi)$ as $\xi\to \pm \infty$ (see Lemma \ref{continuity-homo-wave}), we have $\phi_{0,e}\in L^2(\R^+)$, $1-\phi_{0,e}\in L^2(\R^-)$, and $\phi_{0,e}^{(k)}\in L^2(\R)$ for $1\leq k\leq 3$. Hence, all the linear terms in the definition of $K$ belong to $L^{2}(\R \times \T^N)$. Moreover, since 
$\partial_u f(y,u)$ is bounded in $(y,u)\in\T^N\times\R$ and $f(y,0)=0$, we have 
\begin{equation*}
	\begin{aligned}
		&\int_{\R^+\times\T^N}f^2(y,\phi_{0,e} + L(\chi \cdot e)\phi_{0,e}'+v) \\
		= \;&\int_{\R^+\times\T^N} \left(f(y,\phi_{0,e} + L(\chi \cdot e)\phi_{0,e}'+v)-f(y,0)\right)^2 \vspace{5pt}\\ 
		\leq\; &  C_L\left( \|\phi_{0,e}\|^2_{L^2(\R^+)} + \|\phi_{0,e}'\|^2_{L^2(\R\times\T^N)}+\|v\|^2_{L^2(\R\times\T^N)}\right) <\infty,
	\end{aligned}
\end{equation*}
where $C_L$ is a positive constant depending on $L$. Similarly, using $f(y,1)=0$, we obtain  
$\int_{\R^-\times\T^N}f^2(y,\phi_{0,e} + L(\chi \cdot e)\phi_{0,e}'+v)<+\infty$.  
Therefore, 
\begin{equation}\label{f-L2}
\int_{\R\times\T^N}f^2(y,\phi_{0,e} + L(\chi \cdot e)\phi_{0,e}'+v)<+\infty,
\end{equation} and combining the above, we have
 $K(v,c,L,e)\in L^{2}(\R \times \T^N)$.

Now, for any $v \in H^{1}(\R \times \T^N)$, $c>0$, $e\in\mathbb{S}^{N-1}$ and $L \in \R$, define
\begin{equation*}
G(v,c,L,e) = (G_1, G_2)(v,c,L,e),
\end{equation*}
where
\begin{equation}\label{defi-G}
\begin{cases}
G_1(v,c,L,e) = 
\begin{cases}
	v + M_{c,L,e}^{-1}(K(v,c,L,e)) \quad 
&\hbox{ if }\; L \neq 0, \vspace{5pt}\\
v + M_{c,0,e}^{-1}(\overline{K(v,c,0,e)}) \quad 
&\hbox{ if }\; L = 0,
\end{cases} \vspace{7pt}\\
G_2(v,c,L,e) = 
	\displaystyle \int_{\R^+ \times \T^N} 
	\left(\phi_{0,e} 
	+ L (\chi\cdot e) \phi_{0,e}'
	+ v\right)^2 
	- \int_{\R^+} \phi_{0,e}^2,
\end{cases}
\end{equation}
Here, the function $\overline{K(v,c,0,e)}\in L^2(\R)$ denotes the mean of $K(v,c,0,e)$ over $\T^N$, 
as in \eqref{mean definition}.
By the invertibility and energy estimates established for $M_{c,L,e}$ in Lemma \ref{solvability-M}, 
the function $G$ is therefore well defined, and maps 
$H^{ 1}(\R \times \T^N) \times (0, +\infty) \times \R\times \mathbb{S}^{N-1}$ into $H^{ 1}(\R \times \T^N) \times \R$.

In the following sequence of lemmas, we prove several properties of $G$ needed to apply the implicit function theorem.
We begin with the following identity for the effective drift $\langle B \rangle_{A}$.

\begin{lemma}\label{lem-mean-B}
	Let $\langle B \rangle_{A}$ be the effective drift defined in \eqref{homogenized limit w.r.t. A}. Then for any $e\in\mathbb{S}^{N-1}$, we have 
	$$	\langle B \rangle_{A}\cdot e=  \int_{\T^N}  B\cdot (e + \nabla (\chi \cdot e)).$$
\end{lemma}

\begin{proof}
	Integrating by parts and using equation \eqref{Acorrector} for $\chi(\cdot)$ gives
	\begin{equation*}
		\begin{aligned}
			\int_{\T^N} A^T \nabla\zeta \cdot e&= \int_{\T^N} (Ae) \cdot \nabla \zeta=-\int_{\T^N}\nabla \cdot (Ae)  \zeta=\int_{\T^N}\nabla \cdot (A \nabla  \chi e)  \zeta\\
			&=-\int_{\T^N}A (\nabla  \chi  e) \cdot \nabla  \zeta = -\int_{\T^N}   (A^T \nabla \zeta)\cdot \nabla  (\chi \cdot e)=\int_{\T^N}  \nabla \cdot (A^T \nabla \zeta) (\chi \cdot e).
		\end{aligned}	
	\end{equation*}
	Furthermore, using equation \eqref{Bcorrector} for $\zeta(\cdot)$, we obtain 
	$$\int_{\T^N}  A^T \nabla\zeta \cdot e = \int_{\T^N}   (\nabla \cdot  B) (\chi \cdot e)=-\int_{\T^N} B \cdot  \nabla (\chi \cdot e).$$
	As a consequence, we have
	$$	\langle B \rangle_{A}\cdot e =\int_{\T^N} (B-A^T \nabla \zeta ) \cdot e=  \int_{\T^N} 	B\cdot  (e + \nabla (\chi \cdot e)).$$
	This completes the proof of Lemma \ref{lem-mean-B}.
\end{proof}

As a consequence of the above identity, we have the following observation. 
\begin{lemma}\label{identity-G0}
For any direction $e\in\mathbb{S}^{N-1}$ with $c_{0,e}>0$, we have 
$G(0,c_{0,e},0,e) = (0,0)$.
\end{lemma}

\begin{proof}
Since $(\phi_{0,e}(\cdot),c_{0,e})$ satisfies \eqref{homo-wave}, averaging $K(0,c_{0,e},0,e)$ over $\T^N$, and using Lemma \ref{lem-mean-B} and  the definition of $A^{\rm hom}$, we obtain
 $$\overline{K(0,c_{0,e},0,e)}=(A^{\rm hom}e\cdot e)\phi_{0,e}''
 	+\bigl(\langle B\rangle_A\cdot e+c_{0,e}\bigr)\phi_{0,e}'+\bar f(\phi_{0,e})=0.  
 $$
 Consequently, Lemma \ref{solvability-M} and the definition of $G_1$ yield
 $G_1(0,c_{0,e},0,e)=M_{c_{0,e},0,e}^{-1}(\overline{K(0,c_{0,e},0,e)})=0$. The identity $G_2(0,c_{0,e},0,e)=0$ follows directly from its definition.
\end{proof}

Next, we consider the nonlinear map induced by $f$. For a fixed direction $e\in\mathbb{S}^{N-1}$, define
\begin{equation}\label{define-mapg}
	g[v](\xi,y) = f(y,\phi_{0,e}(\xi) + v(\xi,y))\quad\hbox{ for }\; v \in H^{ 1}(\R \times \T^N).
\end{equation}
By \eqref{f-L2}, $g$ maps $H^{ 1}(\R \times \T^N)$ into $L^2(\R\times\T^N)$. 

\begin{lemma}\label{frechet-g}
The map $g:H^1(\R\times\T^N)\to L^2(\R\times\T^N)$
is continuously Fr\'echet differentiable, with derivative
	$$Dg[v]h=\partial_u f(y,\phi_{0,e}+v)h\quad \hbox{for all }\,\, v,\,h\in H^1(\R\times\T^N).$$
\end{lemma}

\begin{proof}
For each $v\in H^1(\R\times\T^N)$, define
$$\mathcal A(v)h=\partial_u f(y,\phi_{0,e}+v)h\quad\hbox{for }\,\,h\in H^{1}(\R\times\T^N).$$
By the boundedness of $\partial_u f$ on $\T^N\times\R$, $\mathcal A(v)$ is a bounded linear operator from
$H^1(\R\times\T^N)$ to $L^2(\R\times\T^N)$.
Choose an exponent $p\in(2,+\infty)$ such that $H^1(\R\times\T^N)$ is continuously embedded in
$L^p(\R\times\T^N)$.
Since $\partial_{u}^2 f$ is bounded on $\T^N\times\R$, there exists a constant $\omega>0$ such that 
$$ |\partial_u f(y,u_1)-\partial_u f(y,u_2)|
\leq\omega |u_1-u_2|\quad\hbox{for all }\,  y\in\T^N\,\hbox{ and } \,u_1,u_2\in\R.$$

We first prove the continuity of $\mathcal A$ in the operator norm. To this end,  for $v_1,v_2\in H^1(\R\times\T^N)$ and $\delta>0$, set
$$ E_\delta=\left\{(\xi,y)\in\R\times\T^N:\,|v_1(\xi,y)-v_2(\xi,y)|>\delta\right\}.$$
Then, we have $|E_\delta|\leq \delta^{-p}\|v_1-v_2\|_{L^p(\R\times\T^N)}^p$.
Using H\"older's inequality on $E_\delta$ and the Sobolev embedding,
we obtain, for every $h\in H^1(\R\times\T^N)$,
$$\begin{aligned}
	\|(\mathcal A(v_1)-\mathcal A(v_2))h\|_{L^2(\R\times\T^N)}
	&\leq \omega\delta\|h\|_{L^2(\R\times\T^N)}
	+2\|\partial_uf\|_{L^{\infty}(\T^N\times\R)}|E_\delta|^{\frac12-\frac1p}
	\|h\|_{L^p(\R\times\T^N)}\\
	&\leq\left(\omega\delta+C\delta^{1-p/2}
	\|v_1-v_2\|_{H^1(\R\times\T^N)}^{p/2-1}\right)\|h\|_{H^1(\R\times\T^N)},
\end{aligned}$$
where $C>0$ is independent  of $v_1,v_2,h$ and $\delta$.
Since $p>2$, letting first $v_1\to v_2$ in
$H^1(\R\times\T^N)$ and then $\delta\to0^+$ proves
the continuity of
$\mathcal A: H^1(\R\times\T^N) \to \Lcal(H^1(\R\times\T^N),L^2(\R\times\T^N))$.

Next, we show that $g$ is Fr\'echet differentiable. For any $v,h\in H^1(\R\times\T^N)$, set
$$
R_v(h)=g[v+h]-g[v]-\mathcal A(v)h.$$
The fundamental theorem of calculus gives
$R_v(h)=\int_0^1(\mathcal A(v+th)-\mathcal A(v))hdt$.
Consequently, by the continuity of $\mathcal A$ at $v$, we obtain
\begin{equation}\label{estimate-Rvh}
	\begin{aligned}
		\frac{\|R_v(h)\|_{L^2(\R\times\T^N)}}
		{\|h\|_{H^1(\R\times\T^N)}}
		&\le
		\sup_{0\le t\le1}
		\|\mathcal A(v+th)-\mathcal A(v)\|_{\Lcal(H^1,L^2)}
		\to0
	\end{aligned}
\end{equation}
as $\|h\|_{H^1(\R\times\T^N)}\to0$. 
Thus, $g$ is Fr\'echet differentiable with
$Dg[v]=\mathcal A(v)$.
This completes the proof of Lemma \ref{frechet-g}. 
\end{proof}

In the next lemma, we verify the continuity of $G$ and its differentiability with respect to $(v,c)$. 

\begin{lemma}\label{contiinuity-G}
The map $G: H^{ 1}(\R \times \T^N) \times (0, + \infty) \times \R\times\mathbb{S}^{N-1} \to H^{ 1}(\R \times \T^N) \times \R$ is continuous, and it is Fr\'{e}chet differentiable with respect to $(v,c)$ with derivatives given by
\begin{equation}\label{derivatives}
	\partial_{(v,c)}G(v,c,L,e)(\tilde{v},\tilde{c}) =
	\begin{pmatrix}
		\tilde{v} 
		+ M_{c,L,e}^{-1}\left(\partial_u f(y, \phi_{0,e} + L(\chi \cdot e) \phi_{0,e}' + v)\tilde{v}  + \beta\tilde{v}\right) \vspace{5pt}\\
		\quad +\; \tilde{c}M_{c,L,e}^{-1}\left(\phi_{0,e}' + L(\chi \cdot e) \phi_{0,e}'' 
		- \partial_\xi M_{c,L,e}^{-1}(K(v,c,L,e))\right) \vspace{5pt}\\
	\displaystyle	2 \int_{\R^+ \times \T^N} \left(\phi_{0,e} + L(\chi \cdot e) \phi_{0,e}' + v\right)\tilde{v}
	\end{pmatrix},
\end{equation}
when $L\neq 0$, and
\begin{equation}\label{derivative at 0}
\partial_{(v,c)}G(v,c,0,e)(\tilde{v},\tilde{c}) =
\begin{pmatrix}
\tilde{v} + M_{c,0,e}^{-1}\left(\overline{(\partial_u f(y, \phi_{0,e} + v)\tilde{v} + \beta\tilde{v})}\right) \vspace{5pt}\\
\quad +\; \tilde{c}M_{c,0,e}^{-1}\left(\phi_{0,e}'- \partial_\xi M_{c,0,e}^{-1}(\overline{K(v,c,0,e)})\right) \vspace{5pt}\\
	\displaystyle	2 \int_{\R^+ \times \T^N} (\phi_{0,e} + v)\tilde{v}
	\end{pmatrix},
\end{equation}
for all $(v,c) \in H^{ 1}(\R \times \T^N) \times (0, + \infty)$ and $(\tilde{v},\tilde{c}) \in H^{ 1}(\R \times \T^N) \times \R$. 
\end{lemma}

\begin{proof}
	For clarity, we divide the proof into two steps. 	
	
	{\bf Step 1:} We show the continuity of $G$. 
	
	We first consider the continuity of $G$ at a point $(v,c,0,e)$. 
	Let $\{v_n\}_{n \in \N}\subset H^{ 1}(\R \times \T^N)$, 
	$\{c_n\}_{n \in \N}\subset (0, +\infty)$, 
	$\{L_n\}_{n \in \N}\subset \R$ 
	and $\{e_n\}_{n \in \N}\subset \mathbb{S}^{N-1}$  be sequences such that 
   $v_n \to v$ in $H^{ 1}(\R \times \T^N)$, $c_n \to c$, $L_n \to 0$ and $e_n \to e$ as $n\to+\infty$. 
   
	We claim that
	\begin{equation}\label{continue-K}
		K(v_n, c_n, L_n, e_n) \to K(v,c,0,e) \quad \hbox{in }\; L^2(\R \times \T^N)\; \hbox{ as }\; n \to +\infty.
	\end{equation}
Indeed, by Lemma \ref{continuity-homo-wave} and the fact that $A$, $B$ and the relevant corrector coefficients are bounded, the following limits hold in $L^2(\R \times \T^N)$ as $n\to+\infty$:
	\begin{equation*}
		\begin{cases}
			\left(A(I_N+ \nabla \chi)e_n \cdot e_n  + \nabla \cdot (A(\chi \cdot e_n)e_n)\right)\phi_{0,e_n}''  \to  \left(A(I_N+ \nabla \chi)e \cdot e + \nabla \cdot (A(\chi \cdot e)e)\right) \phi_{0,e}'',\vspace{5pt}\\
			B \cdot (I_N+ \nabla \chi) e_n \phi_{0,e_n}'  \to  B \cdot (I_N+ \nabla \chi)e \phi_{0,e}', \vspace{5pt}\\
			c_n \phi_{0,e_n}' \to  c \phi_{0,e}'.
		\end{cases}
	\end{equation*}
	Moreover, it is straightforward to check that
	$\left\|\beta (v_n - v)\right\|_{L^2(\R \times \T^N)}  \to 0$,
	and 
	$$
	|L_n|\left\| (Ae_n \cdot e_n) (\chi \cdot e_n) \phi_{0,e_n}^{(3)}+  (B \cdot e_n) (\chi \cdot e_n) \phi_{0,e_n}''
	+  c_n (\chi \cdot e_n)\phi_{0,e_n}''\right\|_{L^2(\R \times \T^N)} \to 0,
	$$
	as $n\to+\infty$. Therefore, by the definition of $K$, to obtain \eqref{continue-K}, it remains to consider the nonlinear term $f$.  For each $n\in\N$, set
	\begin{equation}\label{define-wn-w}
	w_n(\xi,y) = \phi_{0,e_n}(\xi) + L_n (\chi(y) \cdot e_n)\phi_{0,e_n}'(\xi) + v_n(\xi,y)\quad \hbox{ and } \quad w(\xi,y) = \phi_{0,e}(\xi) + v(\xi,y).
	\end{equation}
	Since $v_n \to v$ in $L^{2}(\R \times \T^N)$, and  
	$\|\phi_{0,e_n}- \phi_{0,e}\|_{L^{2}(\R)}\to 0$ by Lemma \ref{continuity-homo-wave},
	we have $\|w_n - w\|_{L^{2}(\R \times \T^N)}\to 0$ as $n \to +\infty$. 
    Since $\partial_u f$ is bounded on $\T^N\times\R$, the mean value theorem gives
	\begin{equation*}
		\|f(y,w_n) - f(y,w)\|_{L^2(\R \times \T^N)}
		\leq \|\partial_u f\|_{L^{\infty}(\T^N\times\R)}\|w_n - w\|_{L^2(\R \times \T^N)} \to 0 \quad \hbox{as}\; n \to +\infty.
	\end{equation*}
	This proves  \eqref{continue-K}.
	
Then, along any subsequence for which $L_n\neq0$, Lemma \ref{Homogenization} yields
	\begin{equation*}
		M_{c_n,L_n,e_n}^{-1}(K(v_n,c_n,L_n,e_n)) \to M_{c,0,e}^{-1}(\overline{K(v,c,0,e)}) 
		\quad \hbox{in }\; H^{ 1}(\R \times \T^N) \;\hbox{ as }\; n \to +\infty.
	\end{equation*}
Similarly, along any subsequence for which $L_n=0$, we have $\overline{K(v_n,c_n,0,e_n)} \to \overline{K(v,c,0,e)}$ in $L^2(\R)$ as $n \to +\infty$,
	and the same lemma implies  
	$$M_{c_n,0,e_n}^{-1}(\overline{K(v_n,c_n,0,e_n)}) \to M_{c,0,e}^{-1}(\overline{K(v,c,0,e)})	\quad \hbox{in }\; H^{ 1}(\R) \;\hbox{ as }\; n \to +\infty.$$ 
	Thus, $G_1$ is continuous at every point with $L=0$. 

To prove the continuity of $G_2$,
note that the uniform exponential estimates in the proof of
Lemma \ref{continuity-homo-wave} imply that
$\sup_{n\in\N}(\|w_n\|_{L^2(\R^+\times\T^N)}
+\|\phi_{0,e_n}\|_{L^2(\R^+)})<+\infty$, where $w_n$ is defined in \eqref{define-wn-w}.
By the Cauchy-Schwarz inequality, we obtain 
\begin{equation*}
	\begin{aligned}
		&\left|G_2(v_n,c_n,L_n,e_n)-G_2(v,c,0,e)\right|\\
		\leq \,& \|w_n-w\|_{L^2(\R^+\times \T^N)}  \|w_n+w\|_{L^2(\R^+\times \T^N)}+\| \phi_{0,e_n} - \phi_{0,e}\|_{L^2(\R^+)}\| \phi_{0,e_n} + \phi_{0,e}\|_{L^2(\R^+)} \to 0,  
	\end{aligned}
\end{equation*}
as $n\to+\infty$. This gives the continuity of $G_2$ at $(v,c,0,e)$.

Next, at a point $(v,c,L,e)$ with $L\neq0$, the continuity
of $K$ follows from almost the same argument, while the continuity
of the inverse operators follows from Lemma \ref{continuem2}, instead of Lemma \ref{Homogenization}.
Therefore, $G$ is continuous on $H^1(\R\times\T^N)\times(0,+\infty)\times\R
\times\mathbb S^{N-1}$.

{\bf Step 2:} We show that $G$ is Fr\'echet differentiable with respect to $(v,c)$. 
	
Fix $(v,c,L,e)\in H^1(\R\times\T^N)\times(0,+\infty)\times\R \times\mathbb S^{N-1}$.
We first consider the case $L\neq0$.
Let $\tilde v\in H^1(\R\times\T^N)$ and $\tilde c\in\R$ satisfy
$$\|\tilde v\|_{H^1(\R\times\T^N)}+|\tilde c| <\min\left\{1,\frac c2\right\}.$$
Clearly, $c+\tilde c>0$. Since $M_{c+\tilde c,L,e}=M_{c,L,e}+\tilde c \partial_\xi$,
by the definitions of $K$ and $g$, we have
\begin{equation*}
	\begin{aligned}
		&M_{c+\tilde c,L,e}^{-1}
		\bigl(K(v+\tilde v,c+\tilde c,L,e)\bigr)
		-M_{c,L,e}^{-1}\bigl(K(v,c,L,e)\bigr)\\
		={}&
		M_{c,L,e}^{-1}\left(\beta\tilde v+\tilde c\bigl(
		\phi_{0,e}'+L(\chi\cdot e)\phi_{0,e}''\bigr)\right)-
		\tilde c\,M_{c,L,e}^{-1}\left(\partial_\xi M_{c+\tilde c,L,e}^{-1}
		\bigl(K(v+\tilde v,c+\tilde c,L,e)\bigr)\right)\\
		&\qquad+M_{c,L,e}^{-1}\Bigl(g[L(\chi\cdot e)\phi_{0,e}'+v+\tilde v]
		-g[L(\chi\cdot e)\phi_{0,e}'+v]\Bigr),
	\end{aligned}
\end{equation*}
where the function $g$ is defined in \eqref{define-mapg}.
Subtracting the candidate derivative in \eqref{derivatives}
from the increment of $G_1$, we obtain
$$G_1(v+\tilde v,c+\tilde c,L,e)-G_1(v,c,L,e)-\partial_{(v,c)}G_1(v,c,L,e)(\tilde v,\tilde c)=W_1-W_2,$$
where
$$
W_1=M_{c,L,e}^{-1}\Bigl(g[L(\chi\cdot e)\phi_{0,e}'+v+\tilde v]
	-g[L(\chi\cdot e)\phi_{0,e}'+v]-Dg[L(\chi\cdot e)\phi_{0,e}'+v]\tilde v\Bigr),$$
and
$$W_2=\tilde c\,M_{c,L,e}^{-1}\Bigl(\partial_\xi M_{c+\tilde c,L,e}^{-1}
	\bigl(K(v+\tilde v,c+\tilde c,L,e)\bigr)-\partial_\xi M_{c,L,e}^{-1}
	\bigl(K(v,c,L,e)\bigr)\Bigr).$$
Since $L(\chi\cdot e)\phi_{0,e}'+v\in
H^1(\R\times\T^N)$,  Lemma \ref{frechet-g}, the estimate \eqref{estimate-Rvh} and Lemma \ref{solvability-M} give
$$\|W_1\|_{H^1(\R\times\T^N)}=o\bigl(\|\tilde v\|_{H^1(\R\times\T^N)}\bigr)
\quad\hbox{as }\, \, \|\tilde v\|_{H^1(\R\times\T^N)}\to0.$$

To estimate $W_2$, by the definition of $G_1$, we have
\begin{equation*}
	\begin{aligned}
&M_{c+\tilde c,L,e}^{-1}\bigl(K(v+\tilde v,c+\tilde c,L,e)\bigr)-M_{c,L,e}^{-1}\bigl(K(v,c,L,e)\bigr)\\
={}&G_1(v+\tilde v,c+\tilde c,L,e)-G_1(v,c,L,e)-\tilde v.
	\end{aligned}
\end{equation*}
Applying Lemma \ref{solvability-M} once more, we obtain
$$\begin{aligned}
	\|W_2\|_{H^1(\R\times\T^N)}
	\le{}& C|\tilde c|\,
	\Bigl\|
	M_{c+\tilde c,L,e}^{-1}
	\bigl(K(v+\tilde v,c+\tilde c,L,e)\bigr)
	-M_{c,L,e}^{-1}\bigl(K(v,c,L,e)\bigr)
	\Bigr\|_{H^1(\R\times\T^N)}\\
	={}&
	o\bigl(\|\tilde v\|_{H^1(\R\times\T^N)}
	+|\tilde c|\bigr),
\end{aligned}$$
where $C$ is independent of $\tilde v$ and $\tilde c$, and the last estimate follows from the continuity proved in Step 1. 
Combining these estimates gives
$$\begin{aligned}
	&\Bigl\|G_1(v+\tilde v,c+\tilde c,L,e)-G_1(v,c,L,e)
	-\partial_{(v,c)}G_1(v,c,L,e)(\tilde v,\tilde c)\Bigr\|_{H^1(\R\times\T^N)}\\
	={}&o\bigl(\|\tilde v\|_{H^1(\R\times\T^N)}
	+|\tilde c|\bigr),
\end{aligned}$$
as $\|\tilde v\|_{H^1(\R\times\T^N)}+|\tilde c|\to 0$. 
Therefore, $G_1$ is Fr\'echet differentiable with respect to $(v,c)$ at every point with $L\neq0$, with derivative
given by the first component of \eqref{derivatives}.

For $L=0$, the same argument applies after averaging
over $y\in\T^N$. Using $M_{c,0,e}^{-1}\bigl(\overline{K(\cdot)}\bigr)$
in place of $M_{c,L,e}^{-1}(K(\cdot))$, we obtain the first component of \eqref{derivative at 0}.

Finally, the differentiability of $G_2$ follows from the exact identity
$$
G_2(v+\tilde v,c+\tilde c,L,e)-G_2(v,c,L,e)-2\int_{\R^+\times\T^N}
\bigl(\phi_{0,e}+L(\chi\cdot e)\phi_{0,e}'+v\bigr)
\tilde v=\int_{\R^+\times\T^N}\tilde v^2.
$$
Consequently, $G_2$ is Fr\'echet differentiable with respect to $(v,c)$, with derivative given by the second component
of \eqref{derivatives} or \eqref{derivative at 0}, according to whether $L\neq0$ or $L=0$.
This completes the proof of Step 2, and hence, Lemma \ref{contiinuity-G} is proved. 
\end{proof}

The above lemma and Lemma \ref{Homogenization} imply that,
for each fixed 
$(\tilde v,\tilde c)\in H^1(\R\times\T^N)\times\R$,
$$\partial_{(v,c)}G(v,c,L,e)(\tilde v,\tilde c) \to
\partial_{(v,c)}G(v,c,0,e)(\tilde v,\tilde c)
\quad\hbox{in }\;H^1(\R\times\T^N)\times\R \quad\hbox{as }\;L\to0.$$
However, these results do not directly imply that the convergence
is uniform for $(\tilde v,\tilde c)$ in the unit ball of
$H^1(\R\times\T^N)\times\R$, as required for convergence in
the operator norm.
The following lemma proves the operator-norm continuity of
$\partial_{(v,c)}G$ at $(0,c_{0,e},0,e)$ whenever $c_{0,e}>0$,
as well as the invertibility of the derivative at this point.

\begin{lemma}\label{continuity-operotor}
For any $e\in\mathbb{S}^{N-1}$ such that $c_{0,e}>0$, the map $\partial_{(v,c)}G$ is continuous at $(0,c_{0,e},0,e)$ in the operator norm on
$\Lcal(H^1(\R\times\T^N)\times\R)$. Moreover, the operator
$$Q=\partial_{(v,c)}G(0,c_{0,e},0,e): H^1(\R\times\T^N)\times\R \to H^1(\R\times\T^N)\times\R$$
is invertible.
\end{lemma}

\begin{proof}
Fix $e\in\mathbb S^{N-1}$ with $c_{0,e}>0$, and let $\{v_n\}_{n \in \N}\subset H^{ 1}(\R \times \T^N)$, 
$\{c_n\}_{n \in \N}\subset (0, +\infty)$, 
$\{L_n\}_{n \in \N}\subset \R$ 
and $\{e_n\}_{n \in \N}\subset \mathbb{S}^{N-1}$  be sequences such that 
$v_n \to 0$ in $H^{ 1}(\R \times \T^N)$, $c_n \to c_{0,e}$, $L_n \to 0$ and $e_n \to e$ as $n\to+\infty$. 
To prove the continuity of $\partial_{(v,c)}G$ at $(0,c_{0,e},0,e)$ in the operator norm, it suffices to show that, as $n\to+\infty$, 
\begin{equation}\label{frechel-GVC}
\bigl\| \partial_{(v,c)}G(v_n,c_n,L_n,e_n)-\partial_{(v,c)}G(0,c_{0,e},0,e)
\bigr\|_{\Lcal(H^1(\R\times\T^N)\times\R)}\to  0.
\end{equation}
We first consider a subsequence with $L_n\neq0$. The case $L_n=0$ will be addressed briefly at the end of the proof.

We begin with the continuity of $\partial_vG_1$. Since $v_n \to 0$ in $H^{ 1}(\R \times \T^N)$, $e_n \to e$ and $L_n \to 0$, by Lemma \ref{continuity-homo-wave} and the uniform exponential estimates in its proof, we have
$$\left\|\phi_{0,e_n}-\phi_{0,e}+L_n(\chi\cdot e_n)\phi_{0,e_n}'+v_n\right\|_{H^1(\R\times\T^N)}\mathop{\longrightarrow}_{n\to+\infty}0.$$
Then, applying Lemma \ref{frechet-g}, we obtain
$$\sup_{\|h\|_{H^1(\R\times\T^N)}\leq 1}\left\|\left(\partial_u f\left(y,\phi_{0,e_n}+L_n(\chi\cdot e_n)\phi_{0,e_n}'+v_n\right)-\partial_u f(y,\phi_{0,e})\right)h\right\|_{L^2(\R\times \T^N)}\mathop{\longrightarrow}_{n\to+\infty} 0.$$
Thus, by the derivative formulas of $\partial_vG_1$ given in Lemma \ref{contiinuity-G} and 
the uniform inverse estimate in Lemma \ref{solvability-M}, it remains to prove that, as $n\to+\infty$,
\begin{equation}\label{uniform-linearized-resolvent}
\sup_{\|h\|_{H^1}\leq 1} 
\left\|M_{c_n,L_n,e_n}^{-1}\left((\partial_u f(y,\phi_{0,e})+\beta)h\right)-
		 M_{c_{0,e},0,e}^{-1}\left(\overline{(\partial_u f(y,\phi_{0,e})+\beta)h}\right)\right\|_{H^1(\R\times\T^N)} \to 0
\end{equation}
To this end, fix $h\in H^1(\R\times\T^N)$ with $\|h\|_{H^1(\R\times\T^N)}\leq 1$, and set
$$ w_n=M_{c_n,L_n,e_n}^{-1}\left((\partial_u f(y,\phi_{0,e})+\beta)h\right),\qquad w_0=M_{c_{0,e},0,e}^{-1}\left(\overline{(\partial_u f(y,\phi_{0,e})+\beta)h}\right).$$
In the estimates below, $C>0$ is independent of all sufficiently large $n$ and of $h$ in the unit ball of
$H^1(\R\times\T^N)$, and may vary from line to line.

Since 
$\partial_\xi\left((\partial_u f(y,\phi_{0,e})+\beta)h\right)=\partial_{u}^2f(y,\phi_{0,e})\phi_{0,e}'h+(\partial_u f(y,\phi_{0,e})+\beta)\partial_\xi h$,
we have
$$
\|(\partial_u f(y,\phi_{0,e})+\beta)h\|_{L^2(\R\times\T^N)}+\|\partial_\xi((\partial_u f(y,\phi_{0,e})+\beta)h)\|_{L^2(\R\times\T^N)} \leq C.
$$
Differentiating the constant-coefficient equation for $w_0$ in $\xi$ gives
$$(A^{\rm hom}e\cdot e)w_0'''=
	\overline{\partial_\xi\bigl((\partial_u f(y,\phi_{0,e})+\beta)h\bigr)}-(\langle B\rangle_A\cdot e+c_{0,e})w_0''
+\beta w_0'.$$
Together with Lemma \ref{solvability-M}, this implies 
\begin{equation}\label{w0-H3-uniform}
	\|w_0\|_{H^3(\R)}\le C.
\end{equation}
Next,  using the mollification operator $S_\varepsilon u=\rho_\varepsilon *_\xi u$
introduced in the proof of Lemma \ref{Homogenization} (notice that the coefficients of $M_{c_n,L_n,e_n}$ are independent of $\xi$), we have
$$\begin{aligned}
\partial_\xi S_\varepsilon w_n =\rho_\varepsilon' *_\xi w_n
&=M_{c_n,L_n,e_n}^{-1}\bigl(\rho_\varepsilon' *_\xi((\partial_u f(y,\phi_{0,e})+\beta)h)\bigr)\\
&=M_{c_n,L_n,e_n}^{-1}\bigl(S_\varepsilon\partial_\xi((\partial_u f(y,\phi_{0,e})+\beta)h)\bigr).
\end{aligned}$$
For each fixed $n$ and $h$, passing to the limit as $\varepsilon\to 0^+$ yields 
$\partial_\xi w_n= M_{c_n,L_n,e_n}^{-1}\bigl(\partial_\xi((\partial_u f(y,\phi_{0,e})+\beta)h)\bigr)
\in H^1(\R\times\T^N)$.
Applying Lemma \ref{solvability-M} once again yields
\begin{equation}\label{wn-H2-uniform}
\|w_n\|_{H^1(\R\times\T^N)}
+\|\partial_\xi w_n\|_{H^1(\R\times\T^N)} \leq C.
\end{equation}

Now, for each $n\in\N$, set
$$z_n=w_n-w_0-L_n(\chi\cdot e_n)w_0'.$$
By \eqref{w0-H3-uniform} and the cell problem for $\chi$,
we have $z_n\in\mathcal D_{L_n,e_n}$ and $\|z_n\|_{H^1(\R\times\T^N)}\le C$.
Furthermore, Lemma \ref{solvability-M} gives
$$ \|\nabla_yw_n\|_{L^2(\R\times\T^N)} \leq |L_n|(\|\tilde\nabla_{L_n,e_n}w_n\|_{L^2(\R\times\T^N)}
	+\|\partial_\xi w_n\|_{L^2(\R\times\T^N)}) \leq C|L_n|.$$ 
Consequently, we have
\begin{equation}\label{unif-estaime-zn}
\|\nabla_yz_n\|_{L^2(\R\times\T^N)}\le C|L_n|.
\end{equation}
Using the cell problem for $\chi$ and the equation $M_{c_{0,e},0,e}(w_0)=\overline{(\partial_u f(y,\phi_{0,e})+\beta)h}$, we compute
\begin{equation}\label{linearized-corrected-equation}
	M_{c_n,L_n,e_n}(z_n)=R_n,
\end{equation}
where
$$
\begin{aligned}
	R_n={}&(\partial_u f(y,\phi_{0,e})+\beta)h-\overline{(\partial_u f(y,\phi_{0,e})+\beta)h}\\
	&+\Bigl(A^{\rm hom}e\cdot e -A(I_N+\nabla\chi)e_n\cdot e_n -\nabla_y\cdot\bigl(A(\chi\cdot e_n)e_n\bigr)\Bigr)w_0''\\
	&+\Bigl(\langle B\rangle_A\cdot e+c_{0,e}-B\cdot(I_N+\nabla\chi)e_n-c_n\Bigr)w_0'\\
	&-L_n(\chi\cdot e_n)\Bigl((Ae_n\cdot e_n)w_0'''+(B\cdot e_n+c_n)w_0''	-\beta w_0'\Bigr).
\end{aligned}
$$
By Lemma \ref{lem-mean-B} and the definition of $A^{\rm hom}$, together with \eqref{w0-H3-uniform}, we have
\begin{equation}\label{linearized-residual-bounds}
\|R_n\|_{L^2(\R\times\T^N)}\leq C,\quad \hbox{and}\quad  \|\overline{R_n}\|_{L^2(\R)} \leq	C\bigl(|L_n|+|c_n-c_{0,e}|+|e_n-e|\bigr).
\end{equation}

Testing \eqref{linearized-corrected-equation} with
$-z_n$ and using the energy inequality from
Lemma \ref{solvability-M}, we obtain
$$
\begin{aligned}
\frac{\eta_0}{2}\left\|\tilde\nabla_{L_n,e_n}z_n\right\|_{L^2(\R\times\T^N)}^2+\frac{\beta}{2}\left\|z_n\right\|_{L^2(\R\times\T^N)}^2
&\leq -\int_{\R\times\T^N} (R_n-\overline{R_n})(z_n-\overline{z_n})-\int_{\R\times\T^N}\overline{R_n}z_n\\
&\leq C\|\nabla_yz_n\|_{L^2(\R\times\T^N)} +\|\overline{R_n}\|_{L^2(\R)}\|z_n\|_{L^2(\R\times\T^N)},
\end{aligned}
$$
where the Poincar\'e inequality in $y$ was used in the second inequality. By \eqref{unif-estaime-zn},
\eqref{linearized-residual-bounds} and the uniform $H^1$ bound for $z_n$, it follows that
$$\frac{\eta_0}{2}\left\|\tilde\nabla_{L_n,e_n}z_n\right\|_{L^2(\R\times\T^N)}^2+\frac{\beta}{2}\left\|z_n\right\|_{L^2(\R\times\T^N)}^2\leq C\bigl(|L_n|+|c_n-c_{0,e}|+|e_n-e|\bigr).$$
Since
$\|w_n-w_0\|_{L^2(\R\times\T^N)}\leq \|z_n\|_{L^2(\R\times\T^N)}+C|L_n|$,
we conclude that
$$ \sup_{\|h\|_{H^1(\R\times\T^N)}\leq 1} \|w_n-w_0\|_{L^2(\R\times\T^N)}\mathop{\longrightarrow}_{n\to+\infty} 0.$$
Furthermore, \eqref{w0-H3-uniform} and  \eqref{wn-H2-uniform} give
$\|\partial_{\xi}^2(w_n-w_0)\|_{L^2(\R\times\T^N)}\leq C$.
Thus, integration by parts in $\xi$ yields
$$
\|\partial_\xi(w_n-w_0)\|_{L^2(\R\times\T^N)}^2=-\int_{\R\times\T^N}(w_n-w_0)\partial_{\xi}^2(w_n-w_0)\leq C\|w_n-w_0\|_{L^2(\R\times\T^N)}.
$$
Together with
$\|\nabla_y(w_n-w_0)\|_{L^2(\R\times\T^N)}=\|\nabla_yw_n\|_{L^2(\R\times\T^N)}\le C|L_n|$,
this proves \eqref{uniform-linearized-resolvent}, and therefore, we obtain
\begin{equation}\label{continuty-G1V}
\|\partial_vG_1(v_n,c_n,L_n,e_n)
-\partial_vG_1(0,c_{0,e},0,e)\|_{\Lcal(H^1(\R\times\T^N))}\mathop{\longrightarrow}_{n\to+\infty} 0 \quad\hbox{as }\,\,n\to+\infty.
\end{equation}

We next consider the continuity of $\partial_cG_1$. By Lemmas \ref{identity-G0} and \ref{contiinuity-G}, we have
$$ M_{c_n,L_n,e_n}^{-1}(K(v_n,c_n,L_n,e_n)) =G_1(v_n,c_n,L_n,e_n)-v_n \mathop{\longrightarrow}_{n\to+\infty} 0
\quad\hbox{in }\,\, H^1(\R\times\T^N).$$
Consequently,
\begin{equation*}
\phi_{0,e_n}'+L_n(\chi\cdot e_n)\phi_{0,e_n}''-\partial_\xi  M_{c_n,L_n,e_n}^{-1}(K(v_n,c_n,L_n,e_n))
\mathop{\longrightarrow}_{n\to+\infty}\phi_{0,e}'	\quad\hbox{in }\,\,L^2(\R\times\T^N).
\end{equation*}
The derivative formulas of $\partial_cG_1$ given in Lemma \ref{contiinuity-G} and Lemma \ref{Homogenization} therefore give
\begin{equation}\label{continuty-G1c}
	\begin{aligned}
		\partial_cG_1(v_n,c_n,L_n,e_n)
		{}&=M_{c_n,L_n,e_n}^{-1}\Bigl(\phi_{0,e_n}'
		+L_n(\chi\cdot e_n)\phi_{0,e_n}''-\partial_\xi M_{c_n,L_n,e_n}^{-1}(K(v_n,c_n,L_n,e_n))\Bigr)\\
		{}&	\mathop{\longrightarrow}_{n\to+\infty}  M_{c_{0,e},0,e}^{-1}(\phi_{0,e}')=	\partial_cG_1(0,c_{0,e},0,e)
	\end{aligned}
\end{equation}
in $H^1(\R\times\T^N)$.   Since the $c$-variable is scalar, this convergence also holds in $\Lcal(\R,H^1(\R\times\T^N))$.

For the second component, the Cauchy-Schwarz inequality gives
\begin{equation}\label{coninuity-G2-VC} \begin{aligned}
&\|\partial_{(v,c)}G_2(v_n,c_n,L_n,e_n)-\partial_{(v,c)}G_2(0,c_{0,e},0,e) \|_{\Lcal(H^1(\R\times\T^N)\times\R,\R)}\\
&\qquad\qquad\qquad\leq 2\left\| \phi_{0,e_n}-\phi_{0,e} +L_n(\chi\cdot e_n)\phi_{0,e_n}'+v_n \right\|_{L^2(\R^+\times\T^N)}
	\mathop{\longrightarrow}_{n\to+\infty} 0.
\end{aligned}\end{equation}
Combining \eqref{continuty-G1V}, \eqref{continuty-G1c} and \eqref{coninuity-G2-VC} proves
\eqref{frechel-GVC} along every subsequence with $L_n\neq0$.

For completeness, we consider a subsequence with $L_n=0$. In this case,
redefine
$w_n=M_{c_n,0,e_n}^{-1}(\overline{(\partial_u f(y,\phi_{0,e})+\beta)h})$,
and keep $w_0$ as above. Then
$$ M_{c_n,0,e_n}(w_n-w_0)= \bigl(A^{\rm hom}e\cdot e
	-A^{\rm hom}e_n\cdot e_n\bigr)w_0''+ \bigl(	\langle B\rangle_A\cdot(e-e_n)
	+c_{0,e}-c_n\bigr)w_0'.$$
By Lemma \ref{solvability-M} and \eqref{w0-H3-uniform}, we have
$$
\sup_{\|h\|_{H^1(\R\times\T^N)}\leq 1} \|w_n-w_0\|_{H^1(\R)} \leq
C\bigl(|c_n-c_{0,e}|+|e_n-e|\bigr) \mathop{\longrightarrow}_{n\to+\infty} 0. $$
Thus, \eqref{continuty-G1V} also holds along this subsequence.
The other two components are treated in the same way,
using the $L=0$ derivative formula and the corresponding convergence result in Lemma \ref{Homogenization}.
Therefore, \eqref{frechel-GVC} holds for the entire sequence.

Finally, Lemma \ref{properties of H}  allows us to apply the same Fredholm argument as in the proof of \cite[Lemma 3.4]{dhz1} in the one-dimensional case. This proves the invertibility of $Q$ and completes the proof of Lemma \ref{continuity-operotor}.
\end{proof}

We are now ready to prove Theorem \ref{existence small period}.

\begin{proof}[Proof of Theorem {\rm \ref{existence small period}}]

We first prove statement (i). 
Fix $e\in\mathbb{S}^{N-1}$ such that $c_{0,e}\neq0$.
We may assume without loss of generality that
\begin{equation}\label{assume-positive-speed}
	c_{0,e}>0.
\end{equation}
Indeed, if $c_{0,e}<0$, consider the reflected coefficients
\begin{equation}\label{reflection-wave}
\tilde A(y)=A(-y),\qquad\tilde B(y)=-B(-y),\qquad\hbox{and}\quad\tilde f(y,u)=-f(-y,1-u).
\end{equation}
The transformation $u(t,x)\mapsto1-u(t,-x)$ maps solutions of \eqref{equation} to solutions
of the equation with coefficients
$(\tilde A,\tilde B,\tilde f)$.
Regarding the wave profile and wave speed, it maps $\phi_{L,e}$ and $c_{L,e}$, respectively, to
$\tilde\phi_{L,e}(\xi,y)=1-\phi_{L,e}(-\xi,-y)$ and $\tilde c_{L,e}=-c_{L,e}$.
This transformation is reversible and preserves the assumptions on the coefficients.
Moreover, the reflected cell problems have correctors $-\chi(-y)$ and $\zeta(-y)$, respectively.
Consequently,
$\tilde A^{\rm hom}=A^{\rm hom},$ $\langle\tilde B\rangle_{\tilde A}=-\langle B\rangle_A$, and
$\bar{\tilde f}(u)=-\bar f(1-u)$. 
Thus, the normalized homogeneous wave for the reflected
equation is $\tilde\phi_{0,e}(\xi)=1-\phi_{0,e}(-\xi)$, $\tilde c_{0,e}=-c_{0,e}>0$.
It therefore suffices to prove {\rm (i)} under \eqref{assume-positive-speed}.
	
By Lemma \ref{identity-G0}, we have $G(0,c_{0,e},0,e)=(0,0)$. In view of Lemmas \ref{contiinuity-G} and
\ref{continuity-operotor}, we can apply the implicit function theorem with continuous parameters $L$ and
$e'$ (see e.g., \cite{deimling}).
Thus, there exist $L_*(e)>0$, a neighborhood $\Lambda(e)$ of $e$ in $\mathbb S^{N-1}$, and a continuous local solution map
$$(L,e')\mapsto(v_{L,e'},c_{L,e'}) \in H^1(\R\times\T^N)\times(0,+\infty),$$
defined for $|L|<L_*(e)$ and $e'\in\Lambda(e)$, such that
\begin{equation}\label{GL0}
G(v_{L,e'},c_{L,e'},L,e')=(0,0).
\end{equation}
After shrinking $\Lambda(e)$ if necessary, Lemma \ref{continuity-homo-wave} ensures that  $c_{0,e'}>0$ and that
$(0,c_{0,e'})$ lies sufficiently close to $(0,c_{0,e})$ for all $e'\in\Lambda(e)$.
Since $G(0,c_{0,e'},0,e')=(0,0)$ by Lemma \ref{identity-G0} again, the local uniqueness
provided by the implicit function theorem implies that the solution branch at $L=0$ is $(0,c_{0,e'})$.
In particular, we have
\begin{equation}\label{converge-vL}
(v_{L,e},c_{L,e})\to(0,c_{0,e}) \quad\hbox{in }\;H^1(\R\times\T^N)\times\R \quad\hbox{as }\;L\to0^+.
\end{equation}
	
For each $0<L<L_*(e)$, define 
\begin{equation}\label{vL-phiL}
\phi_{L,e}(\xi,y)= \phi_{0,e}(\xi) +v_{L,e}(\xi,y) +L(\chi(y)\cdot e)\phi_{0,e}'(\xi),\qquad (\xi,y)\in\R\times\T^N.
\end{equation}
By \eqref{GL0}, the definition of $G$ and \eqref{raltion-K-M0}, the pair $(\phi_{L,e},c_{L,e})$ satisfies
\eqref{front equation} in the weak sense and obeys the normalization
\begin{equation}\label{normalization-proof}
\int_{\R^+\times\T^N}\phi_{L,e}^2=\int_{\R^+}\phi_{0,e}^2.
\end{equation}
We now verify that this pair is a moving pulsating wave of \eqref{equation} in the sense of Definition \ref{defi-pulsating}. Set
$$U_{L,e}(t,x)= \phi_{L,e}(x\cdot e-c_{L,e}t,x/L).$$
Since $c_{L,e}>0$, the change of variables 
$(t,x)\mapsto (\xi,y)=(x\cdot e-c_{L,e}t,x/L)$ is invertible. Therefore, $U_{L,e}$ is an entire weak solution of \eqref{equation}.
Moreover, since $\phi_{L,e}-\phi_{0,e}\in L^2(\R\times\T^N)$ and $0<\phi_{0,e}<1$, the change of variables and
the periodicity in $y$ yield
$$\begin{aligned}
\sup_{(t_0,x_0)\in\R\times\R^N}\int_{t_0-1}^{t_0+1}\int_{B_1(x_0)}|U_{L,e}(t,x)|^2dxdt
\leq C\left(1+\|\phi_{L,e}-\phi_{0,e}\|_{L^2(\R\times\T^N)}^2\right)<+\infty,
\end{aligned} $$
where $C$ is independent of $(t_0,x_0)$.
Since $f(y,0)=0$ and $\partial_u f$ is bounded, local boundedness estimates for parabolic equations imply that $U_{L,e}$ is globally bounded. Parabolic regularity, applied also to the equation satisfied by $\partial_tU_{L,e}$, then gives the
required classical regularity of $\phi_{L,e}$. The corresponding uniform local estimates also imply
that $\phi_{L,e}$ is uniformly continuous.
Consequently, $\phi_{L,e}-\phi_{0,e}$ is uniformly continuous and belongs to $L^2(\R\times\T^N)$, and hence,
$$\sup_{y\in\T^N} |\phi_{L,e}(\xi,y)-\phi_{0,e}(\xi)| \to 0 \quad\hbox{as }\;|\xi|\to+\infty.$$
The limits of $\phi_{0,e}$ therefore yield that $\phi_{L,e}(\xi,y)$ satisfies  \eqref{limit-condition}. Furthermore, the strong parabolic maximum principle together with $f(y,u)>0$ (resp. $f(y,u)<0$) for all $(y,u)\in\R^N\times(-\infty,0)$ (resp. $(y,u)\in\R^N\times(1,+\infty)$) implies that $U_{L,e}$ ranges in $(0,1)$.
By the invertibility of the change of variables, we have $0<\phi_{L,e}<1$ in $\R\times\T^N$. Thus, $(\phi_{L,e},c_{L,e})$ is a moving pulsating wave of \eqref{equation} in the direction $e$.
	
Finally, it follows from \eqref{converge-vL} and \eqref{vL-phiL} that
$$
\|\phi_{L,e}-\phi_{0,e}\|_{H^1(\R\times\T^N)}\leq \|v_{L,e}\|_{H^1(\R\times\T^N)}+ L\|(\chi\cdot e)\phi_{0,e}'\|_{H^1(\R\times\T^N)} \to0\quad \hbox{as }\; L\to0^+.  $$
Together with $c_{L,e}\to c_{0,e}$ as $L\to 0^+$, this proves \eqref{convergence-phiL} for the homogeneous profile $\phi_{0,e}$ normalized by \eqref{normalization0}.
An appropriate shift of any other representative of the homogeneous profile gives the same conclusion, and hence, statement (i) is proved. 

Note that the same construction and verification apply to every $e'\in\Lambda(e)$ with $c_{0,e'}>0$. In particular, by \eqref{vL-phiL}, Lemma \ref{continuity-homo-wave}, and the continuity of the local solution map given by the implicit function theorem,  the following map
$$ [0,L_*(e))\times\Lambda(e) \to H^1(\R\times\T^N)\times\R,\quad 
(L,e')\mapsto \bigl(\phi_{L,e'}-\phi_{0,e'},c_{L,e'}-c_{0,e'}\bigr)$$
is continuous, with value $(0,0)$ at each point $(0,e')$.

Next, we assume that $c_{0,e}\neq0$ for every
$e\in\mathbb S^{N-1}$, and  prove statement (ii). By Lemma \ref{continuity-homo-wave} and compactness, we have
$\min_{e\in\mathbb S^{N-1}}|c_{0,e}|>0$. If $N=1$, there are only two directions,
and hence, statement (ii) follows directly from (i) by taking a common period bound for these two directions.
We may therefore assume that $N\ge2$.
Since $\mathbb{S}^{N-1}$ is connected and
$e\mapsto c_{0,e}$ is continuous and never vanishes,
$c_{0,e}$ has the same sign for all directions.
Applying the reflection \eqref{reflection-wave} to all directions if
necessary, it suffices to prove {\rm (ii)} under
the assumption that $c_{0,e}>0$ for every
$e\in\mathbb{S}^{N-1}$.

For each direction $e$, the construction above provides a neighborhood $\Lambda(e)$, a number $L_*(e)>0$,
and a continuous local family of normalized moving pulsating waves for $0<L<L_*(e)$, $e'\in\Lambda(e)$.
We now verify that these locally defined families agree wherever their parameter domains overlap.
Consider two local waves $(\phi_{L,e}^{(1)},c^{(1)}_{L,e})$ and $(\phi_{L,e}^{(2)},c^{(2)}_{L,e})$ with the same period $L$ and direction $e$. Since the speeds are positive, both waves are classical solutions of \eqref{front equation}, and hence satisfy {\rm (H2)}.
Theorem 1.2, whose proof will be given in Section 5 and is independent of the present theorem, then applies to these two waves and implies that 
$$c^{(1)}_{L,e} =c^{(2)}_{L,e} \quad\hbox{and}\quad \phi_{L,e}^{(1)}(\xi,y) =\phi_{L,e}^{(2)}(\xi+\tau,y) $$
for some $\tau\in\R$. 
Furthermore, since both profiles satisfy \eqref{normalization-proof}, we have $\tau=0$. Thus, the normalized profiles coincide, and the local families agree on every overlap.

By the compactness of $\mathbb S^{N-1}$, choose finitely many directions $e^1,\ldots,e^m$ such that
$\mathbb S^{N-1}=\bigcup_{i=1}^{m}\Lambda(e^i)$, and set
$$ L_*=\frac12\min_{1\le i\le m}L_*(e^i)>0.$$
Since the local families agree on overlaps, they define a single family $(\phi_{L,e},c_{L,e})$ for all
$0<L\le L_*$ and $e\in\mathbb S^{N-1}$.
Moreover, the map
$ (L,e)\mapsto (\phi_{L,e}-\phi_{0,e},c_{L,e}-c_{0,e}) \in H^1(\R\times\T^N)\times\R$
extends continuously to $[0,L_*]\times\mathbb S^{N-1}$, with value $(0,0)$ at every point $(0,e)$.
Since this parameter set is compact, the map is uniformly continuous. Consequently, 
$$ \sup_{e\in\mathbb S^{N-1}} \left(
\|\phi_{L,e}-\phi_{0,e}\|_{H^1(\R\times\T^N)}
+|c_{L,e}-c_{0,e}| \right) \to0 \quad\hbox{as }\;L\to0^+.$$
Thus, the admissible period range is independent of
$e$, and the convergence in \eqref{convergence-phiL}
is uniform with respect to $e$. This proves {\rm (ii)} and completes the proof of Theorem \ref{existence small period}.
\end{proof}


\SE{Sharp convergence rate: Proofs of Theorems \ref{theo-convergence} and \ref{theo-sharp-speed}} 
Throughout this section, we fix a propagation direction $e\in\mathbb{S}^{N-1}$ such that $c_{0,e}\neq0$, and
suppress the dependence on $e$ in the notation introduced in Sections 2--3.
As explained at the beginning of the proof of Theorem \ref{existence small period}, we may assume
without loss of generality that $c_0>0$.

The proofs of Theorems \ref{theo-convergence} and \ref{theo-sharp-speed} share a similar strategy, 
based on constructing higher-order approximations for the wave profile $\phi_L$ 
​in terms of $L$ around $\phi_0$. For the normalized profiles constructed in the proof
of Theorem \ref{existence small period}, we already have
$$\phi_L - \phi_0-L\chi_1 \phi_0'=o(1) \,\,\hbox{ in }\,\, H^{ 1}(\R\times \T^N)\,\, \hbox{ as } \,\,L\to 0^+,$$ 
where $\chi_1$ is defined by \eqref{corrector-chi1}. Here, $L\chi_1\phi_0'$ is the first-order microscopic
corrector. To obtain the $O(L)$ convergence estimates in
Theorem \ref{theo-convergence}, we introduce a second-order corrector.
Under the additional assumption that $f=f(u)$ is independent of $x$, we construct a third-order approximate profile to derive the first-order expansion of the wave speed in Theorem \ref{theo-sharp-speed}.

\subsection{Proof of Theorem \ref{theo-convergence}}
We construct a function $g\in H^1(\R\times\T^N)$, independent of $L$, and prove that
\begin{equation*}
\|\phi_L- \phi_0-L\chi_1\phi_0' -L^2g \|_{H^{ 1}(\R\times \T^N)} +|c_L-c_0| \leq CL.
\end{equation*}	
for all sufficiently small $L>0$, where $C>0$ is independent of $L$.
The function $g(\xi,y)$ will be expressed in terms of the homogeneous wave profile $\phi_0$, the corrector $\chi_1$ defined by \eqref{corrector-chi1},  the solutions to the cell problems \eqref{corrector-chi2}–\eqref{corrector-eta1} and an additional reaction corrector introduced below. We first prove the solvability of these two cell problems.

\begin{lemma}\label{exisence-chi2-eta1}
	There exist unique zero-mean functions $\eta_1, \chi_2 \in C^{ 2}(\T^N)$ satisfying, respectively, \eqref{corrector-eta1} and \eqref{corrector-chi2}. 
\end{lemma}

\begin{proof}
	The proof is standard. As the arguments will be used to handle more complicated cases, we give
	the proof for $\eta_1$ and indicate the corresponding argument for $\chi_2$.
	Consider the space $H^1_0(\mathbb{T}^N)= \{ v \in H^1(\mathbb{T}^N) : \int_{\mathbb{T}^N} v(y)dy = 0 \}$ endowed with the  $H^1$-norm. Define the bilinear form
	$$\mathcal{A}(u,v)= \int_{\mathbb{T}^N} A \nabla u \cdot \nabla vdy\quad\hbox{for }\,\, u,\,v\in H^1_0(\mathbb{T}^N),$$
	and the linear functional
	$$\ell(v)= -\int_{\mathbb{T}^N} ( \langle B \rangle_{A}\cdot e - B \cdot (e+\nabla \chi_1) ) vdy\quad\hbox{for }\,\, v\in H^1_0(\mathbb{T}^N).$$
	By the uniform ellipticity of $A$ and the Poincar\'e inequality, $\mathcal{A}$ is continuous and coercive on $H^1_0(\mathbb{T}^N)$. Moreover, since the right-hand side of \eqref{corrector-eta1} belongs to $L^2(\T^N)$, $\ell$ is a bounded linear functional on $H^1_0(\mathbb{T}^N)$.
	Then, the Lax--Milgram theorem yields a unique $\eta_1\in H^1_0(\T^N)$ such that
	$\mathcal{A}(\eta_1,v) =\ell(v)$ for all  $v\in H^1_0(\mathbb{T}^N)$. 
 Furthermore, by Lemma \ref{lem-mean-B}, 
$$\int_{\T^N}\left(\langle B\rangle_A\cdot e -B\cdot(e+\nabla\chi_1)\right)dy=0.$$
Subtracting the mean from an arbitrary test function in $H^1(\T^N)$ shows that the weak formulation above holds for all $v\in H^1(\T^N)$, and hence $\eta_1$ is a weak solution of \eqref{corrector-eta1}. Finally, the regularity assumptions on $A$ and $B$, together with elliptic regularity on the torus, imply that $\eta_1\in C^2(\T^N)$.

For equation \eqref{corrector-chi2}, the definition of $A^{\rm hom}$ and the periodicity give
$$ \int_{\T^N} \left(A^{\rm hom}e\cdot e -A(e+\nabla\chi_1)\cdot e -\nabla\cdot(Ae\chi_1)\right)dy=0.$$
The same argument therefore gives a unique zero-mean solution $\chi_2\in C^2(\T^N)$.
\end{proof}	

To account for the spatial oscillations of the reaction term, we introduce an additional corrector $\psi$.
For each $\xi\in\R$, consider the cell problem
\begin{equation}\label{solu-eata}
	\nabla_y \cdot (A(y)\nabla_y \psi) = \bar{f}(\phi_0(\xi)) - f(y,\phi_0(\xi)) \quad\hbox{for }\; y\in \T^N,\quad \hbox{and}\quad \int_{\T^N} \psi(\xi,y)dy=0.
\end{equation} 
The following lemma gives its solvability and the regularity needed to prove Theorem \ref{theo-convergence}.

\begin{lemma}\label{regularity-eta}
	For each $\xi\in\R$, problem \eqref{solu-eata} admits a unique solution $\psi(\xi,\cdot)\in C^{ 2}( \T^N)$. Moreover, $\psi\in C^{ 2}(\R \times \T^N)\cap H^{ 2}(\R\times\T^N)$. 
\end{lemma}

\begin{proof}
Set
$$h(\xi,y)=\bar f(\phi_0(\xi))-f(y,\phi_0(\xi))\quad\hbox{for }\;(\xi,y)\in\R\times\T^N.$$
Since $\phi_0\in C^5(\R)$ and $f\in C^3(\T^N\times[0,1])$, we have
$h\in C^3(\R\times\T^N)$. By the definition of $\bar f$, there holds
\begin{equation}\label{zero-meanhk}
\int_{\T^N}\partial_\xi^k h(\xi,y)dy=0 \quad\hbox{for all }\;\xi\in\R,\,\, 0\le k\le2.
\end{equation}
We first claim that
\begin{equation}\label{L2-hk}
\partial_\xi^k h\in L^2(\R\times\T^N) \quad\hbox{for }\;0\le k\le2.
\end{equation}
Indeed, since $f(y,0)=f(y,1)=\bar f(0)=\bar f(1)=0$, together with the boundedness of $\partial_u f$, $ h\in L^2(\R\times\T^N)$ follows from a similar argument in deriving \eqref{f-L2}. 
Moreover, since 
$\partial_\xi h=(\bar f'(\phi_0)-\partial_u f(y,\phi_0))\phi_0'$, and
$\partial_\xi^2h=(\bar f''(\phi_0)-\partial_u^2f(y,\phi_0))(\phi_0')^2
+(\bar f'(\phi_0)-\partial_u f(y,\phi_0))\phi_0''$, the exponential estimates in the proof of Lemma \ref{continuity-homo-wave} yield that $\partial_\xi^1 h(\xi,y)$ and $\partial_\xi^2 h(\xi,y)$ decay exponentially fast as $\xi\to \pm\infty$ uniformly in $y\in\T^N$. This immediately implies \eqref{L2-hk}. 

By \eqref{zero-meanhk}, the argument used in the proof of Lemma \ref{exisence-chi2-eta1} gives, for each
$\xi\in\R$ and $0\le k\le2$, a unique zero-mean weak solution $\psi_k(\xi,\cdot)\in H^1(\T^N)$ of
\begin{equation}\label{equation-etak}
	\nabla_y \cdot (A(y) \nabla_y (\psi_k(\xi, y))) = \partial_\xi^k h(\xi, y) \quad\hbox{for }\; y\in \mathbb{T}^N, \quad \hbox{and}\quad \int_{\T^N} \psi_k(\xi,y)dy=0
\end{equation}
Here, $\psi_0=\psi$. In particular, elliptic regularity gives $\psi(\xi,\cdot)\in C^2(\T^N)$. 
Since $A$ is independent of $\xi$, by uniqueness and the $C^2$ dependence of $h(\xi,\cdot)$ on $\xi$, we have
\begin{equation}\label{etak-partial-eta}
\psi_k(\xi,\cdot)=\partial_\xi^k\psi(\xi,\cdot) \quad\hbox{for }\;1\le k\leq 2.
\end{equation}
Furthermore, standard elliptic regularity gives 
$\psi\in C(\R;C^2(\T^N))$, $\psi_1\in C(\R;C^1(\T^N))$ and $\psi_2\in C(\R;C(\T^N))$.
Together with \eqref{etak-partial-eta}, these properties show that $\psi\in C^2(\R\times\T^N)$.
		
It remains to prove that $\psi\in H^2(\R\times\T^N)$. For each fixed $\xi\in\R$ and $0\le k\le2$, the
zero-mean condition and the Poincar\'e inequality give
\begin{equation}\label{poincare-in}
\|\psi_k(\xi,\cdot)\|_{L^2(\T^N)} \leq C\|\nabla_y\psi_k(\xi,\cdot)\|_{L^2(\T^N)}, 
\end{equation}
where $C>0$ is independent of $\xi$ and $k$.
Testing \eqref{equation-etak} with $\psi_k(\xi,\cdot)$ and using the uniform ellipticity of $A$, we obtain
$$ \eta_0 \|\nabla_y\psi_k(\xi,\cdot)\|_{L^2(\T^N)}^2
\leq -\int_{\T^N} \partial_\xi^k h(\xi,y)\psi_k(\xi,y)dy
\leq C\|\partial_\xi^k h(\xi,\cdot)\|_{L^2(\T^N)}\|\nabla_y\psi_k(\xi,\cdot)\|_{L^2(\T^N)}.$$
This together with \eqref{poincare-in} implies
$\|\psi_k(\xi,\cdot)\|_{H^1(\T^N)} \leq C\|\partial_\xi^k h(\xi,\cdot)\|_{L^2(\T^N)}$.
Integrating these estimates in $\xi$ and using \eqref{L2-hk}, we obtain
\begin{equation}\label{etak-H1norm}
\sum_{k=0}^2\|\psi_k\|_{L^2(\R\times\T^N)}^2+
\sum_{k=0}^1\|\nabla_y\psi_k\|_{L^2(\R\times\T^N)}^2
\leq C\sum_{k=0}^2\|\partial_\xi^k h\|_{L^2(\R\times\T^N)}^2<+\infty.
\end{equation}
Together with \eqref{etak-partial-eta}, this gives $\psi,\partial_\xi\psi\in H^1(\R\times\T^N)$.
Finally, applying the elliptic $H^2$ estimate on $\T^N$ to \eqref{solu-eata}, for each fixed $\xi$, gives
$$ \|\psi(\xi,\cdot)\|_{H^2(\T^N)} \leq C\left(\|h(\xi,\cdot)\|_{L^2(\T^N)}+\|\psi(\xi,\cdot)\|_{H^1(\T^N)}\right)
\leq C\|h(\xi,\cdot)\|_{L^2(\T^N)},$$ 
where $C$ is independent of $\xi$. Therefore,
$\int_{\R}\|\psi(\xi,\cdot)\|_{H^2(\T^N)}^2d\xi \leq C\|h\|_{L^2(\R\times\T^N)}^2 <+\infty$.
Combining this with \eqref{etak-H1norm}, we conclude that $\psi\in H^2(\R\times\T^N)$. This completes the proof.
\end{proof}

We are now in a position to prove Theorem \ref{theo-convergence}.

\begin{proof}[Proof of Theorem {\rm \ref{theo-convergence}}]
Fix the direction $e$ and assume, as above, that $c_0>0$.
We use the normalized waves constructed in the proof of
Theorem \ref{existence small period}, so that
\begin{equation}\label{recall-G10} 
G(v_L,c_L,L,e)=(0,0),\qquad (v_L,c_L)\to(0,c_0) \quad\hbox{in }\;H^1(\R\times\T^N)\times\R
\end{equation}
as $L\to0^+$. Throughout the proof, $C>0$ is independent of small $L>0$ and may vary from line to line.

Define the $L$-independent function
$$ g(\xi,y) =\eta_1(y)\phi_0'(\xi) +\chi_2(y)\phi_0''(\xi)+\psi(\xi,y).$$
By Lemmas \ref{exisence-chi2-eta1} and
\ref{regularity-eta}, together with $\phi_0'\in H^4(\R)$,
we have $g\in H^2(\R\times\T^N)$. In particular, $g\in\mathcal D_{L,e}$ for every fixed
$L\neq0$. For each $0<L<L_*$, set
$$ w_L=G_1(L^2g,c_0,L,e) =L^2g+M_{c_0,L,e}^{-1}\bigl(K(L^2g,c_0,L,e)\bigr).$$
Thus, $w_L\in\mathcal D_{L,e}$, and since $(\phi_0,c_0)$ is the solution of \eqref{homo-wave}, a direct computation shows that $w_L(\xi,y)$ satisfies
	\begin{equation*}
		\begin{aligned}
			M_{c_0,L,e}(w_L)&= L^2 M_{c_0,L,e}(g) + K(L^2g,c_0,L,e) \vspace{5pt}\\
			&= L^2\tilde{\nabla}_{L,e} \cdot (A\tilde{\nabla}_{L,e} g) 
			+ L^2B \cdot \tilde{\nabla}_{L,e} g 
			+ L^2c_0 \partial_\xi g \vspace{5pt}\\
			&\quad + \left((A(e + \nabla \chi_1)\cdot e - A^{{\rm hom}}e \cdot e) +\nabla \cdot (Ae\chi_1)\right)\phi_0'' 
			+ L(A e \cdot e)\chi_1\phi_0^{(3)} \vspace{5pt}\\
			&\quad + \left(B\cdot(e+\nabla\chi_1) - \langle B \rangle_{A}\cdot e\right)\phi_0' + L(B \cdot e) \chi_1\phi_0'' + Lc_0\chi_1\phi_0'' \vspace{5pt}\\
			&\quad + f(y,\phi_0 + L\chi_1\phi_0' + L^2g) - \bar{f}(\phi_0).
		\end{aligned}
\end{equation*}
Since $\eta_1$, $\chi_2$ and $\psi$ are solutions of \eqref{corrector-eta1}, \eqref{corrector-chi2} and \eqref{solu-eata}, respectively,  we further compute that 
	\begin{equation}\label{Mc0L-estimate}
		M_{c_0,L,e}(w_L)=L R_L(\xi,y) + L^2 S(\xi,y), 
	\end{equation}
	where $S=(Ae\cdot e)\partial_\xi^2g +(B\cdot e+c_0)\partial_\xi g$, and
$$\begin{aligned}
R_L={}&\nabla_y\cdot\bigl(Ae\,\partial_\xi g\bigr) +(A\nabla_y\partial_\xi g)\cdot e+B\cdot\nabla_yg\vspace{5pt}\\
&\qquad +\chi_1\bigl((Ae\cdot e)\phi_0^{(3)} +(B\cdot e+c_0)\phi_0''\bigr)+\frac{f(y,\phi_0+L\chi_1\phi_0'+L^2g)-f(y,\phi_0)}{L}.
\end{aligned}$$

Next, we claim that  
\begin{equation}\label{bound-SRL}	
	\|S\|_{L^2(\R\times\T^N)}+\|R_L\|_{L^2(\R\times\T^N)}\leq C.
\end{equation} 
Indeed, $g\in H^2(\R\times\T^N)$ and the regularity of the coefficients control all the linear terms.
For the nonlinear term, set
$$h_L=\chi_1\phi_0'+Lg\quad\hbox{and}\quad \tilde g_L=f(y,\phi_0+Lh_L)-f(y,\phi_0).$$ 
Then, $\|h_L\|_{L^2(\R\times\T^N)}\le C$, and since $\partial_u f$ is globally bounded, we have 
$\|L^{-1}\tilde g_L\|_{L^2(\R\times\T^N)}\leq \|\partial_u f\|_{L^\infty(\T^N\times\R)} \|h_L\|_{L^2(\R\times\T^N)}\leq C$. 
This proves \eqref{bound-SRL}

Now, by \eqref{Mc0L-estimate}, \eqref{bound-SRL} and Lemma \ref{solvability-M}, we have
$$\|M_{c_0,L,e}(w_L)\|_{L^{2}(\R\times \T^N)}\leq CL\quad\hbox{and}\quad\|w_L\|_{H^{ 1}(\R \times \T^N)} \leq CL.$$ 
Consequently, by \eqref{recall-G10} and the definition of $G_1$, we obtain
	\begin{equation}\label{estimate-G1}
		\|G_1(v_L,c_L,L,e) - G_1(L^2g,c_0,L,e)\|_{H^{ 1}(\R \times \T^N)} \leq CL.
	\end{equation}
For the second component $G_2$, the zero-mean conditions on $\chi_1$, $\eta_1$, $\chi_2$ and $\psi$ imply that
$\int_{\T^N}g(\xi,y)dy=0$ and $\int_{\T^N}h_L(\xi,y)dy=0$.
Using the definition of $G_2$ and the fact that $\phi_0$ is independent of $y$, we therefore obtain
$$G_2(L^2g,c_0,L,e) =\int_{\R^+\times\T^N}\bigl(2L\phi_0h_L+L^2h_L^2\bigr)d\xi dy=L^2\int_{\R^+\times\T^N}h_L^2d\xi dy.$$
Since $G_2(v_L,c_L,L,e)=0$, this yields
\begin{equation}\label{estimate-G2}
	|G_2(v_L,c_L,L,e)-G_2(L^2g,c_0,L,e)|\leq CL^2.
\end{equation}

It remains to convert these residual estimates into an estimate for $(v_L-L^2g,c_L-c_0)$. To this end,
we equip $H^1(\R\times\T^N)\times\R$ with the sum norm
and denote the corresponding operator norm by
$\|\cdot\|_{\mathcal L}$. By Lemma \ref{continuity-operotor}, the operator
$Q_0:=D_{(v,c)}G(0,c_0,0,e)$ is invertible, and $D_{(v,c)}G$ is continuous at $(0,c_0,0,e)$ in the operator norm.
For $t\in[0,1]$, define
$$ Q_{L,t} =D_{(v,c)}G\bigl(tv_L+(1-t)L^2g,tc_L+(1-t)c_0,L,e\bigr).$$
Since $(v_L,c_L)\to(0,c_0)$ as $L\to 0^+$, and $g$ is independent of $L$, all points on this line segment converge to
$(0,c_0,0,e)$ uniformly in $t\in[0,1]$. Thus, there exists $L_1\in(0,L_*)$ such that
$$\sup_{0\le t\le1}\|Q_{L,t}-Q_0\|_{\mathcal L} \leq\frac{1}{2\|Q_0^{-1}\|_{\mathcal L}}
\quad\hbox{for all }\;0<L<L_1.$$
Applying the mean value inequality along this segment, we obtain
$$\begin{aligned}
&\Bigl\|G(v_L,c_L,L,e)-G(L^2g,c_0,L,e)-Q_0(v_L-L^2g,c_L-c_0)\Bigr\|_{H^1(\R\times\T^N)\times\R}\\
\leq {}& \frac{\|v_L-L^2g\|_{H^1(\R\times\T^N)}+|c_L-c_0|} {2\|Q_0^{-1}\|_{\mathcal L}}.
\end{aligned}$$
Multiplying by $\|Q_0^{-1}\|_{\mathcal L}$ and absorbing the resulting half of the unknown norm gives 
$$\|v_L-L^2g\|_{H^1(\R\times\T^N)}+|c_L-c_0|\leq 2\|Q_0^{-1}\|_{\mathcal L} \|G(v_L,c_L,L,e)-G(L^2g,c_0,L,e)\|_{H^1(\R\times\T^N)\times\R}.
$$
It then follows from \eqref{estimate-G1} and \eqref{estimate-G2} that
$$\|v_L-L^2g\|_{H^1(\R\times\T^N)}+|c_L-c_0|\leq CL.$$
Finally, by the relation of $v_L$ and $\phi_L$ in \eqref{vL-phiL}, we have
$$
\|\phi_L-\phi_0\|_{H^1(\R\times\T^N)} \leq\|v_L-L^2g\|_{H^1(\R\times\T^N)} +L\|\chi_1\phi_0'\|_{H^1(\R\times\T^N)}
	+L^2\|g\|_{H^1(\R\times\T^N)}\leq CL.
$$
This completes the proof of Theorem \ref{theo-convergence}. 
\end{proof}


\subsection{Proof of Theorem \ref{theo-sharp-speed}}   

We now assume that $f=f(u)$ is independent of $x$, and derive an asymptotic expansion of the wave speed $c_L$ as $L\to0^+$. As explained above, this requires a third-order approximation of the wave profile $\phi_L$ with respect to $L$. To derive this approximation, in addition to the correctors $\chi_1$, $\chi_2$ and $\eta_1$ introduced in \eqref{corrector-chi1}-\eqref{corrector-eta1}, we consider the following cell problems:
\begin{equation}\label{corrector-chi3}
	\nabla \cdot (A \nabla \chi_3)  =d_3-  \nabla \cdot (A e \chi_2) - A \nabla \chi_2 \cdot e - A e \cdot e \chi_1 + A^{{\rm hom}}e \cdot e \chi_1 \,\,	\hbox{ in } \,\, \T^N,  
\end{equation} 
\begin{equation}\label{corrector-eta2} 
	\nabla \cdot (A \nabla \eta_2) = d_2 - \nabla \cdot (A e \eta_1) - A \nabla\eta_1 \cdot e - B \cdot e \chi_1 - B \cdot \nabla \chi_2 + \langle B \rangle_{A}\cdot e \chi_1\,\,	\hbox{ in } \,\, \T^N,  
\end{equation}            
and
\begin{equation}\label{corrector-eta3}
	\nabla \cdot (A \nabla \eta_3)  = d_1- B \cdot \nabla \eta_1\,\,	\hbox{ in } \,\, \T^N,  
\end{equation} 	                                    
with the normalization 
$$\int_{\T^N}\chi_3dy= \int_{\T^N}\eta_2dy=\int_{\T^N}\eta_3dy=0.$$
Here, $d_1$, $d_2$ and $d_3$ are the constants given by \eqref{formula-d123}.
Their definitions, together with the zero-mean condition for $\chi_1$ and periodicity, imply that the right-hand
sides of \eqref{corrector-chi3}--\eqref{corrector-eta3} have zero mean over $\T^N$.
The argument used in the proof of Lemma \ref{exisence-chi2-eta1} therefore gives unique
zero-mean solutions $\chi_3,\eta_2,\eta_3\in C^2(\T^N)$.

The following lemma simplifies the coefficients $d_1,d_2,d_3$ when $A$ is symmetric.

\begin{lemma}\label{lem-D012}
Assume that the diffusion matrix $A$ is symmetric, and let $\tilde{d}_1$, $\tilde{d}_2$ be the constants defined by \eqref{simple-tildec12}. Then, $d_1=\tilde{d}_1$, $d_2=\tilde{d}_2$ and $d_3=0$.   
\end{lemma}

\begin{proof}
We first prove $d_3=0$. By the definition of $\chi_1$, $\nabla \cdot (A \nabla \chi_1) =- \nabla \cdot (Ae)$ in $\T^N$. Then, since $A$ is symmetric, integrating by parts gives  
\begin{equation*}
	\begin{aligned}	\int_{\T^N}A \nabla \chi_2 \cdot edy=\int_{\T^N}\nabla \chi_2 \cdot (A e) dy&= - \int_{\T^N} \nabla \cdot (Ae) \chi_2dy=\int_{\T^N} \nabla \cdot (A \nabla \chi_1) \chi_2dy\\
		&=-\int_{\T^N} A \nabla \chi_1 \cdot \nabla \chi_2dy= \int_{\T^N} \nabla \cdot (A \nabla \chi_2) \chi_1dy. 
	\end{aligned}
\end{equation*}
Using the equation for $\chi_2$ in \eqref{corrector-chi2} and the fact that $\chi_1$ has zero mean, we further get 
\begin{equation*}
	\begin{aligned}
		\int_{\T^N}A \nabla \chi_2 \cdot e dy
		&= \int_{\T^N} \left(A^{{\rm hom}}e \cdot e \chi_1 - A(\nabla \chi_1 + e) \cdot e \chi_1 - \nabla \cdot (Ae\chi_1) \chi_1\right)dy \\
		&= \int_{\T^N} \left(- A \nabla \chi_1 \cdot e \chi_1 - A e \cdot e \chi_1 - \nabla \cdot (Ae\chi_1) \chi_1 \right)dy \\
		&= \int_{\T^N} \left(- A\nabla \chi_1 \cdot e \chi_1 - Ae \cdot e \chi_1 + A e \chi_1 \cdot \nabla \chi_1\right)dy . 
	\end{aligned}
\end{equation*}
By the symmetry of $A$ again, we have $A\nabla \chi_1 \cdot e \chi_1=A e \chi_1 \cdot \nabla \chi_1$. By the definition of $d_3$ in \eqref{formula-d123}, this proves $d_3=0$. 

Similarly, by the symmetry of $A$, integrating by parts and using the equation of $\chi_1$, we get
\begin{equation*}
	\begin{aligned}
		\int_{\T^N} A \nabla \eta_1 \cdot e dy
		= -\int_{\T^N} \nabla \cdot (Ae) \eta_1 dy
		&= \int_{\T^N}\nabla \cdot (A \nabla \chi_1) \eta_1 dy\\
		&=-\int_{\T^N} (A \nabla \chi_1) \cdot \nabla \eta_1 dy= \int_{\T^N}\nabla \cdot (A \nabla \eta_1) \chi_1 dy.
	\end{aligned}
\end{equation*}
By the equation for $\eta_1$ in \eqref{corrector-eta1}, it then follows that
$$	\int_{\T^N} A \nabla \eta_1 \cdot e dy= -\int_{\T^N}B \cdot (\nabla \chi_1 + e) \chi_1 dy.$$
This, together with  the definition of $\tilde{d}_2$, gives that
$$d_2=\int_{\T^N} \left(A \nabla \eta_1 \cdot e + B \cdot e \chi_1 + B \cdot \nabla \chi_2\right)dy=\int_{\T^N} \left(-(B\cdot \nabla\chi_1) \chi_1+ B\cdot \nabla \chi_2   \right)dy=\tilde{d}_2. $$

Finally, using the equation satisfied by $\zeta$ in \eqref{Bcorrector} and the equation satisfied by $\eta_1$ in \eqref{corrector-eta1}, together with the symmetry of $A$, we obtain
\begin{equation*}
	\begin{aligned}
		d_1=	\int_{\T^N}  B \cdot \nabla \eta_1dy &= -\int_{\T^N} \left(\nabla \cdot B\right) \eta_1 dy = -\int_{\T^N} \nabla \cdot (A \nabla \zeta) \eta_1 dy= \int_{\T^N}  (A \nabla \zeta) \cdot \nabla \eta_1 dy  \\
		& = -\int_{\T^N} \nabla \cdot (A \nabla \eta_1) \zeta dy= \int_{\T^N} B \cdot (\nabla \chi_1 + e) \zeta dy=\tilde{d}_1.
	\end{aligned}
\end{equation*}
The proof of Lemma \ref{lem-D012} is thus complete.
\end{proof}

Next, we turn to the linear operator 
$H: H^2(\R)\to L^2(\R)$ defined in \eqref{define-He}, and its $L^2$-adjoint $H^*$. By Lemma \ref{properties of H},  ${\rm Ker}(H)=\R\phi_0'$, and 
${\rm Ker}(H^*)=\R w$, where $w\in H^2(\R)$ is given by \eqref{adjoint-kernel}.
Moreover, the range of $H$ is closed in $L^2(\R)$.
Thus, $H$ is a Fredholm operator of index zero with range $R(H)=\{h\in L^2(\R): \int_{\R}h(\xi)w(\xi) d\xi=0\}$. 
By the definition of $c_1$ in \eqref{solvability condition-1},
we have
$$\int_{\R} \left(d_3\phi_0'''(\xi)+d_2\phi_0''(\xi)+(d_1+c_1)\phi_0'(\xi)\right)
w(\xi)d\xi=0.$$
The Fredholm alternative therefore yields the following
solvability result.

\begin{lemma}\label{solvable-varphi}
The following equation 
\begin{equation}\label{solvability equation}
H(\varphi)=-\left(d_3\phi_0'''+d_2\phi_0'' + d_1 \phi_0' + c_1\phi_0' \right)\quad\hbox{in }\,\,\R 
\end{equation}
is solvable in $H^2(\R)$. Furthermore, every solution in $H^2(\R)$ can be written
uniquely as
\begin{equation}\label{determine-varphi}
	\varphi=\varphi_0+\mu_0\phi_0',
\end{equation}
for some $\mu_0\in\R$, where $\varphi_0\in H^2(\R)$ is the unique solution of
\eqref{solvability equation} satisfying
$\int_{\R}\varphi_0(\xi)w(\xi) d\xi=0$.
\end{lemma}	

\begin{proof}
The existence follows from the orthogonality condition above
and the Fredholm alternative. Since ${\rm Ker}(H)=\R\phi_0'$, any two solutions of \eqref{solvability equation}
differ by a multiple of $\phi_0'$. Furthermore, \eqref{adjoint-kernel} gives 
$\int_{\R}\phi_0'(\xi)w(\xi) d\xi>0$. Hence, the orthogonality condition uniquely determines $\varphi_0$, and every solution admits the unique representation \eqref{determine-varphi}.
\end{proof}

Finally, every $H^2(\R)$ solution $\varphi$ of \eqref{solvability equation} belongs to $C^4(\R)$ and, together with its derivatives up to order $4$, decays exponentially as $\xi\to\pm\infty$.
Indeed, since $f\in C^3([0,1])$ and $\phi_0^{(k)}$ decays exponentially for $1\le k\leq 5$,
the right-hand side of \eqref{solvability equation} is of class $C^2$ and decays exponentially together
with its derivatives. The equation therefore implies that $\varphi\in C^4(\R)$.
Furthermore, $\varphi\in H^2(\R)$ implies that $\varphi(\xi)\to0$ as $\xi\to\pm\infty$, while the
zeroth-order coefficient $f'(\phi_0)$ of $H$ approaches $f'(1)<0$ as $\xi\to-\infty$ and $f'(0)<0$ as
$\xi\to+\infty$. Comparison with exponential barriers on the two half-lines gives exponential decay of $\varphi$.
Local estimates for \eqref{solvability equation} then give exponential decay of $\varphi'$ and
$\varphi''$, and differentiating the equation twice gives the same property for $\varphi^{(3)}$ and $\varphi^{(4)}$.

Based on the above preparations, we are now ready to prove Theorem  \ref{theo-sharp-speed}.
                                                                                          
\begin{proof}[Proof of Theorem {\rm \ref{theo-sharp-speed}}] 
As in the preceding proof, we assume $c_0>0$. For each small $L>0$, set
\begin{equation*}
p_L(\xi,y) = \phi_0(\xi) + L\phi_1(\xi,y) + L^2\phi_2(\xi,y) + L^3\phi_3(\xi,y),
\end{equation*}
where
\begin{equation*}\left\{
		\begin{aligned}
			\phi_1(\xi,y) &= \chi_1(y)\phi_0'(\xi) + \varphi(\xi), \vspace{5pt}\\
			\phi_2(\xi,y) &= \chi_2(y)\phi_0''(\xi) + \eta_1(y)\phi_0'(\xi) + \chi_1(y)\varphi'(\xi), \vspace{5pt}\\
			\phi_3(\xi,y) &= \chi_3(y)\phi_0'''(\xi) + \eta_2(y)\phi_0''(\xi) + \eta_3(y)\phi_0'(\xi) + \chi_2(y)\varphi''(\xi) + \eta_1(y)\varphi'(\xi). 
		\end{aligned}\right.
	\end{equation*}
Here, $\varphi=\varphi_0+\mu_0\phi_0'$ is the solution in \eqref{determine-varphi}, with
\begin{equation}\label{choose-mu0}
	\mu_0= 2\phi_0^{-2}(0) \int_{\R^+} \varphi_0(\xi) \phi_0(\xi)d\xi. 
\end{equation} 
This choice of $\mu_0$ gives
$$
\int_{\R^+}\phi_0\varphi d\xi =\int_{\R^+}\phi_0\varphi_0 d\xi+\mu_0\int_{\R^+}\phi_0\phi_0'\,d\xi 
=\int_{\R^+}\phi_0\varphi_0 d\xi-\frac{\mu_0}{2}\phi_0(0)^2=0.$$
Moreover, since both $\phi_0'$ and $\varphi$ belong to $H^4(\R)$, we have $p_L - \phi_0 - L\chi_1 \phi_0'\in H^2(\R\times\T^N)\subset \mathcal D_{L,e}$ for every fixed $L\neq0$.

In the following differential identities, $O(L^2)$ denotes a function in $L^2(\R\times\T^N)$ whose norm
is bounded by $CL^2$ with $C>0$ independent of $L$.  Using the equations for $\chi_1$, $\chi_2$ and $\chi_3$, a direct computation yields that
\begin{equation*}
\begin{aligned}
\tilde{\nabla}_L \cdot (A \tilde{\nabla}_L p_L) =& (A^{{\rm hom}}e \cdot e)\phi_0''+\nabla_y\cdot(A\nabla_y\eta_1)\phi_0'
+L(d_3+(A^{{\rm hom}}e \cdot e)\chi_1)\phi_0''' \\
&\quad +L\left( \nabla_y \cdot (A \nabla_y \eta_2)+  \nabla_y\cdot (A e \eta_1) + A \nabla_y\eta_1 \cdot e \right)\phi_0'' \\
&\quad + L\nabla_y\cdot(A\nabla_y\eta_3)\phi_0' +L(A^{{\rm hom}}e \cdot e)\varphi''
+L\nabla_y \cdot (A \nabla_y \eta_1)\varphi'+O(L^2). 
\end{aligned}
\end{equation*}
It is also straightforward to compute that
\begin{equation*}
	\begin{aligned}
B \cdot\tilde{\nabla}_L p_L = B\cdot(\nabla_y\chi_1+ e) \phi_0'&+LB\cdot(\nabla_y\chi_2+\chi_1e )\phi_0''\\
&+ LB\cdot \nabla_y\eta_1\phi_0' +LB\cdot\left(\nabla_y\chi_1+e\right)\varphi'+O(L^2).
\end{aligned}
\end{equation*}
Since $f=f(u)$ is independent of $x$, we have $\bar{f}=f$.  
Differentiating equation \eqref{homo-wave} gives 
$$(A^{{\rm hom}}e \cdot e)\phi_0'''+(\langle B \rangle_{A}\cdot e+c_0)\phi_0''+f'(\phi_0)\phi_0'=0\quad\hbox{in }\,\,\R.$$ 
Then, by Taylor's expansion, 
\begin{equation*}
\begin{aligned}
f(p_L)-f(\phi_0)=&Lf'(\phi_0)\phi'_0\chi_1+Lf'(\phi_0)\varphi+O(L^2) \\
=& -L\left( (A^{{\rm hom}}e \cdot e)\phi_0'''+(\langle B \rangle_{A}\cdot e+c_0)\phi_0'' \right)\chi_1+Lf'(\phi_0)\varphi+O(L^2).
\end{aligned}
\end{equation*}
Combining the above and using the equations for $\eta_1$, $\eta_2$ and $\eta_3$, we obtain
\begin{equation*}
\begin{aligned}
&\quad\tilde{\nabla} \cdot (A(y) \tilde{\nabla} p_L) + B(y) \cdot \tilde{\nabla} p_L + f(p_L) \\
&= -c_0 \phi_0'+L\left(d_3\phi_0'''+d_2\phi_0''+d_1\phi_0'+H(\varphi)-c_0\partial_\xi\phi_1\right)+O(L^2).
\end{aligned}
\end{equation*}
By Lemma \ref{solvable-varphi}, we have 
$ H(\varphi)=-(d_3\phi_0'''+d_2\phi_0''+(d_1+c_1)\phi_0')$, 
where $c_1$ is the constant defined in \eqref{solvability condition-1}.
Hence, 
\begin{equation}\label{simple-eq-pL}
\tilde{\nabla} \cdot (A(y) \tilde{\nabla} p_L) + B(y) \cdot \tilde{\nabla} p_L + f(p_L) 
= -c_0 \phi_0'- L\left( c_0 \partial_{\xi}\phi_1+c_1\phi_0'\right) + O(L^2).
\end{equation}

For sufficiently small $L>0$, we have $c_0+Lc_1\ge c_0/2>0$. Recall that $v_{L}\in H^{ 1}(\R\times \T^N)$ is the function arising from the implicit function argument in the proof of Theorem \ref{existence small period}, and that $(G_1,G_2)(v_L,c_L,L,e) = (0,0)$, where the function $(G_1,G_2)$ is defined in \eqref{defi-G}. 
Define 
\begin{equation*}
	w_L(\xi,y) = G_1(p_L - \phi_0 - L\chi_1 \phi_0', c_0 + Lc_1,L,e). 
\end{equation*}  
Since $p_L-\phi_0-L\chi_1\phi_0'\in H^2(\R\times\T^N)\subset\mathcal D_{L,e}$, we also have $w_L\in\mathcal D_{L,e}$.
By the definitions of $G_1$, $K$ and $M_{c,L,e}$, we compute that    
\begin{equation*}
	\begin{aligned}
		M_{c_0 + Lc_1, L ,e}(w_L)
		&= 	M_{c_0 + Lc_1, L ,e}(p_L - \phi_0 - L\chi_1 \phi_0')+K(p_L - \phi_0 - L\chi_1 \phi_0',c_0+Lc_1,L,e) \\
		&= \tilde{\nabla}_{L} \cdot (A \tilde{\nabla}_{L}p_L)
		+ B \cdot \tilde{\nabla}_{L}p_L	+ (c_0 + Lc_1)\partial_\xi p_L
		+ f(p_L).
	\end{aligned}
\end{equation*}
It further follows from \eqref{simple-eq-pL} that
$$M_{c_0 + Lc_1, L ,e}(w_L)= (c_0 + Lc_1)\partial_\xi (p_L - \phi_0)
-Lc_0 \partial_\xi \phi_1 + O(L^2)= O(L^2). $$
Using the energy estimate in Lemma \ref{solvability-M}, we obtain $\|w_L\|_{H^1(\R \times \T^N)} \leq CL^2$, and hence, 
\begin{equation}\label{estimate-G1+}
\|G_1(p_L - \phi_0 - L\chi_1 \phi_0', c_0 + Lc_1,L,e)-G_1(v_L,c_L,L,e)\|_{H^1(\R\times \T^N)} \leq CL^2.
\end{equation}
On the other hand, by the definition of $G_2$, and the fact that $G_2(v_L,c_L,L,e) =0$, we have
\begin{equation*}
	\begin{aligned}
		\left|G_2(p_L-\phi_0-L\chi_1\phi_0',c_0+Lc_1,L,e) - G_2(v_L,c_L,L,e)\right|=&\int_{\R^+ \times \T^N} \left(p_L^2 - \phi_0^2\right) d\xi  dy \\
	 =& \int_{\R^+ \times \T^N} 
	 \left(2L\phi_1\phi_0 + O(L^2)\right)d\xi dy.
	\end{aligned}
\end{equation*}
The choice of $\mu_0$ in \eqref{choose-mu0} then yields 
\begin{equation}\label{estimate-G2+}
|G_2(p_L-\phi_0-L\chi_1\phi_0',c_0+Lc_1,L,e) - G_2(v_L,c_L,L,e)| \leq CL^2. 
\end{equation} 
As a consequence,  both $(v_L,c_L)$ and $(p_L-\phi_0-L\chi_1\phi_0',c_0+Lc_1)$
converge to $(0,c_0)$ in $H^1(\R\times\T^N)\times\R$ as $L\to0^+$.
By Lemma \ref{continuity-operotor}, the quantitative inverse estimate proved at the end of the proof of
Theorem \ref{theo-convergence} applies to these two
points. Using \eqref{estimate-G1+} and
\eqref{estimate-G2+}, we obtain
\begin{equation*}
	\|p_L - \phi_0 - L\chi_1 \phi_0'- v_L\|_{L^2(\R \times \T^N)} + |c_L - c_0 - Lc_1| \leq CL^2.
\end{equation*}
This immediately gives the asymptotic expansion $c_L=c_0+c_1L+O(L^2)$ with $c_1$ given by \eqref{solvability condition-1}. 

Finally, in the case where $A$ is symmetric, Lemma \ref{lem-D012} implies 
$$ c_1=-  \tilde{d}_1
	- \tilde{d}_2 
	\int_{\R}\phi_0''(\xi)w(\xi)d\xi\left(\int_{\R}\phi_0'(\xi)w(\xi)d\xi\right)^{-1}. $$
Integration by parts gives
$$
\int_\R\phi_0''(\xi)w(\xi)d\xi=-\frac{\langle B\rangle_A\cdot e+c_0}{2A^{\rm hom}e\cdot e}\int_\R\phi_0'wd\xi, $$
where the boundary term vanishes by the exponential decay of $\phi_0'$ and $w$. Consequently, 
we obtain the simplified formula for $c_1$ in \eqref{solvability conditio2}. This completes the proof of Theorem \ref{theo-sharp-speed}.
\end{proof}


\SE{Uniqueness of pulsating waves: Proof of Theorem \ref{theo-unique}}

Throughout this section, $L>0$ and $e\in\mathbb S^{N-1}$ are fixed.
After the change of variables $x=Ly$, we may take $L=1$, replacing
$A$ and $B$  by $L^{-2}A$ and $L^{-1}B$, respectively. We keep the same notation for the rescaled quantities
and suppress the dependence on $L$ and $e$.
Let $U(t,x)=\phi(x\cdot e-ct,x)$ be the pulsating wave supplied by {\rm (H2)}, and write
$\tilde\nabla=e\partial_\xi+\nabla_y$. In particular, $(\phi,c)$ is a classical solution of the profile equation
\eqref{front equation}. Consequently,
$\phi(x\cdot e-ct+\tau,x)$ is an entire solution of
\eqref{equation} for every $\tau\in\R$, including when $c=0$.

The proof of Theorem \ref{theo-unique} is divided into two cases, according to whether a moving or standing pulsating wave of \eqref{equation} exists. In both cases, we use the monotonicity property of moving waves. 

\begin{lemma}\label{mono-lem}
Assume that $\tilde{U}(t,x)=\tilde{\phi}(x\cdot e-\tilde{c}t, x)$ is a pulsating wave of \eqref{equation} with $\tilde{c}\neq 0$. Then, $\partial_{\xi}\tilde{\phi}(\xi,y)<0$ in $\R\times \T^N$. 
\end{lemma}
\begin{proof}
Since $\tilde{\phi}(x\cdot e-\tilde{c}t, x)$ is a moving pulsating wave, a standard sliding argument yields 
the strict monotonicity of $\tilde{\phi}$ in its first variable (see e.g. \cite{x92,x2}). Alternatively, this strict monotonicity follows from an application of  \cite[Theorem 1.11]{bh12}.
\end{proof}

In the following lemma, we assume the existence of a moving pulsating wave of \eqref{equation} in the direction $e$, and show that the pulsating wave in condition (H2) is identically equal to this wave up to a time shift. 

\begin{lemma}\label{nonzerouniqueness} 
Assume that \eqref{equation} admits a moving pulsating wave $\tilde{U}(t,x)=\tilde{\phi}(x\cdot e-\tilde{c}t, x)$, i.e., $\tilde{c}\neq 0$. Then $\tilde{c}= c$, and there exists $\tau_0\in\R$ such that 
$ \tilde{\phi}(\xi,y) \equiv \phi(\xi+ \tau_0,y)$ in $\R \times \T^N$.
\end{lemma}

\begin{proof}
Without loss of generality, we assume that $c \geq \tilde{c}$ and prove that $c = \tilde{c}$.  The case $c \leq \tilde{c}$ can be treated analogously and will be briefly discussed at the end of the proof. 

By the definition of pulsating waves, there exists a large constant $N_0>0$ such that 
\begin{equation}\label{phi-delta0}
\min\{\phi,\, \tilde{\phi} \} > 1-\delta_0\, \hbox{ in }\, (-\infty,-N_0) \times \T^N,  
\quad \hbox{ and } \quad 
\max\{\phi,\, \tilde{\phi} \} < \delta_0 \,\hbox{ in } \,(N_0,+\infty) \times \T^N,
\end{equation}
 where $\delta_0\in (0,1/2)$ is given by \eqref{strong stability}.
Let $\tau > 0$ be sufficiently large such that
$\phi(\cdot,\cdot) \leq  \tilde{\phi}(\cdot-\tau,\cdot)$ in $[-N_0,N_0] \times \T^N$.
We first claim that 
\begin{equation}\label{comp-phi-tilde}
\phi(\cdot,\cdot) \leq  \tilde{\phi}(\cdot-\tau,\cdot)\, \,\hbox{ in }\,\, \R \times \T^N.
\end{equation}
In fact, assume by contradiction that there exists $(\xi_0,y_0) \in [-N_0,N_0]^c \times \T^N$ such that \begin{equation}\label{sup-phi-tilde}
	\phi(\xi_0,y_0) - \tilde{\phi}(\xi_0-\tau,y_0) = \sup_{\R \times \T^N} \left(\phi(\cdot,\cdot) - \tilde{\phi}(\cdot-\tau,\cdot)\right) > 0.
\end{equation}
Then, $\partial_{\xi} (\phi(\xi, y) - \tilde{\phi}(\xi-\tau,y))|_{(\xi_0, y_0)}=0$, $\nabla_y (\phi(\xi, y) - \tilde{\phi}(\xi-\tau,y))|_{(\xi_0, y_0)}=0$, and the Hessian matrix $D^2_{(\xi,y)} (\phi (\xi,y) - \tilde{\phi}(\xi - \tau, y))|_{(\xi_0, y_0)}$ is negative semi-definite. 
Since $A(\cdot)$ satisfies \eqref{uniform-elliptic}, the above yields 	 
\begin{equation*}
	\begin{aligned}
		&\tilde{\nabla}\cdot
		(A(y)\tilde{\nabla}\phi)\big|_{(\xi_0,y_0)}
		-
		\tilde{\nabla}\cdot
		(A(y)\tilde{\nabla}\tilde{\phi})
		\big|_{(\xi_0-\tau,y_0)}
		\\
		&\quad
		={\rm tr}\left(
		M_e^T A(y_0)M_e
		D^2_{(\xi,y)}
		(\phi(\xi,y)-\tilde{\phi}(\xi-\tau,y))
		\right)\big|_{(\xi_0,y_0)}
		\leq 0,
	\end{aligned}
\end{equation*}
where $M_e$ is an $N\times(N+1)$ matrix given by
$M_e=(e,I_N)$. Hence, 
\begin{equation*}
	\tilde{\nabla}\cdot
	(A(y)\tilde{\nabla}\phi)\big|_{(\xi_0,y_0)}
	\leq
	\tilde{\nabla}\cdot
(A(y)\tilde{\nabla}\tilde{\phi})
	\big|_{(\xi_0-\tau,y_0)}.
\end{equation*}
Moreover, since $\partial_u f(y,u) < 0$ whenever $u \in (0,\delta_0) \cup (1-\delta_0,1)$, 
by \eqref{phi-delta0} and \eqref{sup-phi-tilde}, we have 
$f(y_0,\phi(\xi_0,y_0))< f(y_0,\tilde{\phi}(\xi_0-\tau,y_0))$. Furthermore, by condition (H2), $(\phi,c)$ is a classical solution of \eqref{front equation}. Since we are considering the case $\tilde{c}\neq 0$, the pair $(\tilde{\phi},\tilde{c})$ is also a classical solution of \eqref{front equation}. Therefore, we compute that
\begin{equation}\label{sliding}
\begin{aligned}
0 &= \tilde{\nabla} \cdot (A(y) \tilde{\nabla} \phi) 
+ B(y) \cdot \tilde{\nabla} \phi 
+ c \partial_\xi \phi
+ f(y,\phi) \big|_{(\xi_0,y_0)} \\
&< \tilde{\nabla} \cdot (A(y) \tilde{\nabla} \tilde{\phi}) 
+ B(y) \cdot \tilde{\nabla} \tilde{\phi} 
+ c \partial_\xi \tilde{\phi}
+ f(y,\tilde{\phi}) \big|_{(\xi_0-\tau,y_0)} =(c-\tilde{c})\partial_\xi \tilde{\phi}(\xi_0-\tau,y_0).
\end{aligned}
\end{equation}
This gives $(c-\tilde{c})\partial_\xi \tilde{\phi}(\xi_0-\tau,y_0)>0$, which is impossible, since $\partial_\xi \tilde{\phi}<0$ in $\R\times \T^N$ by Lemma \ref{mono-lem}, and we have assumed that $c\geq \tilde{c}$. Therefore, \eqref{comp-phi-tilde} is proved. 

Define now 
$$\tau_0 = \inf \left\{\tau:\ \phi(\xi,y) \leq \tilde{\phi}(\xi-\tau,y) \hbox{ for all } (\xi,y) \in \R \times \T^N\right\}.$$
Clearly, $\tau_0\leq \tau$, and $\tau_0>-\infty$ because $\phi(\xi,y)>0=\tilde{\phi}(+\infty,y)$. We claim that 
there exists $(\xi_*,y_*)\in\R\times\T^N$ such that 
\begin{equation}\label{phi-phit*}
\phi(\xi_* ,y_*) = \tilde{\phi}(\xi_*-\tau_0,y_*).
\end{equation}
Otherwise, by the continuity of $\phi$ and $\tilde{\phi}$, one could slightly decrease $\tau_0$ to preserve the inequality $\phi(\xi,y) \leq \tilde{\phi}(\xi- \tau_0,y)$ in $[-N_0,N_0] \times \T^N$, and hence in the whole space by the same sliding argument used above, contradicting the definition of $\tau_0$. 
Consequently, \eqref{phi-phit*} holds, and we have   
\begin{equation}\label{phi-phit-leq}
\phi(\xi,y) \leq \tilde{\phi}(\xi-\tau_0,y)\,\,\hbox{ in }\,\,\R\times\T^N.
\end{equation} 
Thanks to \eqref{phi-phit*} and \eqref{phi-phit-leq}, a similar computation to \eqref{sliding} yields that $(c-\tilde{c})\partial_\xi \tilde{\phi}(\xi_*-\tau_0,y_*)\geq 0$. Since $\partial_{\xi}\tilde{\phi}<0$ in $\R\times\T^N$, 
this implies $c=\tilde{c}$.

Recall that we have assumed $\tilde{c}\neq 0$. Hence, $c\neq 0$ as well. A direct verification shows that the following two functions
$$u(t,x):=U(t-(\xi_*-y_*\cdot e)/c,x)= \phi(x\cdot e +\xi_*-y_*\cdot e -ct,x)$$ 
and
 $$\tilde{u}(t,x):=\tilde{U}(t-(\xi_*-y_*\cdot e-\tau_0)/\tilde{c},x) =\tilde{\phi}(x\cdot e+\xi_*-y_*\cdot e-\tau_0-\tilde{c}t,x) $$
are entire solutions of \eqref{equation}. Moreover, by \eqref{phi-phit*}, \eqref{phi-phit-leq} and $c=\tilde{c}$, we have 
\begin{equation*}
	u(0,y_*)=\tilde{u}(0,y_*),\quad \hbox{and}\quad u(t,x) \leq \tilde{u}(t,x) \, \hbox{ for }\, (t,x)\in\R\times\R^N.
\end{equation*}
The parabolic strong maximum principle therefore implies $u(\cdot,\cdot)\equiv \tilde{u}(\cdot,\cdot)$ in $\R\times\R^N$, and hence, we obtain  $\phi(\cdot,\cdot) \equiv \tilde{\phi}(\cdot - \tau_0,\cdot)$ in $\R \times \T^N$.

Finally, the case $c \leq \tilde{c}$ can be treated similarly by starting from sufficiently large negative $\tau$
such that
$\phi(\cdot,\cdot) \geq  \tilde{\phi}(\cdot-\tau,\cdot)$ in $\R \times \T^N$, 
and then sliding $\tau$ up to its supremum  $\tilde{\tau}_0$, which also gives $c=\tilde{c}$ and
$\phi(\cdot,\cdot) \equiv \tilde{\phi}(\cdot - \tilde{\tau}_0,\cdot)$. The proof of Lemma \ref{nonzerouniqueness} is thus complete. 
\end{proof}

Next, we consider the case where equation \eqref{equation} admits a standing pulsating wave in the direction $e$. Since the profile of such a standing wave may lack regularity and therefore may not be a classical solution of \eqref{front equation}, the argument used in the previous lemma cannot be applied. In the following lemma, we develop a different sliding argument to prove the uniqueness.

\begin{lemma}\label{zerouniqueness}
Assume that \eqref{equation} admits a standing pulsating wave $\tilde{U}(x)=\tilde{\phi}(x\cdot e, x)$. Then $c=0$. Furthermore, if $\tilde{\phi}\in C(\R\times\T^N)$, then there exists $\tau_0\in\R$ such that 
$\tilde{U}(x) \equiv \phi(x\cdot e + \tau_0,x)$ in $\R^N$.
\end{lemma}

\begin{proof}
We first introduce some notations. Choose a small constant $\varepsilon_0$ satisfying 
\begin{equation}\label{choose-varepsilon}
0<\varepsilon_0 < \min\left\{\delta_0/2,\,\gamma_0\right\},   
\end{equation}
where $\delta_0\in (0,1/2)$ and $\gamma_0>0$ are the constants provided by \eqref{strong stability}.  For any $\varepsilon \in (0, \varepsilon_0/4)$ and $\tau\in\R$, define
\begin{equation}\label{define-underU}
	\underline{U}^{\varepsilon,\tau}(t,x)=\phi(x \cdot e - ct + \tau,x)-  \varepsilon \,\,\hbox{ for }\,\, (t,x)\in \R\times\R^N.
\end{equation}
We claim that $\underline{U}^{\varepsilon,\tau}(t,x)$ satisfies
\begin{equation}\label{under-u-sub}
	\begin{aligned}
		\partial_t \underline{U}^{\varepsilon,\tau} \leq \nabla \cdot (A(x) \nabla \underline{U}^{\varepsilon,\tau}) 
		+ B(x) \cdot \nabla \underline{U}^{\varepsilon,\tau} &+ f(x,\underline{U}^{\varepsilon,\tau}) - \varepsilon^2 \\
		& \hbox{whenever }\,\, \underline{U}^{\varepsilon,\tau} \geq 1 - \varepsilon_0\,\, \hbox{ or }\,\, \underline{U}^{\varepsilon,\tau} \leq \varepsilon_0.
	\end{aligned}
\end{equation}
Indeed, since $\phi(x \cdot e - ct + \tau,x)$ is an entire solution of \eqref{equation}, we have
$$\partial_t \underline{U}^{\varepsilon,\tau}= \nabla \cdot (A(x) \nabla \underline{U}^{\varepsilon,\tau})
+ B(x) \cdot \nabla \underline{U}^{\varepsilon,\tau} +f(x,\phi(x \cdot e - ct + \tau,x)) \,\,\hbox{ in }\,\, \R\times\R^N. $$
For any $(t,x)\in\R\times\R^N$ such that $\underline{U}^{\varepsilon,\tau}(t,x) \leq \varepsilon_0$, there holds
$ 0<\phi(x \cdot e - ct + \tau,x) = \underline{U}^{\varepsilon,\tau}(t,x) + \varepsilon
\leq 2\varepsilon_0$, 
whence by \eqref{strong stability}, \eqref{extension} and \eqref{choose-varepsilon}, we have
\begin{equation*}
	f(x,\phi(x \cdot e - ct + \tau,x)) - f(x,\underline{U}^{\varepsilon,\tau}) \leq   -\gamma_0 \varepsilon \leq -\varepsilon^2.
\end{equation*}
This immediately yields the desired inequality. 
Similarly, for any $(t,x)\in\R\times\R^N$ such that $\underline{U}^{\varepsilon,\tau}(t,x) \geq 1 - \varepsilon_0$,
we have $1>\phi(x \cdot e - ct + \tau,x) = \underline{U}^{\varepsilon,\tau}(t,x) + \varepsilon \geq 1-\varepsilon_0$,
and hence, a similar computation yields the same inequality. 
Therefore, claim  \eqref{under-u-sub} is proved.

Next, we show that there exists a sufficiently large $\tau_+>0$ (depending on $\varepsilon$) such that 
\begin{equation}\label{claim-U-standing}
	\underline{U}^{\varepsilon,\tau}(0,\cdot) < \tilde{U}(\cdot)\,\, \hbox{ in }\,\, \R^N \,\,\hbox{ for all }\,\, \tau \geq \tau_+.
\end{equation}
In fact, the uniform limits of both profiles give a sufficiently large $N_0=N_0(\varepsilon)$ such that 
\begin{equation}\label{standing-sliding}
	\begin{cases}
		\min\left\{\phi(x \cdot e,x), \, \tilde{\phi}(x \cdot e,x)\right\} \geq 1-\varepsilon & \hbox{ for }\,\, x\cdot e \leq -N_0,\vspace{5pt}\\
		\max\left\{\phi(x \cdot e,x), \, \tilde{\phi}(x \cdot e,x) \right\} \leq \varepsilon & \hbox{ for }\,\, x\cdot e \geq N_0.
	\end{cases}
\end{equation}
For any $\tau\geq \tau_+:= 2N_0$, if $x \cdot e \geq -N_0$, then
$\underline{U}^{\varepsilon,\tau}(0,x) = \phi(x \cdot e + \tau,x)-  \varepsilon  \leq 0< \tilde{\phi}(x \cdot e,x)= \tilde{U}(x)$; and if $x \cdot e \leq -N_0$, then
$\underline{U}^{\varepsilon,\tau}(0,x) < 1-\varepsilon \leq \tilde{U}(x)$. This provides \eqref{claim-U-standing}.

{\bf Step 1:} We show that $c=0$.

Assume by contradiction that $c\neq 0$. Then either $c>0$ or $c<0$. Without loss of generality, we assume that $c>0$, as we will sketch below that the case $c<0$ can be handled similarly.

We complete the proof by comparing the function $\underline{U}^{\varepsilon,\tau}(t,x)$ defined in \eqref{define-underU} with the standing pulsating wave $\tilde{U}(x)$. 
Since $c > 0$,  $\phi$ is decreasing in its first variable by Lemma \ref{mono-lem}, and hence, 
the function $\underline{U}^{\varepsilon,\tau}(t,x)$ is increasing in $t\in\R$. Then, \eqref{claim-U-standing} gives that 
$$	\underline{U}^{\varepsilon,\tau_{\varepsilon}}(t,\cdot) < \tilde{U}(\cdot)\,\, \hbox{ in }\,\, \R^N \,\,\hbox{ for all }\,\, t \leq 0,$$
where $\tau_{\varepsilon}=\tau_+$. Moreover, using $c>0$ again, we find a large $\tilde{t}> 0$ such that $\underline{U}^{\varepsilon,\tau_{\varepsilon}}(\tilde{t},\tilde{x}) > \tilde{U}(\tilde{x})$ for some $\tilde{x} \in \R^N$. 
Now, define
\begin{equation*}
t_\varepsilon= \sup \left\{t \in\R : \underline{U}^{\varepsilon,\tau_{\varepsilon}}(t,x)\leq \tilde{U}(x) \hbox{ for all } x\in\R^N\right\}.
\end{equation*}
It follows from the above inequalities that $0\leq t_{\varepsilon}<\tilde{t}$, and 
\begin{equation}\label{pretouch}
	\sup_{x \in \R^N}\left(\underline{U}^{\varepsilon,\tau_{\varepsilon}}(t,x)- \tilde{U}(x)\right) \leq 0\,\,\hbox{ for all }\,\, t \leq t_\varepsilon.
\end{equation}
By the definition of $t_{\varepsilon}$, there exists a sequence $\{x_n\}_{n\in\N} \subset \R^N$ (depending on $\varepsilon$) such that
\begin{equation}\label{touch position}
\lim_{n \to +\infty}\left(\underline{U}^{\varepsilon,\tau_{\varepsilon}}(t_\varepsilon,x_n) - \tilde{U}(x_n)\right) = 0.
\end{equation}

For each $n\in\N$, write $x_n=z_n+y_n$, where $z_n\in \Z^N$ and $y_n\in [0,1]^N$. 
By the periodicity of the coefficients, and standard parabolic and elliptic estimates, up to extraction of a subsequence, we obtain that, as $n\to+\infty$,
$$\underline{U}^{\varepsilon,\tau_{\varepsilon}}(\cdot, z_n + \cdot)  \to V^{\varepsilon,\tau_{\varepsilon}}(\cdot,\cdot) \,\, \hbox{ in }\,\, C^{1,2}_{loc}(\R\times\R^N),\quad \hbox{and}\quad \tilde{U}(\cdot + z_n) \to  V(\cdot) \,\, \hbox{ in }\,\, C^{2}_{loc}(\R^N).   $$
 It is clear that $V$ is a stationary solution of \eqref{equation}, while by \eqref{under-u-sub}, we have
\begin{equation*}
\partial_t V^{\varepsilon,\tau_{\varepsilon}} \leq \nabla \cdot (A(x) \nabla V^{\varepsilon,\tau_{\varepsilon}}) 
		+ B(x) \cdot \nabla V^{\varepsilon,\tau_{\varepsilon}} + f(x,V^{\varepsilon,\tau_{\varepsilon}}) - \varepsilon^2
\end{equation*}
whenever $V^{\varepsilon,\tau_{\varepsilon}} \geq 1 - \varepsilon_0$ or  $V^{\varepsilon,\tau_{\varepsilon}} \leq \varepsilon_0$. In other words, $V^{\varepsilon,\tau_{\varepsilon}}(t,x)$ is a strict subsolution of \eqref{equation} in the set $\{(t,x): V^{\varepsilon,\tau_{\varepsilon}}(t,x) \geq 1 - \varepsilon_0 \hbox{ or } V^{\varepsilon,\tau_{\varepsilon}}(t,x) \leq \varepsilon_0  \}$. Moreover, denoting by $y_{\infty}$ the limit of $y_n$ (after extracting a subsequence if necessary), we obtain from \eqref{pretouch} and \eqref{touch position} that
\begin{equation}\label{touch-V}
\sup_{(t,x) \in (-\infty, t_\varepsilon] \times \R^N}\left(V^{\varepsilon,\tau_{\varepsilon}}(t,x) - V(x)\right)
= V^{\varepsilon,\tau_{\varepsilon}}(t_\varepsilon,y_{\infty})-V(y_{\infty}) = 0.
\end{equation}
Furthermore, we necessarily have 
$\varepsilon_0 \leq V^{\varepsilon,\tau_{\varepsilon}}(t_\varepsilon,y_{\infty}) \leq 1 - \varepsilon_0$. 
Otherwise, the point $(t_\varepsilon,y_{\infty})$ would lie in the region where $V^{\varepsilon,\tau_{\varepsilon}}$ is a strict subsolution, and then the strong parabolic comparison principle would contradict the equality in \eqref{touch-V}.
Consequently, for $n$ large enough,
\begin{equation*}
\varepsilon_0/2 \leq \underline{U}^{\varepsilon,\tau_{\varepsilon}}(t_\varepsilon,x_n) \leq 1 - \varepsilon_0/2.
\end{equation*}

Summarizing the above, we conclude that, for any $\varepsilon\in (0,\varepsilon_0/4)$,
there exist $t_\varepsilon\geq 0$, $\tau_\varepsilon\in \R$ and $x_\varepsilon\in \R^N$ such that 
	\begin{equation}\label{quasitouch}
		\underline{U}^{\varepsilon,\tau_\varepsilon}(t_\varepsilon,x_\varepsilon) - \tilde{U}(x_\varepsilon) \geq -\varepsilon, \quad 
		\sup_{(t,x)\in (-\infty,t_\varepsilon] \times \R^N} \left(\underline{U}^{\varepsilon,\tau_{\varepsilon}}(t,x) - \tilde{U}(x)\right) \leq 0,
	\end{equation}
and 
	\begin{equation}\label{quasitouch position}
		\varepsilon_0/2 
		\leq \underline{U}^{\varepsilon,\tau_\varepsilon}(t_\varepsilon,x_\varepsilon) 
		\leq 1 - \varepsilon_0/2.
	\end{equation}
	Observe that owing to \eqref{quasitouch position} and the definition of $\underline{U}^{\varepsilon,\tau}$, 
	the quantity $x_\varepsilon\cdot e - c t_\varepsilon+ \tau_\varepsilon$ 
	remains bounded as $\varepsilon\to 0^+$.
	For each $\varepsilon$, write $x_{\varepsilon}=z_{\varepsilon}+y_{\varepsilon}$, where $z_{\varepsilon}\in \Z^N$ and $y_{\varepsilon}\in [0,1]^N$. Then, the quantity  $z_\varepsilon\cdot e - c t_\varepsilon+ \tau_\varepsilon$ is also bounded. Hence, after extracting a sequence if necessary, we obtain that, as $\varepsilon\to 0^+$,  $z_\varepsilon\cdot e - c t_\varepsilon+ \tau_\varepsilon\to \xi_*$ and $y_{\varepsilon} \to y_*$  for some $\xi_*\in\R$ and $y_*\in [0,1]^N$, and that   
	\begin{equation*}
\underline{U}^{\varepsilon,\tau_\varepsilon}(t_\varepsilon+ t,z_\varepsilon+ x) \to \phi(x \cdot e - ct + \xi_*,x), \quad \hbox{and}\quad \tilde{U}(x + z_\varepsilon) \to \tilde{V}(x)	
\end{equation*}
locally uniformly in $t\in\R$ and $x\in\R^N$, where $\tilde{V}$ is a stationary solution of \eqref{equation}.
Furthermore, passing to the limit as $\varepsilon\to 0^+$ in \eqref{quasitouch}, we get
	\begin{equation*}
		\phi(y_* \cdot e + \xi_*,y_*) \geq \tilde{V}(y_*),\quad\hbox{and}\quad 
		\phi(x \cdot e - ct + \xi_*,x) \leq \tilde{V}(x) \,\,\hbox{ in }\,\, (-\infty,0]\times\R^N. 
	\end{equation*}
	Since both $\phi(x \cdot e - ct + \xi_*,x)$ and $\tilde{V}(x)$ are entire solutions of \eqref{equation},
	the parabolic strong maximum principle implies that $\tilde{V}(x) \equiv \phi(x \cdot e - ct + \xi_*,x)$ in $\R\times \R^N$. This is impossible, because letting $t\to\pm \infty$ gives, locally uniformly in $x\in\R^N$, $\phi(+\infty,x)=0$ and $\phi(-\infty,x)=1$, whereas $\tilde{V}$ is independent of $t$. 
	Therefore, the case $c>0$ cannot occur.  The case $c < 0$ can be excluded by a similar sliding argument, starting from an analogous supersolution of the form $\phi(x \cdot e - ct + \tau,x)+  \varepsilon$.  As a consequence, we prove that $c = 0$.

	{\bf Step 2:} Assuming that $\tilde{\phi}\in C(\R\times\T^N)$, we show that  $\tilde{U}(x)$ is identically equal to $\phi(x \cdot e + \tau_0,x)$ for some $\tau_0 \in \R$.  
	
	The proof follows from a similar argument to that in Step 1, with some necessary modifications.
	For any $\varepsilon \in (0, \varepsilon_0/4)$ and $\tau\in\R$, define $\underline{U}^{\varepsilon,\tau}$ as in \eqref{define-underU}. Since we have already proved that $c=0$, the function $\underline{U}^{\varepsilon,\tau}$ is now independent of $t$. Moreover, it still satisfies \eqref{under-u-sub} and \eqref{claim-U-standing}. 
On the other hand, applying \eqref{standing-sliding} once again, we can choose some $\tau_-\leq -2N_0$
such that $\underline{U}^{\varepsilon,\tau_-}(\tilde{x}) > \tilde{U}(\tilde{x})$ for some $\tilde{x}\in\R^N$.
As a consequence, defining 
$$\tau_{\varepsilon}=\inf \left\{\tau\in\R : \underline{U}^{\varepsilon,s}(x)\leq \tilde{U}(x) \hbox{ for all } x\in\R^N, s\geq \tau \right\}, $$ 
we have $\tau_-<\tau_{\varepsilon} \leq \tau_+ $, and 
     \begin{equation}\label{touch-stationary}
     	\sup_{x\in\R^N}(\underline{U}^{\varepsilon,\tau_{\varepsilon}}(x) - \tilde{U}(x)) = 0.
     \end{equation}
    Furthermore, by a similar comparison argument to that in Step 1, we find some $x_\varepsilon\in \R^N$ 
    satisfying
	\begin{equation}\label{quasitouch-stationary}
	- \varepsilon\leq  \underline{U}^{\varepsilon,\tau_{\varepsilon}}(x_\varepsilon) - \tilde{U}(x_\varepsilon) \leq 0,
\end{equation}
and
\begin{equation*}
		\varepsilon_0/2 \leq \underline{U}^{\varepsilon,\tau_{\varepsilon}}(x_\varepsilon) 
	\leq 1 -\varepsilon_0/2.
	\end{equation*} 
These estimates also give
$$ \frac{\varepsilon_0}{2}\le\widetilde U(x_\varepsilon)
\leq 1-\frac{\varepsilon_0}{4}. $$
The uniform limits of $\tilde{U}$ and $\phi$ therefore imply that the quantities $x_\varepsilon\cdot e$ and $x_\varepsilon\cdot e + \tau_{\varepsilon}$ are bounded as  $\varepsilon\to 0^+$. For each $\varepsilon \in (0, \varepsilon_0/4)$, let $z_\varepsilon\in \Z^N$ be a point such that $y_{\varepsilon}:=x_\varepsilon- z_\varepsilon \in [0,1]^N$.
	Up to extraction of a sequence, we may assume that 
	$$y_\varepsilon\to y_0 \in [0,1]^N,\quad  z_\varepsilon\cdot e + \tau_{\varepsilon} \to \xi_0 \in\R\quad\hbox{and} \quad z_{\varepsilon}\cdot e \to \tilde{\xi}_0\in\R \quad\hbox{as}\,\, \varepsilon\to 0^+.$$ 
	By the periodicity,  $\tilde{U}(x+z_{\varepsilon})=\tilde{\phi}(x\cdot e+z_{\varepsilon}\cdot e, x)$ is a stationary solution of \eqref{equation}. Moreover, by condition (H2), $\phi\in C^2(\R\times \T^N)$ is a classical solution of \eqref{front equation} with $c=0$. Since translations in the first variable preserve this equation, a direct verification shows that
	$\phi(x\cdot e+\tau_{\varepsilon}+z_{\varepsilon}\cdot e, x)$ is also a stationary solution of \eqref{equation}.
	Therefore, by standard elliptic estimates, after extracting a further subsequence if necessary, the following convergences hold locally uniformly in $\R^N$ as $\varepsilon\to 0^+$:
	\begin{equation*}
		\underline{U}^{\varepsilon,\tau_{\varepsilon}}(x + z_\varepsilon) \to \phi(x \cdot e + \xi_0,x),\quad\hbox{and}\quad  \tilde{U}(x + z_\varepsilon) \to \tilde{\phi}(x \cdot e + \tilde{\xi}_0,x),
	\end{equation*} 
	where the second limit uses the assumption that $\tilde{\phi}\in C(\R\times\T^N)$. 
Passing to the limit as $\varepsilon\to 0^+$ in \eqref{touch-stationary} and \eqref{quasitouch-stationary}, we obtain
	\begin{equation*}
		\phi(y_0 \cdot e + \xi_0,y_0) = \tilde{\phi}(y_0 \cdot e + \tilde{\xi}_0,y_0),
		\quad\hbox{and}\quad
		\phi(x \cdot e + \xi_0,x) \leq \tilde{\phi}(x \cdot e + \tilde{\xi}_0,x) \,\,\hbox{ for } \,\, x \in \R^N.
	\end{equation*}
Since both $\phi(x \cdot e + \xi_0,x)$ and $\tilde{\phi}(x \cdot e + \tilde{\xi}_0,x)$ are stationary solutions of \eqref{equation}, the strong maximum principle yields $\phi(x \cdot e + \xi_0,x) = \tilde{\phi}(x \cdot e + \tilde{\xi}_0, x)$ for  $x\in \R^N$. 
To recover the original slice, fix $\hat{x}\in\R^N$ and
evaluate this identity at $x=\hat{x}-z_\varepsilon$, along
the integer sequence selected above. By the periodicity in
the second variable, we have $\phi(\hat{x}\cdot e-z_\varepsilon\cdot e+\xi_0,\hat{x})
=\tilde\phi(\hat{x}\cdot e-z_\varepsilon\cdot e +\tilde\xi_0,\hat{x})$.
Since $z_\varepsilon\cdot e\to\tilde\xi_0$,
the continuity of both profiles yields
$$ \phi(\hat{x}\cdot e+\xi_0-\tilde\xi_0,\hat{x})
=\tilde\phi(\hat{x}\cdot e,\hat{x})=\tilde U(\hat{x}).$$
Thus the conclusion follows with $\tau_0=\xi_0-\tilde\xi_0$. This completes the proof of Lemma \ref{zerouniqueness}.
\end{proof}

Combining Lemmas \ref{nonzerouniqueness} and \ref{zerouniqueness}, we immediately obtain the desired result in Theorem \ref{theo-unique}.


\section*{Acknowledgements} 
This work has received funding from NSFC (12471197), Guangdong Basic and Applied Basic Research Foundation (2023B1515020034) and Science and Technology Projects in Guangzhou (SL2024A04J00172).

\bigskip 
\noindent
{\bf Declaration on the Use of AI.} The human authors conceived and led the research, formulated the mathematical problem, developed the main results, and independently verified all arguments. AI assisted in refining the notation and improving the language and presentation. The final manuscript was approved by the authors, who take full responsibility for its content.


\end{document}